\pdfoutput=1

\documentclass[10pt]{amsart}
\usepackage[T1]{fontenc}
\usepackage{txfonts}
\usepackage{eucal}

\usepackage{geometry}
\usepackage{fancyhdr}
\usepackage[parfill]{parskip}
\usepackage{lscape}
\usepackage{rotating}

\usepackage{amsmath, amsthm}
\usepackage{amssymb, amsfonts, amscd}
\usepackage{mathrsfs}
\usepackage{xfrac}

\numberwithin{equation}{section}

\usepackage{array}
\usepackage{tabularx}
\usepackage{multirow}
\usepackage{multicol}
\usepackage{booktabs}
\usepackage{makecell}

\usepackage{xcolor}
\usepackage{float}
\usepackage{caption}
\usepackage{subfig}
\usepackage{mdframed}
\usepackage{mwe}

\usepackage{graphicx}
\usepackage{tikz}
\usepackage{pgf}
\usepackage{tikz-cd}
\usepackage{tikzscale}
\usepackage[all]{xy}

\usetikzlibrary{arrows.meta, positioning}
\usetikzlibrary{
    backgrounds,
    shapes,
    shapes.geometric,
    arrows,
    decorations.markings,
    intersections,
    patterns,
    calc
}

\usepackage{enumitem}
\usepackage{xspace}
\usepackage{comment}
\usepackage{blindtext}
\usepackage[ruled, vlined]{algorithm2e}

\definecolor{smoked}{RGB}{216, 212, 204}
\definecolor{mauve}{RGB}{200, 55, 171}
\definecolor{apricot}{RGB}{250, 144, 4}
\definecolor{sky}{RGB}{66, 169, 244}
\definecolor{plum}{RGB}{76, 0, 102}
\definecolor{forest}{RGB}{90, 145, 120}
\definecolor{sand}{RGB}{180, 160, 120}
\definecolor{sabbia}{RGB}{207, 185, 151}
\definecolor{aliceblue}{rgb}{0.9, 0.95, 1.0}

\usepackage{hyperref}
\hypersetup{
    colorlinks=true, 
    linktoc=all,     
    linkcolor=blue,  
    citecolor=blue,
    filecolor=blue,
    urlcolor=blue
}

\usepackage{cleveref}

\newcommand\C{{\mathbb C}}
\newcommand\DD{{\mathbb D}}
\newcommand\HH{{\mathbb H}}
\newcommand\PP{{\mathbb P}}

\newcommand\R{{\mathbb R}}
\newcommand\sph{{\mathbb S}}
\newcommand\Z{{\mathbb Z}}
\newcommand\RP{{\mathbb{RP}}}

\newcommand\rp[1]{{\R\PP^{#1}}}

\newcommand\pc[1]{{\C\PP^{#1}}}
\newcommand\sphn[1]{{\sph^{#1}}}

\newcommand{\sor}[1]{{\mathrm{SO}(#1,\R)}}

\newcommand{\pslr}[1]{{\mathrm{PSL}(#1,\R)}}
\newcommand{\slr}[1]{{\mathrm{SL}(#1,\R)}}

\newcommand{\pglr}[1]{{\mathrm{PGL}(#1,\R)}}

\newcommand{\pslc}[1]{{\mathrm{PSL}(#1,\C)}}
\newcommand{\psu}{{\mathrm{PSU}(1,1)}}
\newcommand{\pso}{{\mathrm{PSO}(2,1)}}
\newcommand{\affr}{{\mathrm{Aff}(\,\R^2\,)}}

\DeclareMathOperator{\hol}{hol}

\newcommand{\pto}[1]{\textnormal{\textbf{#1}}}

\theoremstyle{plain}                   
\newtheorem{thm}{Theorem}[section]
\newtheorem{thma}{Theorem}

\newtheorem{lem}[thm]{Lemma}
\newtheorem{prop}[thm]{Proposition}
\newtheorem{cor}[thm]{Corollary}

\theoremstyle{definition}
\newtheorem{defn}[thm]{Definition}

\newtheorem{rmk}[thm]{Remark}

\theoremstyle{remark}
\newtheorem{ex}[]{Example}

\newtheorem{prob}[thm]{Problem}
\newtheorem{claim}[thm]{Claim}  

\newenvironment{proofclaim}
  {\proof}
  {\endproof}

\tikzstyle{rb} = [rectangle, rounded corners, minimum width=3cm, minimum height=1cm, text width=3cm, text centered, draw=black, fill=blue!30]
\tikzstyle{sb} = [rectangle, minimum width=2cm, minimum height=1cm, text width=3cm, text centered, draw=black, fill=violet!30]

\tikzstyle{co} = [circle, radius=2cm, text width=3cm, text centered, draw=white, fill=orange!50]
\tikzstyle{cr} = [circle, radius=2cm, text width=3cm, text centered, draw=white, fill=red!40]
\tikzstyle{cg} = [circle, radius=2cm, text width=3cm, text centered, draw=white, fill=LimeGreen!40]
\tikzstyle{cb} = [circle, radius=2cm, text width=3cm, text centered, draw=white, fill=blue!30]
\tikzstyle{cv} = [circle, radius=2cm, text width=3cm, text centered, draw=white, fill=violet!30]
\tikzstyle{cy} = [circle, radius=2cm, text width=3cm, text centered, draw=white, fill=yellow!50]
\tikzstyle{cgg} = [circle, radius=2cm, text width=3cm, text centered, draw=white, fill=green!40]
\tikzstyle{cc} = [circle, radius=2cm, text width=3cm, text centered, draw=white, fill=cyan!30]

\title[Branched real projective structures on surfaces\\ and geometrisation of representations]{Branched real projective structures on surfaces\\ and geometrisation of representations}

\author{Gianluca Faraco}
\address[Gianluca Faraco]{School of Mathematics, Monash University, Clayton, VIC 3800, Australia}
\email{gianluca.faraco@monash.edu}

\author{Nicholas Rungi}
\address[Nicholas Rungi]{Department of Mathematics, University of Turin, Italy}
\email{nicholas.rungi@unito.it}

\date{\today}

\begin{document}

\subjclass[2020]{}%
\dedicatory{}

\begin{abstract}
    We introduce and study branched real projective structures on compact surfaces. Our main result shows that every representation of the fundamental group of a compact surface into \(\pslr3\) arises as the holonomy representation of a branched real projective structure, regardless of whether the surface is orientable. We also consider the realisation problem with prescribed branching data, relating the branching degree to the second Stiefel--Whitney class of the representation. The results obtained in this direction suggest several interesting avenues for future research. 
\end{abstract}

\maketitle
\tableofcontents

\section{Introduction}\label{sec:intro}

In the present paper, we introduce and investigate \textit{branched real projective structures} on surfaces. As we shall see in more detail in \S\ref{sec:rpstructures}, branched real projective structures are geometric structures in the sense of Ehresmann--Thurston \cite{goldman2022geometric, thurston2022geometry}, \textit{i.e.}, geometries on surfaces locally modelled on the real projective plane \(\rp2\) along with its group of projectivities \(\pslr3\). Every such a structure naturally determines a homomorphism, say \(\rho \colon \pi_1(S) \longrightarrow \pslr3\), known as the \textit{holonomy representation}, that encapsulates the geometric data of the underlying structure. Although it is an easy matter to determine such a representation, the reverse direction is generally much more challenging. It is therefore natural to ask to what extent a representation encodes such geometric data, specifically determining the conditions under which a representation can be realised as the holonomy of some branched real projective structure on \(S\). The first main result of the present paper directly addresses this existence problem.

\begin{thma}\label{thma:geometrisation}
    Let \(S\) be a closed orientable surface. Every representation \(\rho\colon\pi_1(\,S\,)\longrightarrow\pslr3\) arises as the holonomy of some branched real projective structure on \(S\).
\end{thma}

In the following sections, we shall provide several motivations that led us to study these geometric structures on surfaces. In the first place, branched real projective structures can be seen as the real counterpart of complex branched projective structures, largely studied in last decades from different perspectives, \textit{e.g.}, in the pioneering paper by Gallo–Kapovich–Marden \cite{GalloKapovichMarden2000Monodromy}, which served as a primary source of inspiration for our work, and subsequent related literature. On the other hand, beyond their intrinsic geometric interest, branched real projective structures naturally appear in the study of certain Higgs bundles on Riemann surfaces. In fact, the present paper is the first in a series aiming to study the interplay between these two perspectives, highlighting the rich connections between geometric structures and holomorphic vector bundles on Riemann surfaces. In a forthcoming paper \cite{faracorungi}, we shall address the geometrisation problem by using Higgs bundles theory and by providing a partial refinement of our main Theorem \ref{thma:geometrisation}. More specifically, we shall focus on the realisation of branched projective structures with prescribed data, see \S\ref{ssec:strata} 


\subsection{Geometric foundations and motivation}\label{ssec: foundations and motivations}

Under Klein's Erlangen programme, projective geometry plays a unifying role: classical Euclidean, spherical, and hyperbolic geometries can all be viewed as sub-geometries of real projective geometry. However, projective structures are far more general; for instance, real affine structures fall under this umbrella while generally failing to be Euclidean, spherical, or hyperbolic. Therefore, a pivotal motivation of the present work is also to provide a unifying picture and reference framework for these structures. The existence of a specific geometric structure on a surface is highly sensitive to the underlying topology, which imposes obstructions. For instance, the sphere admits neither a Euclidean nor a hyperbolic structure. To circumvent such topological obstructions, one can allow the structure to develop branch points, leading to the notion of \textit{branched} geometric structures. As we see in detail in \S\ref{sec:rpstructures}, given a model geometry, say \((G,\mathbb X)\), a branched geometry on a surface \(S\) is defined as a local \((G,\mathbb X)\)-structure on the complement of a discrete set of branch points \(\mathcal{B} \subset S\), which extends smoothly across \(\mathcal{B}\) in terms of local branched coverings---a classic illustration of this phenomenon is provided by translation surfaces of genus \(g \ge 2\), which bear branched flat structures, see \S\ref{sssec: quality}. Therefore, branched real projective structures are defined as branched \((\pslr3,\rp2)\)-structures. As is the case for their complex counterparts, a branched real projective structure on \(S\) is uniquely determined by a pair consisting of a developing map and a holonomy representation, where the developing map is constructed by continuation of local charts. Unlike the unbranched setting, where a developing map is simply defined by using the analytic continuation property, in the branched framework a slightly more technical topological argument is required, which we present in detail in \S\ref{ssec:devhol_pair} for the reader's convenience. A similar approach has recently been developed for branched structures on \(3\)-manifolds, see \cite{ballas2021gluing, diaf2025domination}.

\smallskip

Similarly to other class of geometric structures on a surface, branched real projective structures can be classified up to projective diffeomorphisms isotopic to the identity, thereby yielding a moduli space \(\mathcal{RP}^2(\,S\,)\) of marked branched projective structures. By virtue of the fact that a branched real projective structure is uniquely specified by its developing–holonomy pair, assigning to each structure its corresponding holonomy representation leads to a well-defined map that boils down to the so-called \textit{holonomy map} defined as 
\begin{equation}\label{eq:holonomy-map intro}
    \hol \colon \mathcal{RP}^2(\,S\,) \longrightarrow \mathfrak X_3(\,S\,) = \mathrm{Hom}^+\Big(\,\pi_1(\,S\,), \pslr3\,\Big)\,\Big/\,\pslr3,
\end{equation}

see also \eqref{eq:identification} and compare with \eqref{eq:holonomy-map} in \S\ref{sssec:holomap}. The study presented in this paper focuses on understanding the holonomy map from several perspectives. In the following subsections, we outline our main results concerning the geometrisation problem and the fine structure of the holonomy map, \textit{e.g.}, local injectivity, openness and restrictions to strata, which will then be thoroughly investigated in \S~\ref{sec:geometrisation} and \S\ref{sec:stiefel whitney geometrisable representations}.

\smallskip

\subsection{Geometrisation problem}\label{ssec:geometrisation problem} In our framework, a representation \(\rho \colon \pi_1(\,S\,) \longrightarrow \pslr3\) is said to be \textit{geometrisable} if it arises as the holonomy of some, possibly branched, \(\big(\,\pslr3, \rp2\,\big)\)-structure on \(S\). The geometrisation problem asks under what conditions a representation is a holonomy representation: \textit{i.e.}, which representations appear in the image of the holonomy map defined in \eqref{eq:holonomy-map intro}. The first main result of the present paper, see Theorem \ref{thma:geometrisation}, establishes that the holonomy map is surjective, meaning that no obstruction arises in the geometrisation process.

\subsubsection{Geometrisation process}\label{sssec: geo process} There are various ways to geometrise a given representation, say \(\rho\), when possible. A classical approach consists of splitting the topological surface into pieces of lower complexity, typically one-punctured tori and pairs of pants along a fixed pants decomposition. The representation \(\rho\) thus restricts to each of these sub-surfaces, which are geometrised individually. Finally, these geometrised pieces are assembled to recover the original topological surface, now equipped with the desired geometric structure. 

Although this approach has been systematically used over the years, \textit{e.g.}, in classical Teichmüller theory \cite{benedetti2012lectures} and \cite{imayoshi1992introduction} and in the work of Gallo--Kapovich--Marden \cite{GalloKapovichMarden2000Monodromy}, we propose a different strategy here in the same spirit of \cite{FaracoGupta2025}. By recalling that a real projective structure, either branched or unbranched, on a surface \(S\) is specified by a developing-holonomy pair, our approach consists in constructing a smooth \(\rho\)-equivariant map \(\textnormal{dev}\colon\widetilde{S}\longrightarrow\rp2\). By equivariance property, it suffices to construct a fundamental domain, say \(\mathcal{D}\subset\widetilde S\), for the action of the fundamental group and a smooth function \(f\colon\mathcal D\longrightarrow\rp2\). For this latter purpose, given a generic representation, we decompose the surface into \(g\) handles and a \(g\)-punctured sphere, which we call the \textit{core}. We then geometrise these sub-surfaces by first determining fundamental membranes in the sense of Hejhal, see \cite{hejhal1975monodromy}, for each of them, thereby realising the fundamental domains for each piece before gluing them together. A fundamental membrane can be understood intuitively as a fundamental domain in the model space where self-intersections are permitted. The associated fundamental domain is obtained by \textit{unfolding} this membrane, resulting in a genuine topological domain, \textit{i.e.}, a polygon in our setting. The inverse folding map maps this topological domain onto its immersed image in the model space. By gluing the topological fundamental domains together along their boundary identification rules, the corresponding folding maps glue continuously, as do their images. This procedure yields the desired global map \(f\). The resulting domain \(\mathcal D\) is a \(4g\)-gon, where \(g\) denotes the genus of the initial surface, equipped with a natural locally injective map to \(\rp2\), \textit{cf.} Figure \ref{fig:finding a fundamental domain}. By tiling a topological plane with copies of this domain and extending the map equivariantly, we obtain the desired developing map and, consequently, the geometric structure with the prescribed holonomy.

\begin{figure}[htbp]
  \centering
  \begin{tikzpicture}[scale=1.375]
    
    \begin{scope}[rotate=210]
    \draw[white, opacity=0.25] (0,0) circle (2.25cm);
    \shade[ball color = white, opacity=.20, shading angle=0] (090:2)--(210:2)--(330:2)--(090:2);
    \draw[orange] (090:2)--(210:2)--(330:2)--(090:2);

    \foreach \p in {(090:2), (210:2), (330:2), (090:2)} {
    \fill[black] \p circle (1pt);
    }

    \node at (000:0) {core};

    \begin{scope}
        \shade[ball color = sky, opacity=.20, shading angle=0] (090:2)--(330:2)--(345:4)--(030:5)--(075:4)--(090:2);
        \draw[sky] (090:2)--(330:2)--(345:4)--(030:5)--(075:4)--(090:2);

        \foreach \p in {(090:2), (330:2), (345:4), (030:5), (075:4)} {
        \fill[black] \p circle (1pt);
        }
    \end{scope}

    \begin{scope}[rotate=120]
        \begin{scope}
        \shade[ball color = yellow, opacity=.20, shading angle=0] (090:2)--(330:2)--(345:4)--(030:5)--(075:4)--(090:2);
        \draw[yellow] (090:2)--(330:2)--(345:4)--(030:5)--(075:4)--(090:2);

        \foreach \p in {(090:2), (330:2), (345:4), (030:5), (075:4)} {
        \fill[black] \p circle (1pt);
        }
        \end{scope}
    \end{scope}

    \begin{scope}[rotate=-120]
        \begin{scope}
        \shade[ball color = red, opacity=.20, shading angle=0] (090:2)--(330:2)--(345:4)--(030:5)--(075:4)--(090:2);
        \draw[red] (090:2)--(330:2)--(345:4)--(030:5)--(075:4)--(090:2);

        \foreach \p in {(090:2), (330:2), (345:4), (030:5), (075:4)} {
        \fill[black] \p circle (1pt);
        }
        \end{scope}
    \end{scope}
    \end{scope}
    
    \end{tikzpicture}
    \caption{Overview of the geometrisation process in the proof of Theorem \ref{thma:geometrisation}. The fundamental domain \(\mathcal{D}\) and its decomposition into regions corresponding to handles (blue, red, and yellow) and the core surface (white). A detailed version of this construction appears in Figure \ref{fig:gluing} (see \S\ref{sssec:assembly generic}).}
    \label{fig:finding a fundamental domain}
\end{figure}
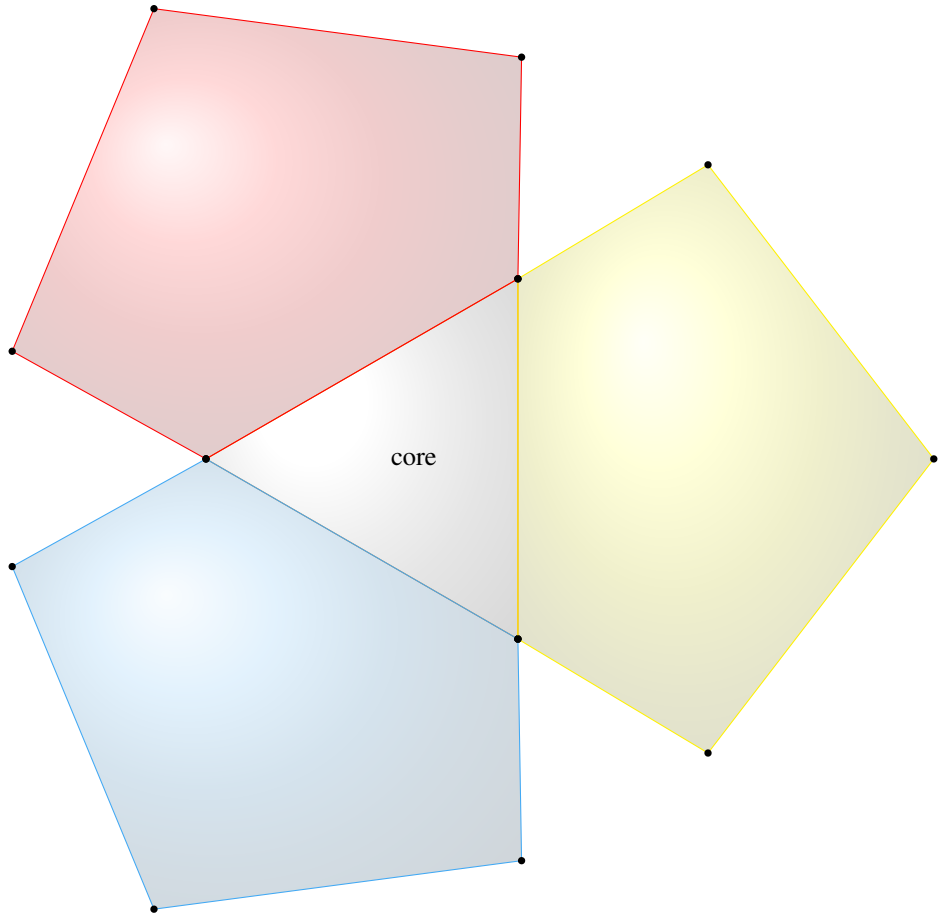

\smallskip

\subsubsection{Qualitative geometrisation and sub-geometries}\label{sssec: quality} As alluded to above, real projective geometry encompasses several subgeometries, including the classical spherical, Euclidean, and hyperbolic geometries; see Table~\ref{tab:flowdiagram} and \S\ref{ssec:examples} for a more detailed account of examples known in the literature.

\smallskip

Viewing any subgeometry \((G, \mathbb{X})\) of \((\pslr3, \rp2)\) as a geometry in its own right, the geometrisation problem asks whether a representation into \(G\) arises as the holonomy of some \((G, \mathbb{X})\)-structure on \(S\). Whenever this occurs, since \((G, \mathbb{X})\) is a subgeometry of real projective geometry, the resulting structure is automatically a, possibly branched, real projective structure. On the other hand, a representation into \(G\) may fail to be realisable as the holonomy of a \((G, \mathbb{X})\)-structure on \(S\), while nonetheless arising as the holonomy of a real projective structure, by Theorem~\ref{thma:geometrisation}. A typical example of this phenomenon is provided by elementary representations in \(\pslr2\). This phenomenon shows that the geometrisation process we shall propose in \S\ref{sec:geometrisation} (summarised in \S\ref{sssec: geo process}) to prove Theorem~\ref{thma:geometrisation} is optimal from a quantitative point of view, as it is general enough to apply to every non-trivial representation. As previously observed, the trivial representation is treated separately. However, since \textit{all} projective structures we shall realise here are branched by design, it remains to address the following refined 

\begin{prob}\label{prob:geometrisation orientable}
    Let \((G,\mathbb X)\) be a subgeometry of \(\big(\,\pslr3, \rp2\,\big)\). Provide necessary and sufficient conditions for a representation \(\rho\colon\pi_1(\,S\,)\longrightarrow G\) to arise as the holonomy of some, possibly branched, \((G, \mathbb{X})\)-structure on \(S\).
\end{prob}

Although the case of translation structures has been extensively studied and fully characterised, see \cite{OH}, the cases of spherical and hyperbolic structures remain open. 

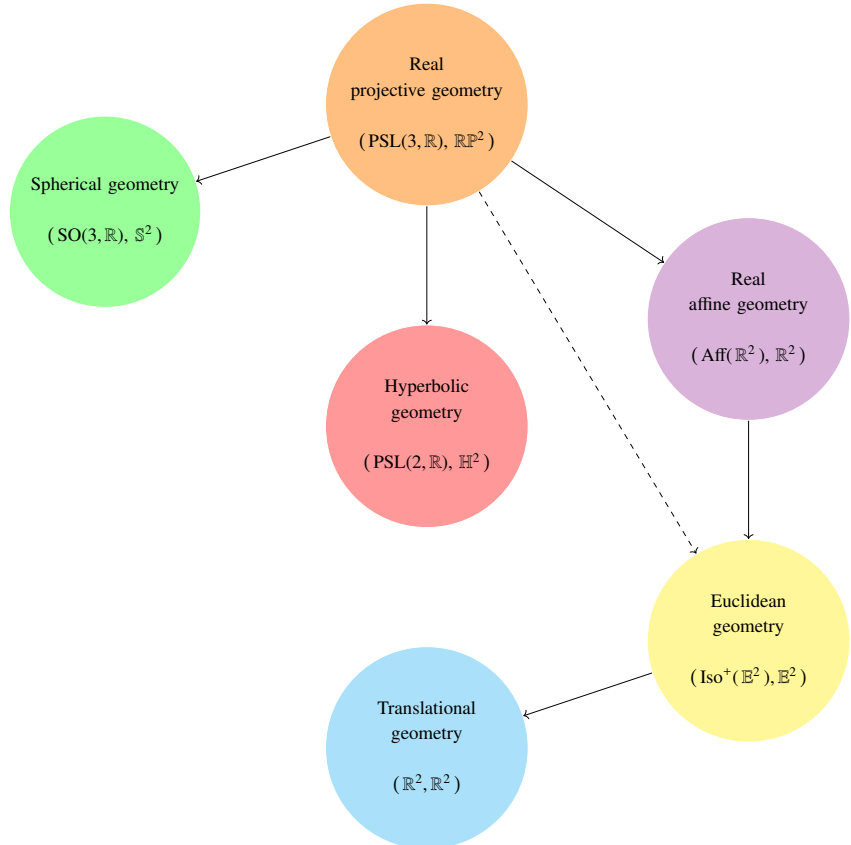
\begin{table}[ht]
    \centering
    \caption{\textit{Geometry graph.} This graph shows all geometries arising from the real projective one in dimension two and their connections. In fact, to Klein's Erlangen program, real projective geometry plays a central role and other geometries can be seen as special cases of projective one.}  
    \resizebox{0.75\textwidth}{!}
    {
    \begin{tikzpicture}[scale=0.125, node distance = 1cm]
    
    \node (realproj)  [co] {Real\\ projective geometry\\ \text{ }\\ \(\big(\,\pslr3,\,\rp2\,\big)\)};

    \node (sphgeo)  [cgg, below of=realproj, xshift=-6cm, yshift=-1cm] {Spherical geometry \\ \text{ }\\ \(\big(\,\sor3,\,\sph^2\,\big)\)};
    \draw[->] (realproj) -- node[anchor=south] {} (sphgeo);
    
    \node (hypgeo)  [cr, below of=realproj, yshift=-5cm] {Hyperbolic \\ geometry \\ \text{ }\\ \(\big(\,\pslr2,\,\HH^2\,\big)\)};
    \draw[->] (realproj) -- (hypgeo);

    \node (eugeo)  [cy, below of=realproj, xshift=6cm, yshift=-9cm ] {Euclidean \\ geometry \\ \text{ }\\ \(\big(\,\textnormal{Iso}^+(\,\mathbb E^2\,),\mathbb E^2\,\big)\)};
    \draw[->, dashed] (realproj) -- node[anchor=south] {} (eugeo);

    \node (transgeo)  [cc, below of=hypgeo, xshift=0cm, yshift=-5cm] {Translational \\ geometry \\ \text{ }\\ \(\big(\,\R^2,\R^2\,\big)\)};
    \draw[->] (eugeo) -- node[anchor=south] {} (transgeo);

    \node (raffgeo)  [cv, below of=realproj, xshift=6cm, yshift=-3cm] {Real \\ affine geometry \\ \text{ }\\ \(\big(\,\text{Aff}(\,\R^2\,),\,\R^2\,\big)\)};
    \draw[->] (raffgeo) -- node[anchor=south] {} (eugeo);
    \draw[->] (realproj) -- node[anchor=south] {} (raffgeo);

    \end{tikzpicture}
    }
    \bigskip
    \label{tab:flowdiagram}
\end{table}

\smallskip

\subsubsection{Non-orientable surfaces}\label{sssec: non orientable} Real, possibly branched, projective structures can be also defined on non-orientable surfaces. In Appendix \ref{app:non orientable surfaces}, we extend the geometrisation process described in \S\ref{sssec: geo process} to the non-orientable framework, thus leading to the following extension of Theorem \ref{thma:geometrisation}.

\begin{thma}\label{thma: geo non orientable}
    Let \(S\) be a closed non-orientable surface. Every representation \(\rho\colon\pi_1(S)\to\pslr3\) arises as the holonomy of some branched projective structure on \(S\).
\end{thma}

Klein surfaces, also known as dianalytic surfaces, extend the classical notion of Riemann surfaces to non-orientable topological contexts. Originally introduced by Felix Klein \cite{Klein1882} in his study of real algebraic curves and non-orientable surfaces, their formal foundations as non-orientable analogues of Riemann surfaces were systematically developed much later by Alling and Greenleaf \cite{AllingGreenleaf1969, AllingGreenleaf1971}. See also \cite{Schaffhauser2016} for a modern introductory account. 

In this framework, a dianalytic structure on a manifold is given by an atlas of charts whose transition maps are either complex analytic maps or complex conjugates of complex analytic maps. Equivalently, every dianalytic manifold is obtained as the quotient of a complex analytic manifold \(\big(\widetilde{M},J\big)\) (possibly non-connected) by a fixed-point-free anti-holomorphic involution \(\tau \colon \big(\widetilde{M},J\big) \longrightarrow \big(\widetilde{M},J\big)\), which satisfies the identity 
\begin{equation}
    \mathrm{d}\tau \circ J = -J \circ \mathrm{d}\tau.
\end{equation}

In light of this, we say that a subgeometry \((G,\mathbb X)\) of \((\pslr3,\rp2)\) is \emph{dianalytic}, or \textit{Kleinian}, if the model space \(\mathbb X\) admits a dianalytic structure that is invariant under the action of \(G\). For instance, while the full projective geometry \((\pslr3, \rp2)\) fails to be dianalytic, the spherical subgeometry \((\mathrm{PO}(3,\mathbb{R}), \rp2)\) is naturally dianalytic. Another prominent example is given by the hyperbolic subgeometry \((\mathrm{PGL}(2,\mathbb{R}), \mathbb{H}^2)\). Consequently, whenever a subgeometry is dianalytic, the study of branched geometric structures naturally bridges with the theory of dianalytic surfaces. In this regard, Problem \ref{prob:geometrisation orientable} extends to the non-orientable setting as follows.

\begin{prob}\label{prob:geometrisation non orientable}
    Let \((G,\mathbb X)\) be a Kleinian subgeometry of \(\big(\,\pslr3, \rp2\,\big)\). Provide necessary and sufficient conditions for a representation \(\rho\colon\pi_1(\,S\,)\longrightarrow G\) to arise as the holonomy of some, possibly branched, \((G, \mathbb{X})\) structure on \(S\), where \(S\) is either orientable or non-orientable.
\end{prob}

\smallskip

\subsection{Strata of branched projective structures}\label{ssec:strata} The moduli space of branched real projective structures on a surface \(S\) admits a natural stratification into strata \(\mathcal{RP}^2(\,\mu\,)\), where \(\mu = (m_1, \dots, m_k)\) is a tuple of positive integers termed the \textit{signature}. A stratum comprises equivalence classes of branched projective structures possessing precisely \(k\) branch points of orders \(m_1, \dots, m_k\), where a branch point has order \(m\) if the local degree of any local chart around it is \(m+1\) (so that regular points correspond to order zero). The holonomy map defined in \eqref{eq:holonomy-map intro} restricts to a holonomy map on each stratum. As we shall demonstrate, this restriction is an open map, thereby yielding an analogue of the Ehresmann--Thurston principle within our framework. A refined formulation of the geometrisation problem then asks whether a given representation can be realised as the holonomy of a branched projective structure belonging to a specified stratum. In contrast to the unstratified setting, obstruction phenomena genuinely arise when attempting to realise a representation as the holonomy of a branched projective structure in a prescribed stratum. Analogously to the case of complex branched projective structures, a fundamental topological obstruction is provided by the Stiefel--Whitney class of the representation --- recall that, for a representation \(\rho \colon \pi_1(\,S\,) \longrightarrow \pslr3\), its second Stiefel--Whitney number \(\textnormal{w}_2(\,\rho\,)\in \mathbb{Z}_2\) measures the topological obstruction to lifting \(\rho\) to the universal covering group \(\widetilde{\mathrm{SL}}(3, \mathbb{R})\). For a given signature \(\mu\), we define its degree, denoted by \(\deg(\,\mu\,)\), as the sum of its components. The aforementioned obstruction relates the second Stiefel--Whitney class of a representation to the degree of the signature specifying the stratum in which the representation is to be realised. Specifically, in \S\ref{sec:stiefel whitney geometrisable representations}, establish the following criterion, see Corollary \ref{cor: char realisation}.

\begin{thma}\label{thma: char realisation}
    If a representation \(\rho \colon \pi_1(\,S\,) \longrightarrow \pslr3\) arises as the holonomy of a branched projective structure in the stratum \(\mathcal{RP}^2(\,\mu\,)\), then its second Stiefel--Whitney class satisfies the compatibility relation
        \begin{equation}
            \textnormal{w}_2(\,\rho\,)\equiv \deg(\,\mu\,) \pmod 2.
        \end{equation}    
\end{thma}

As we may observe, the theorem provides merely a necessary condition for the realisation of a given representation within the strata. In what follows, we shall completely characterise Hitchin representations by demonstrating that every such representation is realised in the stratum \(\mathcal{RP}^2(\,\mu\,)\) for every signature \(\mu\) of even degree. We note that this does not contradict the earlier result of Choi--Goldman \cite{ChoiGoldman1993Convex}. Indeed, their findings combined with ours, see \S\ref{ssec:hitchinbranched}, assert that every Hitchin representation arises as the holonomy of infinitely many branched projective structures, exactly one of which is unbranched and properly convex. There is no complete characterisation for non-Hitchin representations, and the following question remains to be addressed.

\begin{prob}\label{prob:geometrisation strata}
    Determine necessary and sufficient conditions for a non-Hitchin representation \(\rho\) to be realised in a stratum \(\mathcal{RP}^2(\,\mu\,)\), provided that the compatibility condition \(\textnormal{w}_2(\,\rho\,)\equiv \deg(\,\mu\,) \pmod 2\) holds.
\end{prob}

In recent years, questions in the spirit of Problem \ref{prob:geometrisation strata} have attracted considerable attention across various broad geometric contexts, particularly within the frameworks of complex projective and flat structures. Driven by this active line of research, we are naturally led to address this problem in the setting of complex branched projective structures and flat structures on surfaces arising from holomorphic and meromorphic differentials on Riemann surfaces, see \cite{BJJP},  \cite{CFG}, \cite{CF}, \cite{fils} and \cite{fils_complex_proj}. As already alluded above, in the subsequent paper we shall provide a partial answer to Problem \ref{prob:geometrisation strata} by using Higgs bundles and thus a more differential rather than topological approach, see \cite{faracorungi}. This problem is closely linked to the notion of minimality for the degree of the branch divisor. Naturally, one asks whether the total branching degree can be minimised for a given representation. For Hitchin representations, the minimal degree is strictly zero, as established by the classical theorem of Choi--Goldman. For non-Hitchin representations, however, the minimal degree required for geometrisation remains an open question:

\begin{prob}\label{prob: min degree}
    Given a non-Hitchin representation \(\rho\), what is the minimal degree \(d(\,\rho\,) \in \Z^+\) for which \(\rho\) is realised in a stratum \(\mathcal{RP}^2(\,\mu\,)\) with \(\deg(\,\mu\,) = d(\,\rho\,)\)?
\end{prob}

\subsubsection{Uniform lower bound for the minimal degree}\label{sssec:zariski dense} There exist representations for which the minimal branching degree is entirely determined by topological invariants. A primary example is provided by the trivial representation. Indeed, any branched projective structure with trivial holonomy necessarily arises as a branched cover over the model space \(\rp2\). Consequently, the degree of the signature \(\mu\) is constrained by the Riemann--Hurwitz formula applied to the underlying topological branched cover. In the case of the trivial representation, the geometrisation problem thus reduces directly to the realisability of branched coverings over \(\rp2\) with a prescribed branching data (or branch profile), a fundamental question in topological surface theory that, in general, remains an open problem. See \cite{allan1984} and \cite{Petronio2020} for a more recent account. The trivial representation is, admittedly, rather peculiar. More broadly, representations in \(\pslr3\) may be divided into those that are Zariski-dense and those that are not. The former class is vastly larger, as Zariski-dense representations form a subset of full measure in the representation space, whereas the latter represent rather special phenomena. By drawing a comparison with the theory of complex branched projective structures the following is intimately related with Problem \ref{prob: min degree}.

\begin{prob}
    Determine whether a uniform lower bound exists for the minimal degree required to geometrise Zariski-dense representations. If such a bound exists, how is it related to the Stiefel--Whitney number of the representation?
\end{prob}

\subsubsection{Elementary representations}\label{sssec:elementary} In Appendix~\ref{asec:elementary}, we shall introduce a special class of representations in \(\pslr3\), namely \textit{elementary representations}. We shall say that a representation \(\rho\) is elementary if its image preserves a conic in \(\rp2\) or in \(\big(\,\rp2\,\big)^*\). Otherwise, \(\rho\) is called \textit{non}-elementary, see Definition \ref{defn:elementarygroupfirstkind} 
and related Table \ref{tab:elementaryfirst}. 
The reader may notice that this naturally extends the standard notion of elementary representations in \(\pslr2\), by recalling that quadrics in \(\rp1\) are either empty, a single point, or a pair of distinct points. In particular, we shall provide two complete characterisations of these representations from different perspectives. Algebraically, we characterise non-elementary representations as those with Zariski-dense image in \(\pslr3\), see Theorem~\ref{prop:nonelementary-ZD}. We then establish a more dynamical characterisation in terms of invariant conics in the space of all conics of \(\rp2\), see Theorem~\ref{thm: nonelementary-SI}. We suspect that our definition is not specific to dimensions \(2\) and \(3\), but rather extends naturally to any dimension \(n \ge 2\). This leads to the following

\begin{prob}
    For every \(n\ge2\), define a group \(\Gamma<\pslr{n+1}\) to be elementary if it preserves a quadric in \(\rp{n}\) or in \(\big(\,\rp{n}\,\big)^*\) and characterise these groups. Is a group non-elementary if and only if it Zariski-dense in \(\pslr{n+1}\)?
\end{prob}

\subsection{Plan of the paper} The rest of the paper is organised as follows: In \S\ref{sec:rpstructures} we introduce branched projective structures, their moduli spaces, their strata, and the holonomy map. Before proceeding with the geometrisation of representations, we shall provide a detailed account of examples of real projective structures in literature, see \ref{ssec:examples}. In \S\ref{sec:geometrisation} we dive into the geometrisation process by establishing Theorem \ref{thma:geometrisation}. In \S\ref{sec:stiefel whitney geometrisable representations} we begin characterising which representations can be realised as the holonomy of certain branched projective structures in a given stratum, providing a partial answer to Problem \ref{prob:geometrisation strata}. Along the way we shall provide several pictures that we believe they ease the reading and understanding of our geometrisation process. Finally, we shall conclude with two appendices about related topics. In Appendix \ref{app:non orientable surfaces} we extend Theorem \ref{thma:geometrisation} to non orientable surfaces thus establishing Theorem \ref{thma: geo non orientable}. In Appendix \ref{asec:elementary}, we shall provide the definition of elementary representations to which we alluded above.

\smallskip

\textit{Notation:} Long arguments will be organised into a sequence of claims, each accompanied by its own proof. To facilitate reading, we use the diamond symbol \(\diamondsuit\) to indicate the end of the proof of an intermediate claim, and keep the usual square symbol \(\Box\) to mark the conclusion of the proof of the main statement.

\smallskip

\subsection{Acknowledgements} This project was initiated following a visit by G.F. to the University of Turin in February 2026 to deliver a seminar, during which he met the second author, N.R. The first named author is grateful to Andrea Seppi for the invitation and to the Department of Mathematics of the University of Turin for the warm hospitality during his stay. G.F. is also grateful to Robert Tang for his interest in our work, and for sharing his recent developments and details of his approach, which have revealed interesting connections with our work. N.R. is funded by the European Union (ERC, GENERATE, 101124349). Views and opinions expressed are however those of the author(s) only and do not necessarily reflect those of the European Union or the European Research Council Executive Agency. Neither the European Union nor the granting authority can be held responsible for them.

\smallskip

\section{Branched real projective structures on surfaces}\label{sec:rpstructures} 

Let \(\rp2\) denote the real projective plane and \(\pslr3\) denote the group of projective transformations. For the readers' convenience, we recall that in our setting \(\pslr3=\pglr3\) and that the latter is naturally isomorphic to \(\slr3\). In \S\ref{sec:rpstructures} and \S\ref{sec:geometrisation} we shall use the former notation, whereas from \S\ref{sec:stiefel whitney geometrisable representations} onwards we shall use the latter one for our convenience. Real projective structures, or simply \(\rp2\)-structures, on a closed surface is a geometric structure locally modelled on the real projective plane \(\rp2\), with transition maps in the projective group \(\pslr3\). In the present section, we aim to extend this classical notion by introducing the concept of branched real projective structures. These structures generalise the standard definition by allowing the presence of a discrete set of points where the local charts are no longer embeddings, but instead behave like branched coverings of topological discs.

\subsection{A portmanteau of definitions}\label{ssec:basicdefinitions} Let \(S\) be an oriented and connected surface. A \textit{real projective structure} on \(S\) is the datum of a maximal atlas of \(\rp2\)-valued charts that differ by some M\"obius transformation on their overlaps, see \cite[\S2]{goldman1990}. In order to extend the definition to branched structures, it is first necessary to extend the notion of a chart. We begin with the following

\begin{defn}\label{def:branchechart}
    Let \(S\) be a surface. A \textit{branched real projective chart} is a pair, say \((U,\, \varphi)\), where \(U \subseteq S\) is an open subset and \(\varphi \colon U \to \varphi(\,U\,) \subseteq \rp2\) is a finite degree branched covering map. More specifically, for every \(p \in S\), there exist an open neighbourhood \(U \subset S\) of \(p\), an open neighbourhood \(V \subset \rp2\) of \(\varphi(\,p\,)\), and homeomorphisms \(f\colon U \to \DD\) and \(g\colon V \to \DD\), where \(\DD\) is the unit disc in \(\R^2\), such that \(f=g\circ\varphi\) and the condition \(f(\,p\,) = g\circ\varphi(\,p\,) = 0\) holds. By using the polar coordinates \((r,\theta)\) on \(\DD\), the map \(\varphi\) is isotopic to the map
    \begin{equation}
        (r,\theta)\longmapsto \Big(\,r^{m+1},\,(m+1)\theta\,\Big)
    \end{equation}
    for some integer \(m \geq 0\). Moreover, \(p\) is said to be a \textit{branch point of order} \(m\). We shall say that two branched charts, say \(\varphi_1\) and \(\varphi_2\), at \(p\) are equivalent of there exists a M\"obius transformation \(g\in\pslr{3}\) such that the identity \(\varphi_2=g\circ\varphi_1\) holds.
\end{defn}

\begin{defn}\label{def:branchedprojestructure}
    A \textit{branched \(\rp2\)-structure} on \(S\), denoted by \(\sigma\), is the datum of a maximal atlas, say \(\mathcal U\), of (possibly branched) real projective charts. More precisely, a \(\rp2\)-structure consists of the following data: 
    \begin{itemize}
        \item[1.] a discrete collection of points, say \(\mathcal B \subset S\),
        \item[2.] a maximal atlas of \(\rp2\)-valued charts on \(S \setminus \mathcal B\), and
        \item[3.] a maximal collection branched charts centred at points of \(\mathcal B\), 
    \end{itemize}
    such that the transition maps on \(S\) are restrictions of elements of \(\pslr3\), \textit{i.e.}, for each pair of (possibly branched) charts, say \((U_1,\varphi_1),\,(U_2,\varphi_2)\in\mathcal U\), there is \(g\in\pslr3\) such that the identity \(\varphi_2=g\circ\varphi_1\) holds.
\end{defn}

\noindent We observe that the order of a branch point is well-defined and it does not depend on the choice of any branched chart around the branch point. In fact, for every branched point \(p\in(S,\sigma)\), any pair of branched charts around \(p\) are equivalent because they differ by a M\"obius transformation. Since the latter is a homeomorphism of \(\rp2\), it readily follows that the order of the branched point does not depend on the choice of the local branched chart.

\begin{rmk}\label{rmk:branchingchain}
    Every branched real projective structure on a closed surface \(S\) yields a \(0\)-chain in \(C_o(S,\Z)\), the group of \(0\)-chains on \(S\), that records the orders of the branch points. More specifically, let \(\sigma\) be such a structure on \(S\), and let \(\mathcal B=\{\,p_1,\dots,p_n\,\}\) be its set of branch points with respective orders \(m_1,\dots,m_n\). Then, the \textit{branching chain} is defined as 
    \begin{equation} 
        D=\sum_{p_i\in \mathcal B}\, m_i\,p_i.
    \end{equation}
    Recall that the degree map is the group homomorphism \(\deg\colon C_o(S,\Z) \longrightarrow \mathbb{Z}\) that assigns to each such a formal sum its total sum of coefficients. The degree \(\deg(\,D\,)\) of \(D\) is defined as \textit{total branching order}. 
\end{rmk}

\noindent We shall provide several examples of branched projective structures later or in \S\ref{ssec:examples}. Before doing so, we continue with further generalities about these structures. 

\smallskip

\textit{Convention and terminology.} From now on, we shall refer to branched real projective structures simply as branched projective structures, omitting the adjective \textit{real}. In this work, we do not consider their complex counterpart, \textit{i.e.}, complex projective structures, hence explicitly specifying their real nature is unnecessary. Should complex projective structures be mentioned for any reason, their real or complex nature will be explicitly stated to avoid confusion.

\smallskip

\subsubsection{Moduli spaces of branched projective structures}\label{sssec:moduli spaces} Having established the notion of branched projective structures, we are now ready to define their moduli spaces. In order to define the deformation space of branched \(\rp2\)-structures, let us introduce a natural notion of isomorphism for these structures.

\smallskip

\begin{defn}\label{def:projectivemaps}
Let \(\Omega\subset\rp2\) be an open set. A map \(\phi\colon\Omega\longrightarrow \rp2\) is said to be \textit{locally projective} if for each connected components \(C\subset\Omega\), there is a projectivity \(g\colon\rp2\longrightarrow \rp2\) such that \(\phi|_{C}=g|_{C}\). Let \(\sigma_1\) and \(\sigma_2\) be two branched \(\rp2\)-structures on \(S\) and let \(f \colon (S,\sigma_1) \longrightarrow (S,\sigma_2)\) be a smooth map. We shall say that \(f\) is \textit{projective} if it locally projective in local charts. We say it is a \textit{projective isomorphism} if it is also a homeomorphism. Finally, we shall denote by \(\textnormal{Proj}(\,\sigma\,)\) the group of projective isomorphisms of \(\sigma\) on \(S\).
\end{defn}


\begin{defn}
A \textit{marked branched \(\rp2\)-structure} on \(S\) is a triple \((S, \sigma, f)\), where \(\sigma\) is a surface endowed with a branched \(\rp2\)-structure and \(f \colon S \longrightarrow (S,\sigma)\) is an orientation-preserving diffeomorphism. Two marked structures, say \((S, \sigma_1, f_1)\) and \((S, \sigma_2, f_2)\), are declared to be equivalent if \(f_2 \circ f_1^{-1} \colon (S,\sigma_1) \longrightarrow (S, \sigma_2)\) is isotopic to a projective isomorphism \(h \colon (S, \sigma_1)\longrightarrow (S, \sigma_2)\). In the following, we shall denote by \(\mathcal{RP}^2(\,S\,)\) the set of marked branched \(\rp2\)-structures on \(S\) up to this equivalence relation.
\end{defn}

\noindent The space defined above acquires the structure of a moduli space when equipped with a suitable topology. We shall recall the details of this topology in \S\ref{sssec:defspace}. In what follows, we shall be primarily interested in certain subspaces of the moduli space just defined. More specifically, we introduce a stratification of the moduli space \(\mathcal{RP}^2(\,S\,)\) into \textit{strata}, each of which become subspaces once the ambient moduli space is equipped with its natural topology.

\begin{defn}\label{def:strata}
    The space of branched projective structures admits a natural stratification into strata indexed by \textit{signatures} \(\mu=(\,m_1,\dots,m_k\,)\) of positive integers where each one comprises all branched structures with exactly \(k\) branch points of orders \(m_1,\dots,m_k\) (repetitions are allowed). For an element \(\mu\in\mathbb Z^k_+\), its \textit{degree} is defined as \(\deg(\,\mu\,)=m_1+\cdots+m_k\), and we shall denote the corresponding stratum as \(\mathcal{RP}^2(\,\mu\,)\). Under this perspective, the space of \textit{unbranched} projective structures is a stratum inside the overall space \(\mathcal{RP}^2(\,S\,)\) that we shall denote as \(\mathcal{RP}^2_{\textnormal{o}}(\,S\,)\).
\end{defn}

\subsubsection{Geodesics, boundaries and corner points}\label{sssec:other elements} The real projective plane \(\rp2\) carries a natural Riemannian metric which, once the sphere \(\sph^2\) is equipped with its standard elliptic structure, makes the covering projection \(\pi_{\textnormal{std}}\colon\sph^2\longrightarrow \rp2\) a local isometry (in particular, locally projective). Every line in \(\rp2\) is a geodesic line with respect to this Riemannian structure. In particular, every geodesic line in \(\rp2\) is a circle, and hence any two distinct points on it determine two geodesic segments. A \textit{segment} in \(\rp2\) is any such portion of a line bounded by two distinct points. Based on these premises, we introduce the following concepts, which will be extensively used in what follows. For this purpose, let \(S\) be an orientable surface, possibly with non-empty boundary \(\partial S\). Let \((S,\sigma)\) be a, possibly branched, projective structure. 

\smallskip

\begin{enumerate}
    \item[1.]\textit{Geodesic segments.} according to Goldman, see \cite[\S2.2]{goldman1990}, we shall say that a segment in \((S,\sigma)\) is a geodesic segment if it is projectively equivalent to a segment in \(\rp2\) in local charts. We may notice that this definition extends the well-known notion of geodesic segment to the specific setting of hyperbolic structures. 
    \smallskip
    \item[2.] \textit{Geodesic boundary.} A real projective structure on a surface with non-empty boundary is said to have a geodesic boundary if its boundary is locally projectively equivalent to a geodesic segment in \(\rp2\). That is, for every point \(p\in\partial S\) and every chart \((U,\varphi)\) centred at \(p\), then \(\varphi(U\cap\partial S)\) is a geodesic segment in \(\rp2\).
    \smallskip
    \item[3.] \textit{Corner point.} A real projective structure \(\sigma\) on a surface \(S\) with non-empty boundary is said to have a corner point at \(p \in \partial S\) if for every chart \((U, \varphi)\) centred at \(p\), the image \(\varphi(U \cap \partial S)\) is a concatenation of two geodesic segments on two geodesic lines in \(\rp2\) intersecting at \(\varphi(p)\). In this case, we shall say that \((S,\sigma)\) has piecewise geodesic boundary.
\end{enumerate}

\noindent These foundational geometric features complete the basic definitions and will be used throughout the rest of our investigation.

\smallskip

\subsection{Developing-holonomy pair}\label{ssec:devhol_pair} From the classical theory of geometric structures, every unbranched structure is completely captured by a pair known as the developing-holonomy pair; see \cite{goldman2022geometric} for an excellent treatment in the most general framework. In particular, every unbranched real projective structure determines such a pair. Let \(\widetilde S\) be the universal cover of \(S\) and let \(\pi\colon\widetilde{S}\longrightarrow S\) be the covering projection. For any \(\rp2\)-structure on \(S\), there exists a unique \(\rp2\)-structure on \(\widetilde S\) that makes the covering projection a locally projective map as in Definition \ref{def:projectivemaps}. If \(\sigma\) is any real (unbranched) projective structure on \(S\), then its \textit{developing map} is defined as the unique locally projective map \(\textnormal{dev}_\sigma\colon\widetilde S\longrightarrow \rp2\) up to post-composition by an element of \(\pslr3\). Concretely, a developing map is constructed by applying the principle of analytic continuation to any chart of the structure's atlas, which is guaranteed by the fact that real projective structures are (real) analytic geometric structures. Since two charts defined on the same open set on \(S\) must differ by an element in \(\pslr3\), it readily follows that a developing map is uniquely determined up to post-composition by an element in \(\pslr3\). The developing map satisfies an equivariant property with respect to a representation \(\rho\colon\pi_1(\,S\,)\longrightarrow \pslr3\) called \textit{holonomy representation} of \(\sigma\). Clearly, if \(\textnormal{dev}_\sigma\) and \(g\cdot\textnormal{dev}_\sigma\) are two developing maps for \(\sigma\) and the former is equivariant with respect to a representation \(\rho\), then the latter is equivariant with respect to \(\textnormal{Ad}_g(\,\rho\,)\). We stress that the argument just presented applies to any surface, not necessarily closed nor of finite type. See \cite[\S2.5]{alessandrini2019} for more details.

\smallskip

\subsubsection{Realising a developing map}\label{sssec:realisingdevelopingmaps} Let us now extend these notions to the setting of branched projective structures. Although the existence of an analogous pair is known to hold also in the branched case, since this does not seem to be explicitly detailed in the literature, we provide full details here for the reader's convenience. For this purpose, let \(\sigma\) be a branched projective structure on \(S\) and let \(\widetilde\sigma\) be its lift to the universal cover \(\widetilde S\). Let \(\mathcal B\) be the set of branch points of \(\sigma\) and let \(\widetilde{\mathcal B }\) denote the lift of \(\mathcal B\) to the universal cover. We may notice that \(\widetilde{\mathcal B}\) is invariant under the action of \(\pi_1(\,S\,)\). We first consider the surface \(\Sigma=\widetilde S\setminus \widetilde{\mathcal B}\) and observe that the branched \(\rp2\)-structure on \(\widetilde S\) restricts to an unbranched \(\rp2\)-structure on \(\Sigma\). The existence and realisation of the developing map for branched structures is collected in the following claims. We begin with the following one concerning the holonomy representation of the branched structure on \(\Sigma\).

\begin{claim}\label{claim:holonomypuncturedsurface}
    Let \(\theta\colon\pi_1(\,\Sigma\,)\longrightarrow \pslr3\) be the holonomy representation of \(\widetilde\sigma\) restricted to \(\Sigma\). Then \(\theta\) is the trivial representation.
\end{claim}

\begin{proofclaim}
    Let \(p\in\widetilde{\mathcal B}\subset \widetilde S \) be any branch point and let \(\Delta\) be an open disc centred at \(p\). Up to shrinking it if necessary, we may assume that the closure of \(\Delta\) does not contain any branch point other than \(p\). Let \(\gamma\) be the boundary of \(\Delta\). By design, this is a loop around \(p\) that bounds a punctured disc and does not enclose any other puncture in \(\Sigma\). Moreover, since the former structure on \(\widetilde S\) is a branched \(\rp2\)-structure with a branch point at \(p\) of order \(m\ge1\), we may also assume that \(\Delta\) is properly contained in some branched chart \((U, \varphi)\), as in Definition \ref{def:branchechart}, that contains no branch points other than \(p\). We may notice that the image of \(\gamma\) via the map \(\varphi\) is a closed loop in \(\rp2\) that winds \(m+1\) times around \(\varphi(\,p\,)\). Let \(q\in\gamma\) be a base-point for \(\gamma\) in \(U\setminus \{\,p\,\}\). By compactness, we can cover \(\gamma\) by a finite, ordered chain of charts \(\mathcal V=\{\,V_o, \dots, V_n\,\}\) such that two sets overlap along \(\gamma\) if and only if their indices are consecutive modulo \(n+1\). According to our standing assumptions, the restriction of \(\varphi\) to any \(V_\alpha\in\mathcal V\) is a local homeomorphism onto its image and two overlapping open sets differ by an element of \(\pslr3\). Let \(V_o\) denote the open set covering \(q\). We may assume that this is the only open set of \(\mathcal V\) covering \(q\). If not, we shrink the open sets of \(\mathcal V\) a little so that the property holds. Then the following identity holds:
    \begin{equation}
        \varphi|_{V_o}=g_{01}\,g_{12}\,g_{23}\,\cdots\,g_{n0}\,\circ\,\varphi|_{V_o}.
    \end{equation}
    Since the holonomy \(\theta(\,\gamma\,)\) does not depend on the choice of the loop surrounding \(p\) nor on the choice of the open cover \(\mathcal V\), it follows that \(\theta(\,\gamma\,)=g_{01}\,g_{12}\,g_{23}\,\cdots\,g_{n0}\). Since local charts are locally projective maps and these are real analytic, it follows that \(g_{01}\,g_{12}\,g_{23}\,\cdots\,g_{n0}=\textbf{1}\). Hence \(\gamma\) has trivial holonomy. Finally, since the argument just developed does not depend on the choice of \(p\in\widetilde{\mathcal B}\), it follows that \(\theta\) is the trivial representation, \textit{i.e.}, the projective structure \(\widetilde\sigma\) restricted to \(\Sigma\) has trivial holonomy.
\end{proofclaim}

\begin{claim}\label{claim:develpoingmapconstruction}
    There exists a locally projective map \(f\colon \Sigma\longrightarrow\rp2\) unique up to post-composition by an element of \(\pslr3\). Moreover, such a map extends continuously to a well-defined map \(\textnormal{dev}_\sigma\colon \widetilde S\longrightarrow \rp2\). 
\end{claim}

\begin{proofclaim}
    Let \(\mathcal A_{\widetilde\sigma}\) be the maximal for \(\widetilde\sigma\) on \(\Sigma\) and let \((U,\varphi)\in\mathcal A_{\widetilde\sigma}\) be any chart (necessarily unbranched). Up to shrinking \(U\) a little if necessary, we may assume \(U\) is homeomorphic to a disc. We then apply the analytic continuation property to \((U, \varphi)\) to define a well-defined locally projective map \(f\colon \Sigma\longrightarrow\rp2\). Notice that, as a direct consequence of Claim \ref{claim:holonomypuncturedsurface}, non-trivial homotopy classes of loops on the surface do \textit{not} create obstructions to the unique analytic continuation of a projective chart. Therefore, there is a global projective map \(f\), which depends only on the choice of \((U,\varphi)\) and the principle of unique analytic continuation guarantees the uniqueness of \(f\) up to an element of \(\pslr3\). We finally wish to extend the projective map just defined to the whole surface \(\widetilde S\). The lifts of the branch charts of \(\sigma\) on \(S\) provide the desired extension, and since \(\Sigma\subset\widetilde S\) is dense, this is the unique possible extension. As a consequence, the resulting map \(\textnormal{dev}_\sigma\colon \widetilde S\longrightarrow \rp2\) is the desired developing map.
\end{proofclaim}

\noindent Even in the branched framework, the developing map \(\textnormal{dev}_\sigma\colon \widetilde S\longrightarrow \rp2\) satisfies an equivariant property with respect to a representation \(\rho\colon\pi_1(\,S\,)\longrightarrow \pslr3\). In principle, the map \(f\) just constructed, whose existence is guaranteed by Claim \ref{claim:develpoingmapconstruction}, satisfies an equivariance property with respect to a representation of the fundamental group of the surface covered by \(\Sigma\), namely a representation \(\rho'\colon\pi_1(\,S_{g,n}\,)\longrightarrow \pslr3\), where \(\pi_1(\,S_{g,n}\,)\) is the fundamental group of the punctured surface \(S_{g,n}\cong S\setminus\mathcal B\). Since the representation \(\theta\) is trivial by Claim \ref{claim:holonomypuncturedsurface}, the representation \(\rho'\) maps non-essential loops to the identity. Therefore, \(\rho'\) factors naturally to a representation of \(\rho\colon\pi_1(\,S\,)\longrightarrow \pslr3\).

\smallskip

\begin{rmk}
From a different perspective, Disarlo and Tang \cite{DT} have independently developed a related notion of developing map, where the latter is seen as an equivariant map from a \textit{Farey complex} associated to a punctured surface to \(\mathbb{RP}^2\) (private communication).
\end{rmk}

\smallskip

\subsubsection{Branched projective structures as developing maps}\label{sssec:BPSasmpas} In the present section we aim to provide the following characterisation.

\begin{prop}\label{prop:BPSasmaps}
    A branched projective structure \(\sigma\) on \(S\) is equivalent to the existence of a real analytic map \(\textnormal{dev}_\sigma\colon \widetilde S\longrightarrow \rp2\) which is locally injective except in a discrete subset \(\widetilde{\mathcal B}\) and which is equivariant with respect to a representation \(\rho\colon\pi_1(\,S\,)\longrightarrow \pslr3\).
\end{prop}

\begin{proof}
    In \S\ref{sssec:realisingdevelopingmaps}, we have seen that every branched projective structure determines a developing-holonomy pair. Therefore, it remains to show the opposite direction. In the first place, we notice that the subset \(\widetilde{\mathcal B}\) and its complement, say \(\Sigma\), are invariant under the action of \(\pi_1(\,S\,)\) as a consequence of the equivariance property of the mapping \(\textnormal{dev}_\sigma\). Let \(\mathcal B=\pi(\,\widetilde{\mathcal B}\,)\), where \(\pi\) denotes the covering projection. Our standing assumptions imply that the restriction of \(\textnormal{dev}_\sigma\) to any sufficiently small open set in \(\Sigma\) is a homeomorphism onto its image, and its precomposition with the local inverse of the covering projection yields a local chart on \(S\setminus \mathcal B\). It remains to show that \(\textnormal{dev}_\sigma\) restricts to a branched chart on sufficiently small open neighbourhoods around every point of \(\widetilde{\mathcal B}\). Let \(p\in\widetilde{\mathcal B}\) be any point and let \(U\) be an open set centred at \(p\). This is a critical point where the differential \(\mathrm{d}(\textnormal{dev}_\sigma)_p\) vanishes. Since \(\widetilde S\) and \(\rp2\) are real analytic surfaces and \(\textnormal{dev}_\sigma\) is a real analytic map, we can choose local real analytic coordinates centred at \(p\) and \(\textnormal{dev}_\sigma(p)\) such that \(\textnormal{dev}_\sigma\) restricts to a real analytic mapping \(\varphi \colon U \longrightarrow \R^2\subset\rp2\) with \(\varphi(\,p\,)=0\). Since the branch locus \(\widetilde{\mathcal{B}}\) is discrete, \(p\) is an isolated critical point of \(\varphi\). It follows from the theory of real analytic maps with isolated critical points that there exist local polar coordinates \((r, \theta)\) on \(U\) and corresponding polar coordinates on the codomain such that \(\varphi\) takes the local normal form:
    \begin{equation}
        (r, \theta) \longmapsto \left(\,r^{m+1}, (m+1)\theta\,\right),
    \end{equation}
    where \(m \ge 1\) is the vanishing order of the differential at \(p\). This local model corresponds precisely to a cyclic branched covering map of degree \(m+1\) over an open subset of \(\rp2\). Consequently, \(\textnormal{dev}_\sigma\) restricts to a branched projective chart around \(p\).
\end{proof}

\smallskip

\subsubsection{Deformation space of branched projective structures}\label{sssec:defspace} As a direct consequence of the characterisation provided in \S\ref{sssec:BPSasmpas}, the deformation space of branched projective structures can be identified with the space of equivalence classes of developing-holonomy pairs modulo the action of \(\pslr3\) and orientation-preserving diffeomorphisms. Consequently, the deformation space of branched projective structures can be framed from an alternative viewpoint as the orbit space of developing-holonomy pairs under the simultaneous action of \(\pslr3\) by post-composition and the mapping class group by pre-composition. More precisely, we quotient the space of such pairs by post-composition with elements of \(\pslr3\) and by pre-composition of the developing maps with orientation-preserving diffeomorphisms of \(S\) that are isotopic to the identity. 
\begin{equation}\label{eq:identification}
    \mathcal{RP}^2(\,S\,) = \left\{\,\,\begin{gathered} 
    \text{branched \(\rp2\)-structures}\\ 
    \text{on \(S\) as maximal}\\
    \text{atlases of charts}
\end{gathered}\,\, \right\} \longleftrightarrow \Big\{\,\big[\,g\cdot\textnormal{dev},\,\textnormal{Ad}_g(\,\rho\,)\,\big]\,\,\big\vert\,\,g\in\pslr3\,\Big\}_{\sim \textnormal{Diff}_o(\,S\,)}\,\,.
\end{equation}
\smallskip

\noindent By equipping the set of pairs \((\textnormal{dev}, \rho)\) with the compact-open topology, the deformation space naturally inherits the corresponding quotient topology.

\smallskip

\subsection{The representation space and holonomy map}\label{sec: representation space} We now aim to recall some generalities about the space of representation in \(\pslr3\) and the associated character variety. Let \(S\) be a closed oriented surface of genus \(g \ge 2\). The representation space \( \mathrm{Hom}\big(\pi_1(\,S\,),\,\pslr3\big) \) is defined as the space of representations of the fundamental group of \(S\) into \(\pslr3\), which admits a natural structure of a real analytic variety. The \textit{character variety} is defined as the quotient
\begin{equation}
    \mathfrak X_3(\,S\,) = \mathrm{Hom}^+\Big(\,\pi_1(\,S\,), \pslr3\Big)\Big/\pslr3.
\end{equation}
where \(\pslr3\) acts by conjugation on representations, and \(\mathrm{Hom}^+(\pi_1(\,S\,), \pslr3)\) denotes the subset consisting of completely reducible representations. 
In this way, the induced topology on the space is Hausdorff. In his seminal work \cite{Hitchin1992}, Hitchin investigated the connected components of character varieties associated with split real Lie groups. In the specific case of \(\pslr3\), a key outcome of his work is the following

\begin{thm}[Hitchin]\label{thm: connected components character variety}
    The character variety \(\mathfrak{X}_3(\,S\,)\) has three connected components: the one containing the class of the trivial representation, the one consisting of representations whose associated flat \(\rp2\)-bundles have non-zero second Stiefel-Whitney class, and the one consisting of representations connected to those arising from uniformization. Moreover, the third one is contained in the smooth locus of \(\mathfrak{X}_3(\,S\,)\) and it is diffeomorphic to \(\R^{16g-16}\).
\end{thm}

\noindent It is worth noting that the only topological invariant for a representation \(\rho \in \mathfrak X_3(\,S\,)\) is a characteristic class with values in \(\mathbb{Z}_2\), namely the second Stiefel-Whitney class \(\mathrm{w}_2(\,\rho\,)\). This invariant determines that two of the connected components have \(\mathrm{w}_2 = 0\), while only one has \(\mathrm{w}_2 \neq 0\). In the following, we denote by \(\mathcal{C}_1\) the unique component associated with representations \(\rho\) such that \(\mathrm{w}_2(\,\rho\,)\neq 1\). The other two components both satisfy the condition \(\mathrm{w}_2 = 0\). We shall denote by \(\mathcal{C}_0\) the components containing the trivial representation, and by \(\mathrm{Hit}_3(\,S\,)\), the component diffeomorphic to \(\mathbb{R}^{16g-16}\) which is nowadays called the \textit{Hitchin component}. We stress that, since they share the same vanishing invariant, the Stiefel-Whitney class does not allow one to distinguish the component \(\mathrm{Hit}_3(\,S\,)\) from \(\mathcal{C}_0\), see \cite[\S 10]{Hitchin1992} or \cite[Example 1.10]{collier2019sigma}.

\smallskip

\subsubsection{Hitchin component}\label{sssec:hitchincomponent} A key feature of the Hitchin component \(\mathrm{Hit}_3(\,S\,)\) is that it contains a copy of the Teichm\"uller space \(\mathcal{T}(\,S\,)\). Recall that the Teichm\"uller space \(\mathcal{T}(\,S\,)\) can be identified with the space of faithful and discrete representations \(\rho_o\colon\pi_1(\,S\,)\longrightarrow \mathrm{PSL}(2,\mathbb{R})\) called \textit{Fuchsian representations}. Every representation in \(\pslr2\), in particular Fuchsian once, can be systematically promoted to a representation into \(\pslr3\) by post-composing it with the unique irreducible representation \( \jmath\colon\mathrm{PSL}(2,\mathbb{R})\longrightarrow \pslr3\) (up to conjugation). The assignment \(\rho_o \mapsto \jmath \circ \rho_o\) defines a natural embedding of \(\mathcal{T}(\,S\,)\) as a closed sub-manifold inside \(\mathrm{Hit}_3(\,S\,)\). According to Labourie in \cite{Labourie2006Anosov}, we shall henceforth adopt the subsequent terminology: a representation in the Hitchin component is called \textit{\(3\)-Fuchsian representations}, if it is defined as the post-composition of Fuchsian representations into \(\pslr2\) with the irreducible representation.

\smallskip

\subsubsection{Holonomy map}\label{sssec:holomap} As previously discussed in \S\ref{ssec:devhol_pair}, every branched \(\rp2\)-structure determines an equivalence class of developing-holonomy pairs. The holonomy map is thus defined as the mapping that associates each branched structure with its corresponding holonomy representation, that is,
\begin{equation}\label{eq:holonomy-map}
    \hol \colon \mathcal{RP}^2(\,S\,) \longrightarrow \mathfrak X_3(\,S\,) = \mathrm{Hom}^+\Big(\,\pi_1(\,S\,), \pslr3\,\Big)\Big/\pslr3.
\end{equation}

\noindent It can be shown that this map is continuous with respect to the topologies on the domain and target. Under this perspective, one of the primary goals of the present paper is to determine the image of this map, \textit{i.e.}, to determine which representations arise as the holonomy of some branched \(\rp2\)-structure on a closed surface, see \S\ref{sec:geometrisation}. In what follows, we shall adopt this terminology.
\begin{defn}
    A representation \(\rho\colon\pi_1(\,S\,)\longrightarrow\pslr3\) is said to be \textit{geometrisable} whenever it lies in the image of the holonomy map, that is, if it arises as the holonomy of some, possibly branched, real projective structure.
\end{defn}

\noindent 
Our next goal is to show that this space is open by proving a version of the Ehresmann-Thurston principle. As a consequence of Theorem \ref{thma:geometrisation} (see also Theorem \ref{thm:geometrisation}) it readily follows that the holonomy map is automatically open. We now prove a strengthened version of this result, namely we show that the holonomy map is also open when restricted to strata. This provides further motivation to understand the realisation problem for branched structures with prescribed branching data. More precisely, we establish the following result whose proof is an adaption of \cite[Theorem 2.7]{Kapovich2020}.

\begin{prop}
    The space \(\hol\big(\,\mathcal{RP}^2(\,\mu\,)\,\big)\) is an open subset of \(\mathfrak{X}_3(\,S\,)\).
\end{prop}
\begin{proof}
    Let \((S,\sigma) \in \mathcal{RP}^2(\,\mu\,)\) be a branched projective structure with developing map \(\textnormal{dev}_{\sigma}\colon\widetilde{S} \longrightarrow \rp2\) and holonomy \(\rho\colon \pi_1(\,S\,) \longrightarrow\pslr3\). Let \(\rho_k\colon \pi_1(\,S\,) \longrightarrow\pslr3\) be a sequence of representations converging to \(\rho\) in \(\mathfrak{X}_3(\,S\,)\). In order to prove the desired result, we need to find a sequence of branched projective structures \((S,\sigma_k) \in \mathcal{RP}^2(\,\mu\,)\) with holonomy representations \(\rho_k\) such that \(\lim \sigma_k = \sigma\) for \(k\) large enough. For this purpose, choose a triangulation \(\mathcal T\) of \(S\) such that each edge is a segment (\textit{i.e.}, it is projective equivalent to a segment in a projective line in \(\rp2\)), with respect to \(\sigma\), and each simplex is contained in a valid projective coordinate neighbourhood. We may assume that the set of vertices of \(\mathcal T\) contains all the branch points, say \(\{p_1, \dots, p_n\}\) of \(\sigma\). Lift this triangulation to a \(\pi_1(\,S\,)\)-invariant triangulation \(\widetilde{\mathcal T}\) of \(\widetilde{S}\). Pick a finite collection, say \(\Delta_1, \dots, \Delta_m\) of 2-simplexes in \(\widetilde{\mathcal T}\), representing a fundamental domain, say \(\mathcal D=\bigcup_j \Delta_j\), for the action of \(\pi_1(\,S\,)\). Let \(\gamma_i \in \pi_1(\,S\,)\) be the elements of the deck-transformation group such that:
    \begin{equation}
        \gamma_i\left(\,\mathcal D\,\right)\, \cap \,\mathcal D \neq \emptyset.
    \end{equation}
    Let \(\mathcal C \subset \widetilde{S}\) be a compact subset whose interior contains both \(\mathcal D\) and its images under the \(\gamma_i\)'s. For each perturbed representation \(\rho_k\), we realise a continuous, \(\rho_k\)-equivariant mapping \(f_k\colon\mathcal C \longrightarrow \rp2\) such that:
    \begin{enumerate}
        \item[1.] \(f_k\) maps each 2-simplex homeomorphically to a projective 2-simplex in \(\rp2\), and 
        \item[2.] the sequence \(f_k\) converges to \(\textnormal{dev}_\sigma|_{\mathcal C}\) uniformly on compact subsets.
    \end{enumerate}
    We now extend each function \(f_k\) just defined to a global \(\rho_k\)-equivariant mapping \(\textnormal{dev}_{\sigma_k}\colon \widetilde{S} \longrightarrow\rp2\) via the representation \(\rho_k\). It remains to show that for sufficiently large \(k\), each mapping \(\textnormal{dev}_{\sigma_k}\) is a local homeomorphism away from the designated branch points, and behaves as a local ramified covering of degree \(m(\,v\,)+1\) at each vertex \(v \in \mathcal D\), where \(m(\,v\,)\ge1\) if and only if \(v\in\{\,p_1,\dots,p_n\,\}\). It suffices to check this condition for points within the fundamental domain \(\mathcal D\). We consider case by case.
    \begin{itemize}
        \item[1.] If \(v \in \mathrm{int}(\,\mathcal C\,)\) belongs to the interior of a 2-simplex, the claim follows immediately since each \(f_k\) is constructed to be a homeomorphism on each individual simplex.
        \smallskip
        \item[2.] If \(v\) belongs to the interior of a common geodesic arc \(\delta\) of two \(2\)-simplexes, say \(\Delta_i\) and \(\Delta_j\). Since the unperturbed developing map \(\textnormal{dev}_\sigma\) is a local homeomorphism there, \(\textnormal{dev}_\sigma(\,\Delta_i\,)\) and \(\textnormal{dev}_\sigma(\,\Delta_j\,)\) lie (locally) on opposite sides of the projective segment \(\textnormal{dev}_\sigma(\,\delta\,) \subset \rp2\). By uniform convergence, the same geometric configuration holds for \(\textnormal{dev}_{\sigma_k}\) when \(k\) is sufficiently large. Thus, \(\textnormal{dev}_\sigma\) does not fold along the arc \(\delta\) and remains a local homeomorphism at \(v\).
        \smallskip
        \item[3.] Lastly, if \(v\) is a vertex of the triangulation corresponding to a branch point \(p_i \in \mathcal D\), the local degree of the original developing map \(\textnormal{dev}_\sigma\) at \(v\) equals \(m(\,v\,) + 1\). Consider a small loop \(\gamma\) surrounding \(v\) within its star. Because the triangulation \(\mathcal{T}\) is fixed and \(\textnormal{dev}_{\sigma_k}\) is a local homeomorphism on all faces and internal edges (by cases (1) and (2)), the mapping degree of \(\textnormal{dev}_{\sigma_k}\) along \(\gamma\) is well-defined. By the homotopy invariance of the degree under uniform convergence, this degree must coincide with that of \(\textnormal{dev}_\sigma\) for sufficiently large \(k\). Since no folding can occur inside the star of \(v\), \(\textnormal{dev}_{\sigma_k}\) is forced to remain a branched covering at \(v\) of the exact same local degree \(m(\,v\,) + 1\).
    \end{itemize}
    The \(\rho_k\)-equivariance of the mappings \(\textnormal{dev}_{\sigma_k}\) implies that they converge to \(\textnormal{dev}_{\sigma}\) uniformly on compact subsets of \(\widetilde{S}\), yielding a well-defined sequence of branched structures \((S,\sigma_k) \in \mathcal{RP}^2(\,\mu\,)\).
\end{proof}

\noindent In the following sections, see \S\ref{sssec:holopropconvexstructures} and \S\ref{ssec:examples}, we shall recall several examples of branched projective structures and discuss their holonomy representations case by case. It will turn out that the holonomy map has been investigated from various perspectives depending on the specific geometric structures considered; with the present work, we aim to provide a unified picture.


\smallskip

\subsection{Convex domains and convex branched projective structures}\label{ssec:convexstructures} Let \(S\) be a closed orientable surface of genus \(g\ge2\). We now recall an important class of real projective structures, \textit{i.e.} convex projective structures. According to Goldman, see \cite[\S3]{goldman1990}, a domain \(\Omega\subset\rp2\) is \textit{convex} if there exists a projective line, say \(\ell\subset\rp2\), such that \(\Omega\,\cap\,\ell=\varnothing\) and \(\Omega\) is a convex subset of the affine plane \(\rp2\setminus\ell\). It is worth recalling that, according to this definition, the real projective plane is \textit{not} convex, while the affine plane \(\R^2\) is convex. The crucial property of convex structures is the existence of a well-defined notion of a geodesic. For any pair of points in \(\Omega\) there exists a unique projective line segment passing through them.

\smallskip

\noindent Convex structures are particularly special because they are uniformisable. More specifically, a \(\rp2\)-structure, say \(\sigma\), on a closed surface \(S\) is convex if and only if it arises as the quotient \((S,\sigma) \cong \Gamma \backslash \Omega\), where \(\Omega \subset \rp2\) is a convex domain and \(\Gamma \subset \pslr3\) is a subgroup that preserves \(\Omega\) and acts freely and properly discontinuously. In other words, they constitute the projective analogue of the well-known hyperbolic structures on closed surfaces, see also \S\ref{sssec:bhs}. This is equivalent to requiring that the developing map of \(\sigma\) is a diffeomorphism onto a convex domain in \(\rp2\). With the same spirit, we may extend the notion of convex structures as follows.

\begin{defn}\label{defn:branchedconvex}
    A \textit{branched} \(\rp2\)-structure \(\sigma\) on a surface \(S\) is said to be \textit{convex} if its developing map is a branched covering onto a convex domain in \(\rp2\).
\end{defn}

\noindent Convex domains may be further distinguished by the additional property of being \textit{properly convex}. More specifically, a convex domain \(\Omega\) is said to be properly convex if it is bounded in some affine chart. In turn, a projective structure \(\sigma\) on \(S\) is said to be properly convex if \((S,\sigma)\cong\Gamma\backslash\Omega\), where \(\Omega\) is a \(\Gamma\)-invariant properly convex domain in \(\rp2\) where \(\Gamma\) is a discrete group acting freely and properly discontinuously. Once again, complete hyperbolic structures on a closed surface are examples of such structures, although they are far from providing the generic case. We extend the definition to branched structures as follows.

\begin{defn}\label{defn:branchedproperlyconvex}
    A \textit{branched} \(\rp2\)-structure \(\sigma\) on a surface \(S\) is said to be \textit{properly convex} if its developing map is a branched covering onto a properly convex domain in \(\rp2\).
\end{defn}

\smallskip

\noindent According to Choi and Goldman \cite{ChoiGoldman1993Convex}, a real projective structure on a closed surface \(S\) of genus \(g \ge 2\) is properly convex if and only if its holonomy representation lies in the Hitchin component. In the same spirit, in \S\ref{ssec:prop convex BPS} we provide a characterisation of properly convex branched projective structures on closed surfaces of genus \(g \ge 2\). We conclude the present section with some remarks regarding convex projective structures.

\subsubsection{Holonomy representation of properly convex projective structures}\label{sssec:holopropconvexstructures} We shall denote by \(\mathcal{CRP}^2(\,S\,)\) the deformation space of properly convex projective structure on \(S\). In his paper \cite{goldman1990}, Goldman showed that such a deformation space is a smooth manifold of dimension \(16g-16\), thus implying that there exist many non-trivial ways to obtain such structures by deforming hyperbolic ones. In their follow-up \cite{ChoiGoldman1993Convex}, Choi--Goldman show that \(\mathcal{CRP}^2(\,S\,)\) is diffeomorphic to the Hitchin component \(\mathrm{Hit}_3(\,S\,)\) of \(\mathfrak X_3(\,S\,)\), see \S\ref{sec: representation space}. As a direct consequence, every group \(\Gamma\) uniformising a properly convex projective structure on \(S\) is isomorphic to \(\rho\big(\,\pi_1(\,S\,)\,\big)\) for some \(\rho\in\mathrm{Hit}_3(\,S\,)\). 

\smallskip

\subsubsection{Singular Blaschke metric}\label{sssec:blaschkemetric} Another reason that makes properly convex projective structures pivotal is that, by a profound theorems of Cheng--Yau in \cite{cheng1977, cheng1986}, see also \cite{gigena1978, li1990, li1992, loftin2001, sasaki1980}, the properly convex domain \(\Omega\) always admits a \(\Gamma\)-invariant Riemannian metric \(h\), known as the \textit{Blaschke metric}, which consequently descends to \(S\). In the same fashion, every properly convex branched projective structure admits a \textit{singular} Blaschke metric. Indeed, given such a structure, say \((S,\sigma)\), its developing map takes values in some properly convex domain \(\Omega\) by definition. The pullback of the Blaschke metric on this domain yields a singular Riemannian metric \(h\) on \(\widetilde{S}\) that makes the covering projection a local isometry away from the singularities. By construction, \(h\) is invariant under the action of the fundamental group and hence descends to a singular metric, say \(h_\sigma\) on \(S\). See \S\ref{sec: examples branched covering blaschke} for more details.

\smallskip

\subsection{A medley of examples}\label{ssec:examples} Real projective \(\rp2\)-structures are highly ubiquitous. Standard examples arise naturally from classical hyperbolic and Euclidean geometries, whose constant-curvature metrics embed isometrically into \(\rp2\). These classical cases, however, represent only a small part of the theory, as the vast majority of \(\rp2\)-structures are not induced by locally homogeneous Riemannian metrics.  To illustrate their diverse geometric natures and their deep connections to other fields, in the present section we collect examples ranging from branched hyperbolic metrics to affine structures, in particular translation surfaces, and spherical geometries. Given the main purpose of the present work, along the way we also recall the current status of the realisation problem for this kind of structures.

\subsubsection{Branched hyperbolic structures}\label{sssec:bhs} As already alluded in \S\ref{ssec:convexstructures}, hyperbolic structures are classic examples of projective structures: specifically, they are convex projective structures. This happens because a hyperbolic structure is locally modelled on the geometry \(\big(\,\DD, \psu\,\big)\), where \(\DD\) denotes the Poincar\'e disc and \(\psu\cong\pslr2\) its group of orientation-preserving isometries, which naturally embeds into the real projective geometry \(\big(\,\rp2, \pslr3\,\big)\) in the canonical way. Consequently, branched hyperbolic structures provide concrete examples of branched \(\rp2\)-structures on closed surfaces. An explicit example of branched hyperbolic structure is as follows. In the Poincaré disc \(\DD\), consider a regular \(4g\)-gon centred at the origin, so that all its vertices lie at the same distance from the centre. A standard argument in hyperbolic geometry asserts that the internal angles, and consequently their sum, depend exclusively on the distance of the vertices from the centre. In particular, there exists a unique distance for which the sum of the internal angles equals \(2\pi(m+1)\), where \(m\) is an integer satisfying \(0 \le m \le 2g-3\). By identifying opposite sides in the standard manner, the resulting surface has genus \(g\), by Radó's Theorem \cite{Rado1924}, and carries a branched hyperbolic structure with a single branch point of cone angle \(2\pi(m+1)\) by design. For \(m=0\), we notice that this construction yields a genuine hyperbolic structure. 

\smallskip

\noindent Such structures are far from rigid. Indeed, by the Ehresmann–Thurston principle, small deformations of the holonomy representation of such a structure remain the holonomy of a branched hyperbolic structure. It is well known that representations arising as the holonomy of branched hyperbolic structures constitute an open subset in the character variety of \(\pslr2\), a space whose connected components has been classified by Goldman in \cite{goldman1988}. Furthermore, every representation into \(\pslr2\) defines a topological invariant known as the Euler number, denoted by \(\textnormal{eu}(\,\rho\,)\). If \(\rho\) arises as the holonomy of a branched hyperbolic structure with a single branch point of order \(m\), then its Euler number is equal to \(\textnormal{eu}(\,\rho\,)=2-2g+m\). By regarding any of these structures as a branched \(\rp2\)-structure, and thus viewing its holonomy, say \(\rho\) as a representation into \(\pslr3\), its Stiefel–Whitney invariant is equal to \(m \pmod 2\) (more generally it is equal to the parity of the Euler number, \textit{i.e.}, \(\textnormal{eu}(\,\rho\,) \pmod 2\)). The problem of determining which representations arise as the holonomy of some branched projective structure has been considered by several authors in recent years from different perspectives, see \cite{faraco2021}, \cite{faracomaret2025}, \cite{mathews2011}, \cite{mathews2012}. In all of these works, the approach adopted by the authors has a topological flavour since the structures are realised geometrically. On the other hand, as a consequence of the main result of \cite{marche2016}, a dynamical argument implies that, for surfaces of genus two, almost every representation arises as the holonomy representation of some branched hyperbolic structure if the Euler number is non-zero. The same statement is conjectured to hold for surfaces of any genus \(g\ge2\). By contrast, an argument via the Gauss–Bonnet theorem shows that representations into \(\pslr2\) with vanishing Euler number never arise as the holonomy of a branched hyperbolic structure.

\smallskip

\noindent In contrast to the case of genuine, \textit{i.e.}, unbranched, hyperbolic structures on surfaces, branched hyperbolic structures admit much more flexibility, as they can be deformed in the moduli space without altering the holonomy. The underlying idea is that, since a local chart around a branch point is a branched covering map onto some open disc in the model space, it is possible to locally deform the branched cover by changing the branch value. Because the latter branched chart is not compatible with the former one, the resulting structure is different. On the other hand, this deformation takes place within a simply connected open neighbourhood of the branch point, and thus it does not affect the holonomy representation. This surgical method of deforming structures has been extensively utilised in the literature to study the moduli spaces of branched structures, see \textit{e.g.}, \cite{calsetall2014}. In particular, the same surgery applies to the structures that will be discussed in the subsequent sections.

\smallskip

\subsubsection{Elliptic structures}\label{sssec:elliptic} On the other side of the spectrum, opposite to branched hyperbolic structures, we find the elliptic structures, that is, geometric structures locally modelled on \((\rp2, \textnormal{PSO}(3,\mathbb R))\). These structures are intimately related to licensed spherical structures, that is, structures modelled on \((\sph^2, \textnormal{SO}(3,\mathbb R))\). One direction is clear: a spherical structure yields an elliptic one by post-composing the developing map with the natural projection \(\sph^2\to\rp2\). Conversely, the lift of the developing map of any \(\rp2\) structure to \(\sph^2\) yields a spherical structure. As a direct consequence of the Gauss-Bonnet Theorem, a spherical structure on a surface \(S\) with positive genus must have branched points. Spherical structures have been the subject of investigations in recent years, see \cite{EremenkoMondelloPanov2023, FaracoGupta2025, MondelloPanov2016, MondelloPanov2019, MondelloPanov2024}, thus providing further motivations for us to extend the notion to (real) branched projective structures.

\smallskip

\subsubsection{Real affine structures}\label{sssec:realaffine} Between elliptic (spherical) and hyperbolic geometry, we find flat geometry, with Euclidean geometry being a special case that completes the trichotomy of geometries induced by (possibly singular) Riemannian metrics in dimension two. In this case, the model space is the plane \(\R^2\) along with its group of real affine isometries, here denoted as \(\affr\). To transition to the projective setting, the affine group \(\affr\) can be embedded as a closed subgroup of \(\textnormal{PSL}(3,\mathbb{R})\) by viewing the affine plane as the standard open chart inside \(\rp2\). Explicitly, an affine transformation acting on \(\mathbb{R}^2\) is defined by an invertible matrix \(A \in \textnormal{GL}(2,\mathbb{R})\) and a translation vector \(v \in \mathbb{R}^2\), that is, \( x \mapsto Ax + v\). This transformation is mapped to a projective class in \(\textnormal{PSL}(3,\mathbb{R})\) via the injective homomorphism: 
\begin{equation}
    \imath \colon \affr \hookrightarrow \textnormal{PSL}(3,\mathbb{R}), \quad (A,v) \longmapsto  \frac{1}{\sqrt[3]{\,\det(\,A\,)}}\,\begin{pmatrix} A & v \\ 0 & 1 \end{pmatrix} 
\end{equation}

\noindent Geometrically, this embedding identifies \(\affr\) with the subgroup of projective transformations that globally stabilises the line at infinity, which is defined in homogeneous coordinates by the equation \(z=0\). Consequently, any flat geometric structure with real affine holonomy naturally is a projective structure, allowing one to study affine developments and their degenerations directly through the perspective of projective geometry. Unlike the case of complex affine structures, real affine structures on surfaces have been historically less studied. In~\cite{Kuiper1953}, Kuiper described the geodesically complete affine structures on the two-torus \(\mathbb T^2\), demonstrating that they fall into two distinct types: either Euclidean structures (corresponding to flat Riemannian structures) or alternative structures defined by flat, non-Riemannian connections. The classification of affine structures on surfaces was later completed by Nagano--Yagi~\cite{NaganoYagi1974} and Arrowsmith--Furness~\cite{ArrowsmithFurness1975}. Notably, the deformation space of all affine structures on \(\mathbb T^2\) fails to be Hausdorff. On the other hand, Baues~\cite{Baues1999, Baues2000, Baues2010} showed that the subspace of \textit{complete} affine structures on \(\mathbb T^2\) is homeomorphic to \(\mathbb{R}^2\). Along the way, he shows that every faithful and discrete representation \(\pi_1(\,\mathbb T^2\,)\longrightarrow\affr\) arises as the holonomy of some complete affine structure on the torus; hence of some real projective structure. In a recent work, Goldman~\cite{Goldman2025} established that this deformation space of complete affine structures on a torus can be identified geometrically with a cone over a twisted cubic curve in \(\rp3\). Finally, although real affine structures on surfaces of genus greater than one have received little attention, the special case of complex affine structures has, by contrast, been extensively studied, see \textit{e.g.} \cite{Ghaz}.

\smallskip

\subsubsection{Euclidean structures and Translation surfaces}\label{sssec:transurfaces} By further restricting the isometry group of the Euclidean plane, we obtain Euclidean structures, among which we find the well-known translation surfaces. Since orientation-preserving Euclidean isometries---specifically translations---preserve the complex structure on \(\R^2\), these have often been regarded as examples of branched \textit{complex} projective structures. However, by forgetting the complex structure, they may just as well be viewed as branched real projective structures. The beginning of the investigation into the question of which representations into \(\textnormal{Isom}^+(\,\mathbb E^2\,)\) arise as the holonomy of some, necessarily branched, structures dates back to Haupt, see \cite{OH}. Recently, this work was completed first by Ghazouani in \cite{Ghaz} and subsequently refined by the works of \cite{BJJP, fils}.

\smallskip

\subsubsection{Pull-back of convex $\RP^2$-structures via branched coverings}\label{sec: examples branched covering}  Let \(S\) and \(\Sigma\) be closed orientable surfaces of genus \(g,h\ge2\) and let \(f\colon S\longrightarrow \Sigma\) be a branched covering map of degree \(d\). Let \(\mathcal P\subset \Sigma\) be the set of ramification values and let \(\mathcal B=f^{-1}(\,\mathcal P\,)\) be the set of ramification points of \(f\). Notice that \(f\) restricts to a genuine covering map of punctured surfaces \(f\vert_{S\setminus \mathcal B}\colon S\setminus \mathcal B\to \Sigma\setminus \mathcal{P}\) of the same degree. Every \(\rp2\)-structure on \(\Sigma\) naturally pulls back to a branched \(\rp2\)-structure on \(S\) with branch points at \(\mathcal B\). In particular, if \(\sigma\) is a (properly) convex projective structure on \(\Sigma\) then \(f^*\sigma\) is a (properly) convex branched \(\rp2\)-structure on \(S\). In fact, let \(\sigma\) be a \(\rp2\)-structure on \(\Sigma\) and let \(\mathcal A_{\sigma}=\{(U_\alpha,\varphi_\alpha)\}_{\alpha\in I}\) be the corresponding maximal atlas. We thus define a maximal atlas of charts, say \(\mathcal A_{f^*\sigma}=\{(V_\alpha,\psi_\alpha)\}\), on \(S\) in the usual way, \textit{i.e.}, by setting \(V_\alpha=f^{-1}(\,U_\alpha\,)\) and \(\psi_\alpha=\varphi_\alpha\circ f|_{U_\alpha}\). According to Definition \ref{def:branchedprojestructure}, \(\mathcal A_{f^*\sigma}\) defines a branched \(\rp2\)-structure on \(S\) with branch points at \(\mathcal B\). According to Remark \ref{rmk:branchingchain}, if \(\mathcal B=\{p_1,\dots,p_k\}\), then the branching chain is given by
\begin{equation}
    D=\sum_{i=1}^n\,\big(\deg_{p_i}(\,f\,)-1\big)\,p_i,
\end{equation}
where \(\deg_{p_i}(\,f\,)\) is the local degree of \(f\) at \(p_i\). Recall from \S\ref{ssec:devhol_pair} that every (possibly branched) projective structure yields a developing-holonomy pair. For structures obtained via pullback along a (possibly branched) covering map \(f\), the developing-holonomy pair of the pullback structure depends on the choice of a lift of \(f\). In other words, let \((\mathrm{dev}_{\sigma},\rho_\sigma)\) be the developing-holonomy pair of \(\sigma\) on \(\Sigma\) and let \(\widetilde f\) be any lift of \(f\). Then, the developing map \(\mathrm{dev}_{f^*\sigma}\) of \(f^*\sigma\) on \(S\) factors through \(\widetilde f\), that is the identity 
\begin{equation}
    \mathrm{dev}_{f^*\sigma} = \mathrm{dev}_\sigma\circ \widetilde f.
\end{equation}
Notice that \(\mathrm{dev}_{f^*\sigma}\) does not depend on the choice of the lift. In principle, two different lifts yield developing maps that differ by some deck transformation on both the source and target; the latter, however, is naturally absorbed by \(\mathrm{dev}_{\sigma}\). As a consequence, two different lifts of \(f\) yield developing maps that differ only by pre-composition with some deck transformation. Moreover, the holonomy representation of \(f^*\sigma\) on \(S\) factors through \(f_*\), the homomorphism induced by \(f\) at the level of the fundamental groups. In fact, \(\widetilde f\) satisfies the an equivariance property
\begin{equation}
    \widetilde f(\,\gamma\cdot p\,)=f_*(\,\gamma\,)\cdot\widetilde f(\,p\,)
\end{equation}
for every \(\gamma\in\pi_1(\,S\,)\) and every \(p\in{\widetilde{S}}\). As a consequence the developing map \(\mathrm{dev}_{f^*\sigma}\) is equivariant with respect to the representation \(\rho_{f^*\sigma}=\rho_\sigma\circ f_*\). In fact,
\begin{align}
    \mathrm{dev}_{f^*\sigma}(\gamma\cdot p) = \mathrm{dev}_{\sigma}\big(\widetilde{f}(\gamma\cdot p)\big) &= \mathrm{dev}_\sigma\big(f_*(\gamma)\cdot \widetilde{f}(p)\big) \nonumber \\
    &= \rho_\sigma\big(f_*(\gamma)\big)\cdot\mathrm{dev}_\sigma\big(\widetilde{f}(p)\big) = \rho_{f^*\sigma}(\gamma)\cdot\mathrm{dev}_{f^*\sigma}(p).
\end{align}

\smallskip

\noindent The pair \((\,\mathrm{dev}_{f^*\sigma}, \rho_\sigma\circ f_*\,)\) is precisely the one associated to the branched \(\rp2\)-structure realised from the atlas \(\mathcal A_{f^*\sigma}\), see Proposition \ref{prop:BPSasmaps}. By definition, it readily follows that \(\mathrm{dev}_{f^*\sigma}\) is a local diffeomorphism away from \(\mathcal B\) and its image is still contained in the convex subsets \(\Omega\) by design. We finally notice that \(\rho_{f^*\sigma}\) is injective if and only if \(f\) is a genuine covering map. By the Simple Loop Conjecture proved by Gabai in \cite{gabai1985}, the map \(f_*\) is either injective or there exists a simple closed curve in the kernel \(\ker(f_*)\) and the injectivity holds when \(f\) is a genuine covering map. In particular the representation \(\rho_{f^*\sigma}\) never lies in the Hitchin component of \(\mathfrak X_3(\,S\,)\) when \(f\) is a branched covering map.

\begin{rmk}[The trivial representation is realisable]\label{rmk:trivial holo realisable}
    It is worth noting that the previous discussion leads to the following peculiar conclusion. Let \(\Sigma\) be either the sphere \(\sph^2\) or the projective plane \(\rp2\), and let \(S\) be a closed oriented surface of arbitrary genus. If \(f\colon S\longrightarrow \Sigma\) is a branched covering map (notice that we do not impose any conditions on the branching data). Then the natural \(\rp2\)-structure on \(\Sigma\) pulls back to a branched \(\rp2\)-structure on \(S\). In particular, the latter always has trivial holonomy, thus showing that the trivial representation is always realisable.
\end{rmk}

\smallskip

\subsubsection{Realising singular Blaschke metrics}\label{sec: examples branched covering blaschke} 
By adopting the same notation used in \S\ref{sec: examples branched covering}, assume \((\Sigma,\sigma)\) is a properly convex projective structure and let \(h_\sigma\) be Blaschke metric on \(\Sigma\) whose existence has been already discussed in \S\ref{sssec:blaschkemetric}. Let \(f\colon S\longrightarrow\Sigma\) be a branch covering map and let \(\mathcal B\) be the set of ramification points of \(f\). Finally, let \(h_{f^*\sigma}=f^*h_\sigma\) be the pull back metric on \(S\). We aim to show that \(h_{f^*\sigma}\) is a singular metric with integral conical singularities at the points of \(\mathcal B\). For this purpose, let \(p\in\mathcal{B}\) and let \(q=f(\,p\,)\in\mathcal{P}\), and consider open neighbourhoods \(V\) and \(U\) of \(p\) and \(q\), respectively. Let \(z\) be a judicious local coordinate such that \(f\) is of the form \(w=f(\,z\,)=z^{m+1}\) with \(m\ge 1\). Recall that there exists a smooth function \(u\colon U\longrightarrow \R\) such that the Riemannian metric \(h_\sigma\) is of the form \(2e^{2u}|\,\mathrm{d}w\,|^2\) on \(U\). As a consequence, we have that
\begin{equation}
    h_{f^*\sigma}|_V=(\,f^*h_\sigma\,)|_V=2e^{2u(\,z^{m+1}\,)}|\,\mathrm dz^{m+1}\,|^2=2(m+1)^2 e^{2u(\,z^{m+1}\,)}|\,z\,|^{2m}|\,\mathrm dz\,|^2 .
\end{equation}
In other words, the point \(p \in \mathcal{B}\) corresponds to a conical singularity of the metric \(h_{f^*\sigma}\) of order \(m\). Finally, outside the set \(\mathcal{B}\), the tensor \(h_{f^*\sigma}\) is a smooth Riemannian metric in the classical sense.

\begin{rmk}
    The main insight from the Example in \S\ref{sec: examples branched covering} is that, starting from a properly convex \(\rp2\)-structure on \(\Sigma\), one can construct a branched one on another surface \(S\) (still of genus \(g\ge2\)), equipped with a metric having integral conical singularities. Moreover, the associated holonomy representation is not Hitchin and the developing map fails to be a local diffeomorphism onto a convex domain, despite its image still being contained in \(\Omega\). As a consequence, the resulting structure is, a priori, richer (hence stronger) than that of Definition \ref{def:branchedprojestructure}. 
\end{rmk}


\smallskip

\section{Geometrisation of representations}\label{sec:geometrisation}

\noindent In the present section, we aim to investigate the holonomy map of branched \(\mathbb{R}\mathbb{P}^2\)-structures. Specifically, our goal is to determine whether a given representation arises as the holonomy of some branched structure. In the realisation process, it will emerge that branched projective structures are flexible in the sense that the a representation does not encode enough geometric information to determine a unique structure. Moreover, we notice that, as soon as we drop the condition for a structure to be strictly convex, even in the unbranched framework a holonomy representation does not determined a structure uniquely. This phenomenon has been already observe by Goldman in \cite[Corollary, pag. 793]{goldman1990} and more recently by Fujii in \cite{fujii2026grafting}. The purpose of this section is to establish Theorem \ref{thma:geometrisation}, which we restate here for convenience.

\begin{thm}\label{thm:geometrisation}
    Every representation \(\rho\colon\pi_1(\,S\,)\longrightarrow\pslr3\) arises as the holonomy of some branched projective structure on \(S\).
\end{thm}

\noindent In the next sections, we will provide a general argument to geometrise any representation into \(\pslr{3}\). In \S\ref{ssec:hitchinbranched}, we show that every Hitchin representation arises as the holonomy of some branched projective structure on surfaces of genus \(g\ge2\), by providing a refined argument. 

\begin{rmk}
    In what follows, we shall assume that \(\rho\) is non-trivial. Indeed, we have already observed in Remark \ref{rmk:trivial holo realisable} that such a representation always arises as the holonomy of a branched projective structure on any surface of positive genus.
\end{rmk}


\subsection{Initialisation of the geometrisation process}\label{ssec:geom process initialisation} In this section, we outline the general scheme of the geometrisation process that we shall develop in \S\ref{ssec:base surface}, \S\ref{ssec:handle realisation}, and \S\ref{ssec:final assembly}. We start defining a decomposition of the surface into subsurfaces of lower complexity. On each of these components, a given representation restricts to a collection of sub-representations in \(\pslr3\), which we shall geometrise individually in \S\ref{ssec:base surface} and \S\ref{ssec:handle realisation}. As we shall see, the process of realising these individual pieces allows for a certain flexibility, highlighting the fact that branched real projective structures are far from being rigid. Finally, in \S\ref{ssec:final assembly}, we assemble these building blocks to yield a structure with the desired holonomy.

\smallskip

\noindent Let \(S\) be a closed, connected, and orientable surface of genus \(g\ge2\) and let \(\rho\colon\pi_1(\,S,b\,)\longrightarrow \pslr3\) be any representation, where \(b\in S\) is a preferred base point (the way \(b\in S\) is chosen will become clear later on during the realisation process). Let \(\mathcal B=\big\{\,\alpha_1,\beta_1,\,\dots,\alpha_g,\beta_g\,\big\}\) be a generating set of \(\pi_1(S,b)\) that satisfies the well-known presentation \( \pi_1(S,b)=\big\langle\,\alpha_1,\beta_1,\,\dots,\alpha_g,\beta_g\,\,|\,\,\big[\alpha_1,\,\beta_1\big]\cdots\big[\alpha_g,\,\beta_g\big]=1\,\big\rangle\,\). In the following we shall adopt the following terminology.

\begin{defn}[Handles, handle-generators]\label{defn:handlegenerators}
On a marked surface \(S\) of some positive genus \(g\ge0\), a \textit{handle} is an embedded subsurface \(H\) that is homeomorphic to \(S_{1,1}\), and a \textit{handle-generator} is a simple closed curve that is one of the generators of \(\pi_1(S_{1,1})\). A pair of handle-generators for a handle will refer to a pair of simple closed curves \(\{\alpha, \beta\}\) that generate \(\pi_1(S_{1,1})\); in particular, \(\alpha\) and \(\beta\) intersect once.
\end{defn}

\noindent For every \(i=1,\dots,g\), set \(\gamma_i=\big[\alpha_i,\,\beta_i\big]\), let \(H_i\) be the handle bounded by \(\gamma_i\) and consider its fundamental group. We decompose \(S\) along the curves \(\gamma_i\) into \(g\) handle and a a genus-\(0\) surface with \(g\) boundary components. Notice that its interior is homeomorphic to the \(g\)-punctured sphere and we set \(b\in S_{o,g}\) as the principal basepoint. We specify basepoints \(b_i \in H_i\) as follows: on each handle we take \(b_i\) on the boundary. To specify how \(\pi_1(H_i, b_i)\) and \(\pi_1(S, b)\) relate, take a combinatorial tree \(\textnormal T\) dual to the decomposition of \(S\), with one vertex for each \(b_i\), and a map \(\tau\colon \textnormal T \longrightarrow S\) mapping the vertices of \(\textnormal T\) onto the corresponding \(b_i\). This gives well-defined paths between the \(b_i\).  See Figure \ref{fig:decomposition}.

\begin{figure}[htbp]
  \centering
  \begin{tikzpicture}[scale=1.625]

    \draw[black!50] (0,0) circle (2.25cm);
    \shade[ball color = white, opacity=0.25] (0,0) circle (2.25cm);

    \foreach \theta/\i in {0/1, 72/2, 144/3, 216/4, 288/5} {
    \fill[white, rotate=\theta] (1.5, 0.5) to[out=300, in=60] (1.5,-0.5) to[out=115, in=270] (1.365,0) to[out=90, in=245] (1.5, 0.5);
    \draw[thin, apricot, rotate=\theta] (1.65,0) to[out=-15, in=195] (2.25,0);
    \draw[thin, apricot, dashed, rotate=\theta] (1.65,0) to[out=15, in=165] (2.25,0);
    \node[apricot] at ({(\theta)}:2.375) {\tiny \(\alpha_{\i}\)};
    \draw[thin, sky, rotate=\theta] (1.5,0) ellipse (0.25 and 0.75);
    \node[sky] at ({(\theta)}:1.125) {\tiny \(\beta_{\i}\)};
    }

    \foreach \theta in {0, 72, 144, 216, 288} {
    \draw [black!50, thin, rotate=\theta] 
            (1.5, 0.5) to[out=300, in= 60] (1.5, -0.5);
    }

    \foreach \theta in {0, 72, 144, 216, 288} {
    \draw [black!50, thin, rotate=\theta] 
            (1.525, -0.55) to[out=120, in=240] (1.525,  0.55);
    }

    \begin{scope}[shift={(3.5,-4.5)}]
        \draw[black!50] (0,0) circle (2.25cm);
        \shade[ball color = white, opacity=.2] (0,0) circle (2.25cm);

        \foreach \theta/\i in {0/1, 72/2, 144/3, 216/4, 288/5} {
        \fill[white, rotate=\theta] (1.5, 0.5) to[out=300, in=60] (1.5,-0.5) to[out=115, in=270] (1.365,0) to[out=90, in=245] (1.5, 0.5);
        }

        \foreach \theta in {0, 72, 144, 216, 288} {
        \draw [black!50, thin, rotate=\theta] 
            (1.5, 0.5) to[out=300, in= 60] (1.5, -0.5);
        }

        \foreach \theta in {0, 72, 144, 216, 288} {
        \draw [black!50, thin, rotate=\theta] 
            (1.525, -0.55) to[out=120, in=240] (1.525,  0.55);
        }

        \foreach \theta/\i in {0/1, 72/2, 144/3, 216/4, 288/5} {
        \draw[thin, red] (0,0)--(\theta:0.75);
        \fill[black, rotate=\theta] (0.75,0) circle (0.03);
        \draw[thin, plum, rotate=\theta] 
        (20:2.25) to[out=180, in= 90] (  0:0.75)
        ( 0:0.75) to[out=270, in=180] (-20:2.25);

        \draw[thin, plum, dashed, rotate=\theta] 
        (-20:2.25) to[out=180, in=270] ( 0:0.90)
                   to[out=90 , in=180] (20:2.25);
  
        \node[plum] at ({(\theta+20)}:2.625) {\tiny \(\big[\,\alpha_{\i}, \beta_{\i}\,\big]\)};
        \node[black] at ({(\theta+10)}:0.65) {\tiny \(b_{\i}\)};
        \node[black] at (30:0.15) {\tiny \(b\)};
        \fill[plum] (0,0) circle (0.03);
        }   
    \end{scope}

  \end{tikzpicture}
  \caption{On the top left, a system of handle generators as in Definition \ref{defn:handlegenerators}; on the bottom right, their commutators along with a possible combinatorial tree \(\textnormal T\) drawn in red.}
  \label{fig:decomposition}
\end{figure}
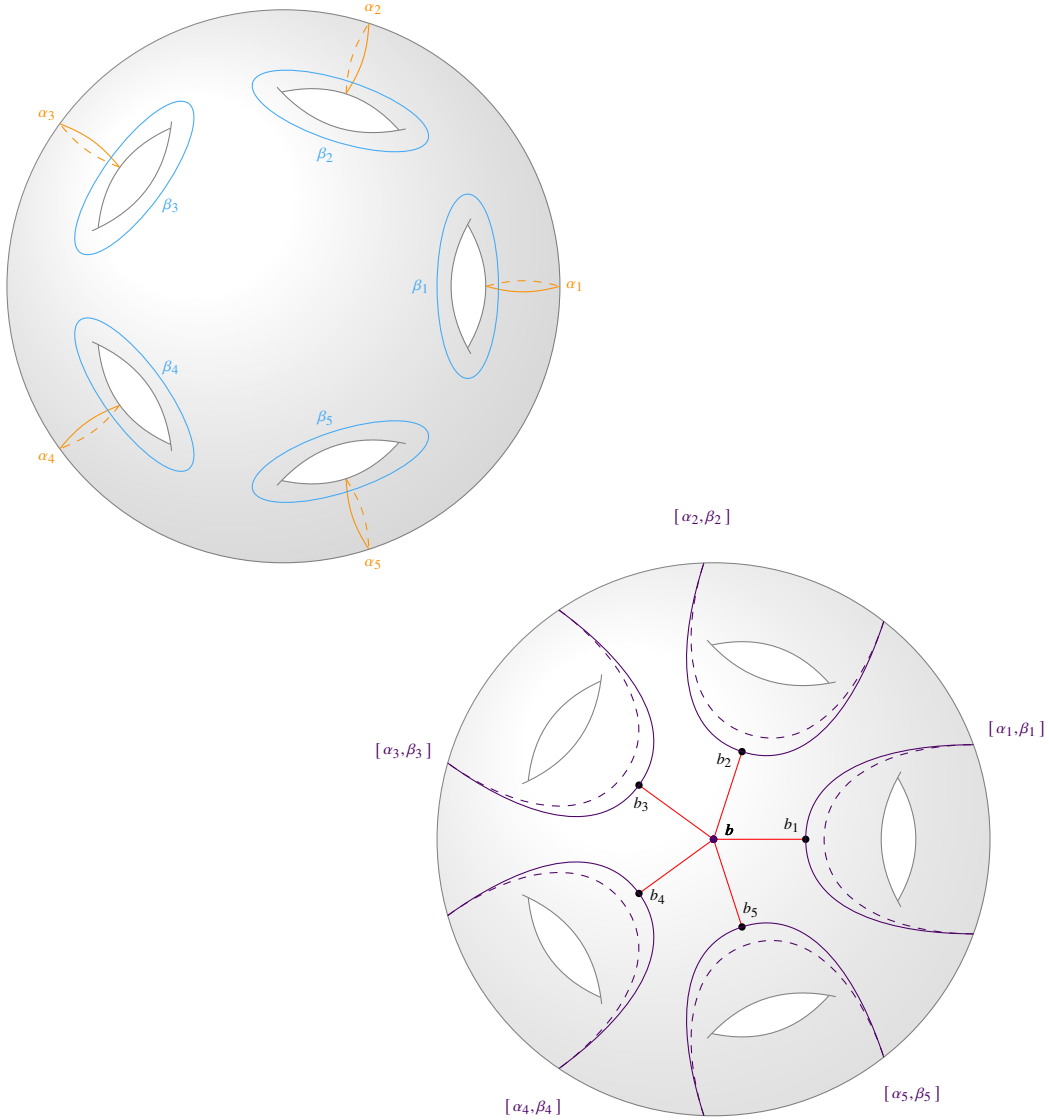

\smallskip

\noindent We have inclusions \(\imath_i\colon\pi_1(H_i, b_i) \hookrightarrow \pi_1(S, b_i)\) (note basepoints). Let \(\delta_i\) be the unique path from \(b\) to each \(b_i\) along the tree \(\mathcal T\), then we have well-defined isomorphisms as
\begin{equation}
    \jmath_i\colon \pi_1(S, b_i) \xrightarrow{\;\;\;\cong\;\;\;}
    \pi_1(S, b), \quad \xi \longmapsto \delta_i \,\xi\,\delta_i^{-1}.
\end{equation}
\noindent For every \(i=1,\dots,g\), define a representation \(\rho_i\colon \pi_1(H_i, b_i) \longrightarrow\pslr3\) as the composition 
\begin{equation}
    \rho_i\colon\pi_1(H_i, b_i) \xrightarrow{\;\;\;\imath_i\;\;\;}\pi_1(S, b_i) \xrightarrow{\;\;\;\jmath_i\;\;\;} \pi_1(S, b) \xrightarrow{\;\;\;\rho\;\;\;} \pslr3.
\end{equation}
Finally, define \(\rho_o\colon\pi_1(\,S_{o,g},b\,)\to\pslr3\) as the representation arising by pre-composing the representation \(\rho\) with the inclusion \(\imath_o\colon \pi_1(\,S_{o,g},b\,)\hookrightarrow\pi_1(S,b)\). In the next sections, we will show how to geometrise generic representations of the fundamental groups of an \(n\)-punctured sphere and of the once-punctured torus as the holonomy of some (possibly branched) real projective structure on the respective surfaces. To achieve this, we shall need to realise closed polygonal chains of projective segments, defined by the iterated images of a carefully chosen base point. We therefore introduce the following preliminary lemma, which will be employed in the subsequent sections.

\begin{lem}[Base point]\label{lem:base point}
    Let \(\rho\colon\pi_1(\,S_g\,)\longrightarrow\pslr3\) be a representation and let \(\Gamma\) be its image. Then there are infinitely many points \(p \in \rp2\) such that \(p\) is not fixed by any non-trivial element of \(\Gamma\).
\end{lem}

\begin{proof}
    Since the surface group \(\pi_1(\,S_g\,)\) is finitely generated, its image \(\Gamma \subset \pslr3\) is a countable set. For every non-trivial element \(g \in \Gamma \setminus \{\,\textnormal{id}\,\}\), its fixed-point set \(\mathrm{Fix}(\,g\,)\) in \(\rp2\) corresponds to the projectivisation of the eigenspaces of a non-scalar \(3 \times 3\) matrix. The eigenspaces of such a matrix have real dimension at most \(2\). Therefore, upon projectivisation, the set \(\mathrm{Fix}(\,g\,)\) consists of at most a projective line and an isolated point.  In particular, \(\mathrm{Fix}(\,g\,)\) is a proper algebraic sub-variety of positive codimension, and consequently, it has Lebesgue measure zero in \(\rp2\). It follows that the union of all fixed points of non-trivial elements,
    \begin{equation}\label{eq: bad base points}
        \textnormal{B}(\,\rho\,) = \bigcup_{g \in \Gamma \setminus \{\,\textnormal{id}\,\}} \mathrm{Fix}(\,g\,),
    \end{equation}
    is a countable union of sets of measure zero, and thus \(\textnormal{B}(\,\rho\,)\) itself has measure zero in \(\rp2\). Therefore, its complement \(\rp2 \setminus \textnormal{B}(\,\rho\,)\) is a full-measure, dense set. As a consequence, we can readily conclude that there exist infinitely many points \(p \in \rp2 \setminus \textnormal{B}(\,\rho\,)\), each of which is not fixed by any non-trivial element of \(\Gamma\).
\end{proof}

\smallskip

\noindent Having completed the initialisation of the geometrisation process, we now proceed to realise the individual representations. The construction presented in the next section will form the core of the entire process. In the following, we shall adhere to the same notation introduced here unless otherwise specified. In particular, we shall not always explicitly recall objects that have already been defined in the preceding text.

\smallskip

\subsection{Realisation of the core surface}\label{ssec:base surface} We begin with the realisation of the primary base surface that serves as the core subsurface onto which handles, as realised in subsequent \S\ref{ssec:handle realisation}, will be attached. More specifically, the main result of the present subsection is the following

\begin{prop}\label{prop:coredisc}
    Let \(\pi_1(S_{o,n}) = \big\langle \gamma_1, \dots, \gamma_n \;|\; \gamma_1 \gamma_2 \cdots \gamma_n = 1 \big\rangle\) and let \(\rho \colon \pi_1(S_{o,n}, b) \longrightarrow \pslr3\) be a non-trivial representation. Then there exists a branched projective structure on a closed disc \(\Delta\) with at most a single branch point in its interior and with a piecewise geodesic boundary oriented such that the disc lies to its left. Under the developing map, this boundary develops onto a chain defined by
    \begin{equation}\label{eq:chain}
        p_1 \longmapsto \rho(\,\gamma_n\,)(\,p_1\,) = p_n            \longmapsto \rho(\,\gamma_{n-1}\,)(\,p_n\,) = p_{n-1}  \longmapsto \cdots 
            \longmapsto \rho(\,\gamma_1\,)(\,p_2\,) = p_1
    \end{equation}
    with \(n\) corner points. Moreover, the winding number of this chain with respect to the branch point exceeds the order of the branch point by one.
\end{prop}

\noindent Strictly speaking, the representation is used to realise a branched projective structure on a closed disc rather than directly on an \(n\)-punctured sphere. By identifying the \(n\) corner points yields a space that, while not homeomorphic to an \(n\)-punctured sphere, is homotopically equivalent to it and thus have the same fundamental group. As we shall see in \S\ref{ssec:final assembly}, defining the structure on this homotopically equivalent space is sufficient for our purposes. The rest of the present section is devoted to prove the aforementioned Proposition. We split the argument of the proof into three subsections.

\smallskip

\subsubsection{Defining a polygonal chain on the sphere \(\sph^2\)}\label{sssec:polychain} We begin with some preliminary considerations are in order. We may assume without loss of generality that \(\rho(\,\gamma_i\,)\) is not trivial for all \(i=1,\dots,n\). Indeed, if any of the \(\rho(\,\gamma_i\,)\in\pslr3\) were trivial, the chain as defined in Equation \eqref{eq:chain} would have one segment less. Notice that the representation being non-trivial ensures that at least one \(\rho(\,\gamma_i\,)\) is non-trivial.

\smallskip

\noindent We define a polygonal chain in \(\rp2\) as follows. Let \(p_1\in\rp2\) be any point and let \(\mathcal V=\{\,p_1,\dots,p_n\,\}\) be the collection of points defined recursively as \(p_i=\rho(\,\gamma_{i}\,)(\,p_{i+1}\,)\). 

\begin{claim}\label{claim:choose the base point}
    We can choose the base point \(p_1\) in such a way that \(p_{i+1}\) is not fixed by \(\rho(\,\gamma_{i}\,)\), that is, \(p_i\neq p_{i+1}\) for every \(i=1,\dots,n\) (where indexes are taken modulo \(n\)).
\end{claim}

\begin{proofclaim}
    For every \(i\in\big\{\,1,\dots,n\,\big\}\), set \(g_i=\rho(\,\gamma_i\,)\). The following chain of implications hold: 
    \begin{equation}
        \begin{aligned}
        p_i=g_i(\,p_{i+1}\,)=p_{i+1}&\Longleftrightarrow\,g_i\,\Big(\,g_{i+1}\,\cdots\,g_n(\,p_1\,)\,\Big)\,=\, g_{i+1}\,\cdots\,g_n(\,p_1\,) \\
        &\Longleftrightarrow\,g_{i+1}\,\cdots\,g_n(\,p_1\,)\in\textnormal{Fix}\big(\,g_i\,\big) \\
        &\Longleftrightarrow p_1 \in \big(\,g_{i+1}\,\cdots\,g_n\,\big)^{-1}\textnormal{Fix}\big(\,g_i\,\big) \\
        &\Longleftrightarrow p_1 \in \textnormal{Fix}\Big(\,\big(\,g_{i+1}\,\cdots\,g_n\,\big)^{-1}g_i\big(\,g_{i+1}\,\cdots\,g_n\,\big)\,\Big),
        \end{aligned}
    \end{equation}
    \textit{i.e.}, if \(p_1\) belongs to the set of fix points of some element in \(\Gamma\). By Lemma \ref{lem:base point}, there always exists a point (in fact, infinitely many) that are not fixed by every element of \(\Gamma\). Anyone of these points satisfy the desired condition.
\end{proofclaim}

For each \(i \in \{1, \dots, n\}\), we shall denote by \(\ell_i \subset \rp2\) the unique projective line passing through \(p_i\) and \(p_{i+1}\). Recall that, in the real projective plane, every pair of points can be joined by two different segments, both arising from the unique line passing through them. We thus need to establish a criterion for joining the points in the collection \(\mathcal{V}\).

\begin{claim}
    There exists a line \(\ell_\infty\) disjoint from \(\mathcal V\) which we designate as the line at the infinity. Moreover, among all possible chains of projective segments joining the vertices of \(\mathcal{V}\) along \(\ell_i\), there exists a unique choice of segments \(s_1, \dots, s_n\), with \(s_i \subset \ell_i\), such that the closed polygonal chain \(\mathcal P\) defined as
    \begin{equation}\label{eq:chain two}
        p_1 \xrightarrow{\;\;s_n\;\;} \rho(\,\gamma_n\,)(\,p_1\,) = p_n \xrightarrow{\;\;s_{n-1}\;\;} \rho(\,\gamma_{n-1}\,)(\,p_n\,) = p_{n-1} \xrightarrow{\;\;s_{n-2}\;\;} \dots \xrightarrow{\;\;s_1\;\;} \rho(\,\gamma_1\,)(\,p_2\,) = p_1
    \end{equation}
is contained in the interior of a non-degenerate conic \(\mathcal{C} \subset \mathbb{RP}^2\) disjoint from \(\ell_\infty\).
\end{claim}

\begin{proofclaim}
    Let \(\mathcal{V} = \{p_1, \dots, p_n\}\) be the collection of points just defined. Since \(\mathcal{V}\) is finite, there exists a line at infinity \(\ell_{\infty}\) disjoint from it. Recall that two straight lines in \(\rp2\) always intersect at exactly one point. Let \(\ell_i\) be the unique line joining \(p_i\) and \(p_{i+1}\). Notice that these points always single out two geodesic arcs; we set \(s_i\) the (unique) arc of \(\ell_i\) disjoint from \(\ell_\infty\). We define \(\mathcal P\) the resulting polygonal chain. By design, \(\mathcal P\) is entirely contained in the interior of some non-degenerate conic and disjoint from the designated line at infinity \(\ell_\infty\).
\end{proofclaim}

\begin{rmk}
    Alternatively, one may consider the complement \(\mathbb{A}^2 =\mathbb{RP}^2 \setminus \ell_{\infty}\) forms an affine chart containing \(\mathcal{V}\), in which any pair of distinct points determines a unique line and a unique straight segment connecting them. 
    Let \(s_i\) be the unique segment joining \(p_i\) and \(p_{i+1}\). We define \(\mathcal P_{\textnormal{aff}}\) the polygonal chain arising from the concatenation of these segments. By design, \(\mathcal P_{\textnormal{aff}}\) is a closed polygonal chain. Moreover, the cardinality of \(\mathcal V\) being finite, the chain \(\mathcal P_{\textnormal{aff}}\) is compact and hence contained in the interior of some ellipse in \(\mathbb A^2\). By embedding \(\mathbb A^2\hookrightarrow \rp2\), we get the desired chain of segments \(\mathcal P\subset \rp2\).
\end{rmk}

\noindent Let \(\pi_{\textnormal{std}} \colon \sph^2 \longrightarrow \rp2\) be the standard covering projection. The designated line \(\ell_\infty\) lifts to a well-defined great circle on the sphere, denoted by \(\mathcal E\), which we shall regard as the \textit{equator} separating the lifts of \(\rp2 \setminus \ell_\infty\). There are two possible lifts for the polygonal chain \(\mathcal{P}\). We choose either of these lifts and denote by \(\mathcal{U}\subset\sph^2\) the hemisphere containing the latter. With a small abuse of notation, we denote such a chain also with \(\mathcal P\), its set of vertices also with \(\mathcal V\) and \(s_i\) denotes the \textit{oriented} edge joining \(p_{i+1}\) to \(p_i\). Next, we denote by \(\mathcal{L}\) the other hemisphere. Finally, let \(q\in\mathcal L\) be any point. See Figure \ref{fig:chain on the sphere}. From now on, we shall work on the sphere endowed with its natural real projective structure, which makes \(\pi_{\textnormal{std}}\) a projective map. By recalling that great circles project to projective lines via the projection \(\pi_{\textnormal{std}}\), we shall adopt the following terminology: we shall define every arc of a great circle as a \textit{projective segment}.

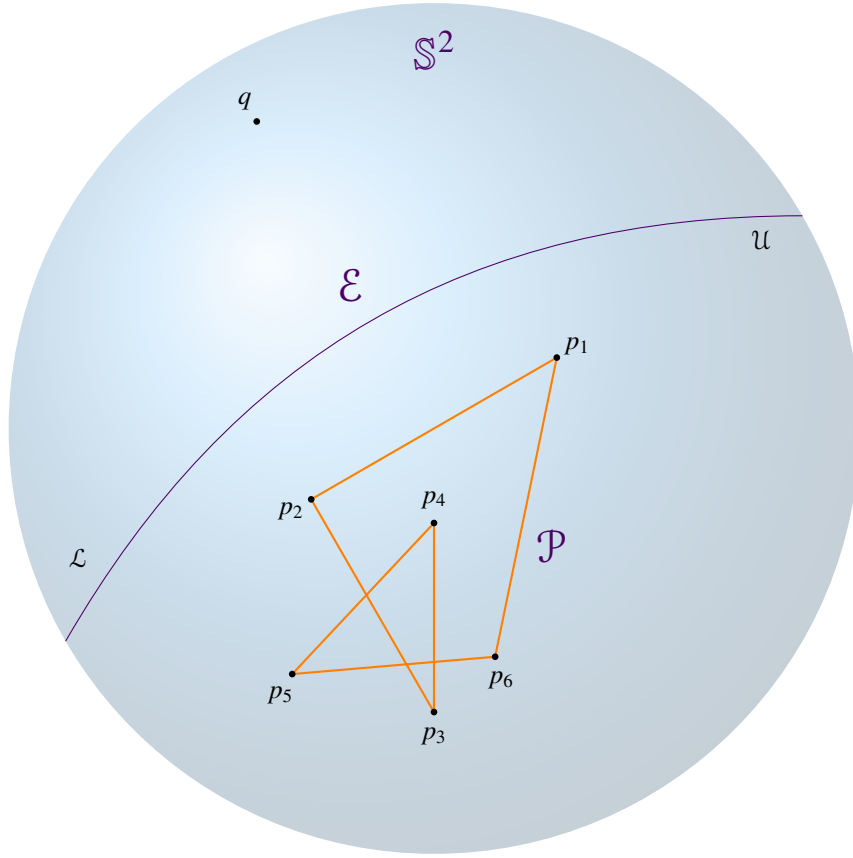
\begin{figure}[htbp]
  \centering
  \begin{tikzpicture}[scale=2.5]


    \draw[white, opacity=0.2] (0,0) circle (2.25cm);
    \shade[ball color=sky, opacity=.25, shading angle=0] (0,0) circle (2.25cm);
    \draw[plum] (210:2.25) to[bend left] (030:2.25);

    \draw[thick, orange] (210:0.75)--
                        (030:0.75)--
                        (285:1.25)--
                        (240:1.50)--
                        (270:0.5 )--
                        (270:1.5 )--
                        (210:0.75);
    \foreach \p in {(210:0.75), (030:0.75), (285:1.25), (240:1.50), (270:0.5), (270:1.5)} {
        \fill \p circle (0.5pt);
    }

    \fill (120:1.875) circle (0.5pt);
    \node at (030:2    ) {\(\mathcal U\)};
    \node at (200:2    ) {\(\mathcal L\)};
    \node at (120:2    ) {\(q\)};
    \node[plum] at (120:0.875) {\huge \(\mathcal E\)};
    \node[plum] at (315:0.875) {\huge \(\mathcal P\)};

    \node at (030:0.875) {\(p_1\)};
    \node at (285:1.375) {\(p_6\)};
    \node at (240:1.625) {\(p_5\)};
    \node at (270:0.375) {\(p_4\)};
    \node at (270:1.625) {\(p_3\)};
    \node at (210:0.875) {\(p_2\)};

    \node[plum] at (090:2.000) {\huge \(\sph^2\)};
    
  \end{tikzpicture}
  \caption{The figure shows the sphere \(\sph^2\) divided into the two hemispheres determined by the lift of \(\mathcal E\). Furthermore, it depicts the polygonal chain \(\mathcal P\subset\mathcal U\) once lifted to the universal cover \(\sph^2\) of \(\rp2\).}
  \label{fig:chain on the sphere}
\end{figure}

\smallskip

\subsubsection{Projective polygons}\label{sssec:proj poly} In the present subsection, we use the polygonal chain \(\mathcal P\) as defined in \S\ref{sssec:polychain}, the equator \(\mathcal E\) and the point \(q\in\mathcal L\) to determine a finite collection of \(2n\) polygons such that, once properly combined as explained in \S\ref{sssec:projective disc}, they determine the desired projective structure on a disc with boundary. More specifically, we shall determine a collection \(\mathcal Q\) of \(n\) embedded quadrilaterals and a collection \(\mathcal T\) of \(n\) embedded triangles. We begin by defining the collection \(\mathcal Q\) as follows.

\smallskip 

\noindent Let \(\mathcal{P}\) be a polygonal chain as defined in Equation \eqref{eq:chain}, with vertices \(\mathcal V=\{p_1, \dots, p_n\}\) and recall that the segments \(\{\,s_i\,\}\) of \(\mathcal P\) are oriented such that \(p_{i+1}\) is connected to \(p_i\), with indices taken modulo \(n\), so that \(p_1\) is connected to \(p_n\). By recalling that the boundary of the hemisphere is the equator \(\mathcal{E}\), we orient the equator in such a way that it keeps the open hemisphere \(\mathcal U\) on its left. For each vertex \(p_i\), we fix a point, say \(q_i \in \mathcal{E}\), with \(q_i \neq q_j\) for \(i \neq j\), and then connect \(p_i\) to \(q_i\) via a projective segment \(r_i\), oriented from \(p_i\) to \(q_i\). For our convenience, we shall denote by \(-r_i\) the oriented projective segment connecting \(q_{i}\) to \(p_{i}\). Notice that, by construction, the projective segments \(-r_i\) and \(r_{i}\) coincide but have opposite orientation. Moreover, we can make sure that two consecutive segments do not intersect. The points \(q_i\) and \(q_{i+1}\) partition the equator into two sub-arcs. Let \(e_i \subset \mathcal{E}\) the unique sub-arc such that the closed loop formed by the concatenation 
\begin{equation}
    p_{i+1} \xrightarrow{\;\;s_i\;\;} p_i \xrightarrow{\;\;r_i\;\;} q_i \xrightarrow{\;\;e_i\;\;} q_{i+1} \xrightarrow{\;\;-r_{r+1}\;\;} p_{i+1}
\end{equation}

\noindent bounds a quadrilateral \(\textnormal{Q}_i\) on its left, see Figure \ref{fig:quadrilateral}.

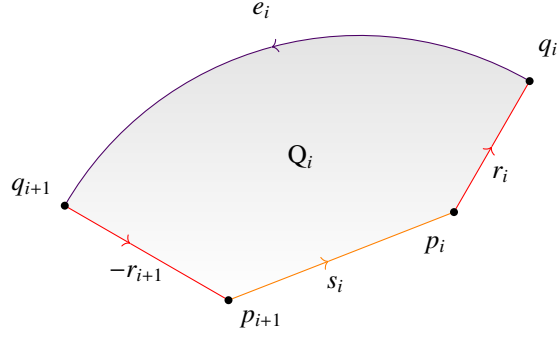
\begin{figure}[htbp]
  \centering
  \begin{tikzpicture}[scale=2]

    \fill[white, opacity=.25, shading angle=0] (150:2.25) arc (150:060:2.25)--(060:1.25)--(150:1.00)--(150:2.25);
    \draw[plum      ] (150:2.25) arc (150:105:2.25);
    \draw[plum, <-  ] (105:2.25) arc (105:060:2.25);
    \draw[red, ->    ] (150:2.25)--(150:1.75);
    \draw[red] (150:1.75)--(150:1.00);
    \draw[red    ] (060:2.25)--(060:1.75);
    \draw[red, <-] (060:1.75)--(060:1.25);
    \draw[orange, ->] (150:1.00)--(105:0.785);
    \draw[orange] (105:0.785)--(060:1.25);

    \foreach \p in {(150:2.25), (060:2.25), (060:1.25), (150:1.00)} {
        \fill[black] \p circle (0.75pt);
        }

    \node at (150:2.5 ) {\(q_{i+1}\)};
    \node at (060:2.5 ) {\(q_{i}\)};
    \node at (060:1) {\(p_{i}\)};
    \node at (150:0.75) {\(p_{i+1}\)};

    \node at (105:2.5  ) {\(e_{i}\)};
    \node at (105:0.625) {\(s_{i}\)};
    \node at (055:1.625) {\(r_{i}\)};
    \node at (155:1.625) {\(-r_{i+1}\)};
    
    \node at (105:1.5) {\(\textnormal{Q}_i\)};
  \end{tikzpicture}
  \caption{A quadrilateral \(\textnormal{Q}_i\). The edges are all oriented so that they leave the quadrilateral on their left.}
  \label{fig:quadrilateral}
\end{figure}

\smallskip

\noindent We define \(\mathcal Q\) to be the collection of quadrilaterals defined in this way. We notice that two consecutive quadrilaterals \(\textnormal{Q}_i\) and \(\textnormal{Q}_{i-1}\) always overlap along their common edge \(r_i\). Moreover, we do not require that these quadrilaterals do not overlap in their interiors, \textit{e.g.}, see Figure~\ref{fig:upper hemi}.

\smallskip

\begin{figure}[htbp]
  \centering
  \begin{tikzpicture}[scale=2]
    \pattern [pattern=north west lines, pattern color=plum, opacity=0.25] (000:2.25)--(000:2.75) arc (000:360:2.75)--(000:2.25) arc (360:000:2.25);
    
    \draw[white, opacity=0.25] (0,0) circle (2.25cm);
    \shade[ball color = sky, opacity=.20, shading angle=0] (0,0) circle (2.25cm);
    \begin{scope}[shift={(0,0.75)}]
        \draw[thick, orange] (210:0.75)--
                             (030:0.75)--
                             (285:1.25)--
                             (240:1.50)--
                             (270:0.5 )--
                             (270:1.5 )--
                             (210:0.75);
        
        \draw[red] (210:0.75) -- (160  :1.875);
        \draw[red] (030:0.75) -- (055  :1.650);
        \draw[red] (285:1.25) -- (302.5:2.850);
        \draw[red] (240:1.50) -- (225  :2.720);
        \draw[red] (270:0.50) -- ( 90  :1.500);
        \draw[red] (270:1.50) -- (270  :3.000);
        \foreach \p in {(210:0.75), (030:0.75), (285:1.25), (240:1.50), (270:0.5), (270:1.5)} {
        \fill \p circle (0.5pt);
        }        

        \node at (025:0.875) {\(p_1\)};
        \node at (285:1.375) {\(p_6\)};
        \node at (240:1.625) {\(p_5\)};
        \node at (285:0.375) {\(p_4\)};
        \node at (275:1.600) {\(p_3\)};
        \node at (210:0.875) {\(p_2\)};
    \end{scope}
    \draw[thick, plum] (0,0) circle (2.25cm);

    \begin{scope}[shift={(0,0.75)}]
        \foreach \p in {(160  :1.875), (055  :1.650), (302.5:2.850), (225  :2.720), ( 90  :1.500), (270  :3.000)} {
        \fill[red] \p circle (0.75pt);
        }
        
        \node at (055  :1.825) {\(q_1\)};
        \node at (302.5:3.000) {\(q_6\)};
        \node at (225  :2.850) {\(q_5\)};
        \node at ( 90  :1.625) {\(q_4\)};
        \node at (270  :3.150) {\(q_3\)};
        \node at (160  :2.000) {\(q_2\)};

        \node[plum] at (315:0.875) {\huge \(\mathcal P\)};
    \end{scope}

    \node at (045:2.50) {\(\mathcal L\)};
    \node at (225:2.00) {\(\mathcal U\)};
    \node at (180:2.40) {\(\mathcal E\)};

  \end{tikzpicture}
  \caption{Subdivision of the hemisphere \(\mathcal U\) into overlapping polygons bounded by projective segments.}
  \label{fig:upper hemi}
\end{figure}
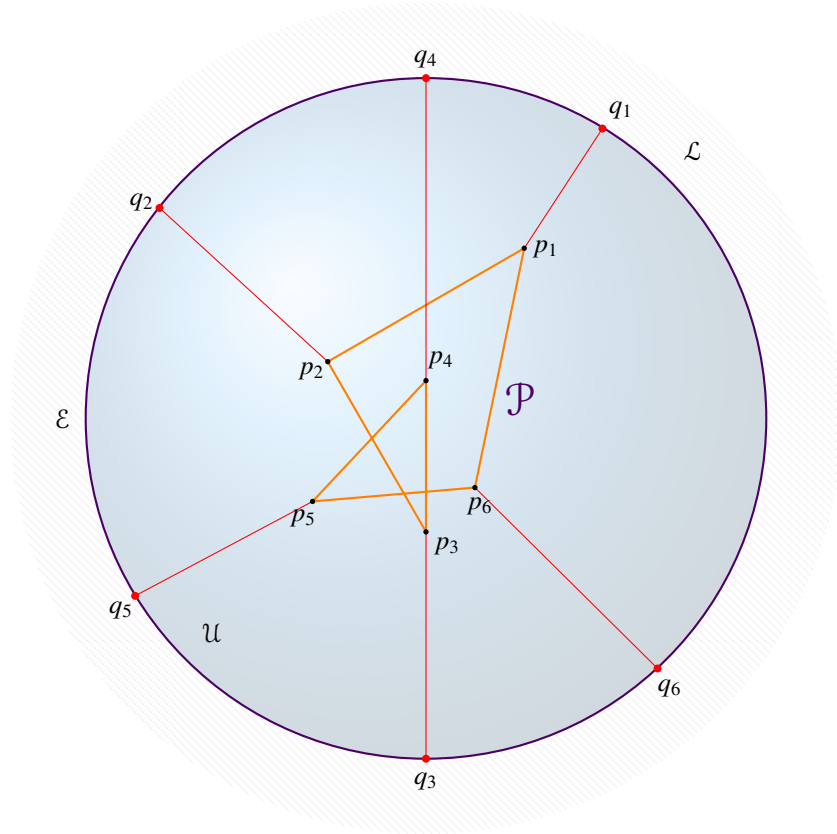

\noindent We now define the collection \(\mathcal T\) of projective triangles. For this purpose, we consider the hemisphere \(\mathcal L\) along with the equator \(\mathcal E\). On this latter, the collection of points \(\{\,q_1,\dots,q_n\,\}\) just defined determine the collection of arcs \(e_i\). We stress that two arcs may very well overlap in their interiors, \textit{e.g.}, according to Figure \ref{fig:upper hemi}, the arcs \(e_1\) and \(e_3\) overlap. Moreover, recall that there is a preferred point \(q\in\mathcal L\). Let \(c_i\) be the oriented projective segment from \(q_i\) to \(q\). Similarly to the convention adopted above, for our convenience we shall denote by \(-c_i\) the oriented projective segment connecting \(q_{i}\) to \(p_{i}\). Even in this case we may notice that the projective segments \(-c_i\) and \(c_{i}\) coincide but have opposite orientation. For every \(i=1,\dots,n\), we define \(\textnormal{T}_i\) the projective triangle determined by the concatenation 
\begin{equation}
    q \xrightarrow{\;\;-c_{i+1}\;\;} q_{i+1} \xrightarrow{\;\;-e_i\;\;} q_{i} \xrightarrow{\;\;c_i\;\;} q
\end{equation}

\noindent that bounds a triangle \(\textnormal{T}_i\) on its left (note the signs of the edges), see Figure \ref{fig:triangles}. We define \(\mathcal T\) to be the collection of triangles defined in this way. We notice that two consecutive triangles \(\textnormal{T}_i\) and \(\textnormal{T}_{i-1}\) always overlap along their common edge \(c_i\). Moreover, we do not require that these triangles do not overlap in their interiors.

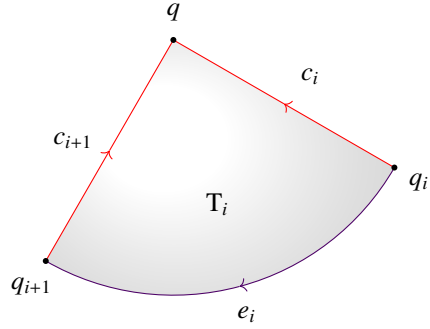
\begin{figure}[htbp]
  \centering
  \begin{tikzpicture}[scale=1.5]

    \shade[ball color = white, opacity=.25, shading angle=0] (330:2.25) arc (330:240:2.25)--(240:0)--(330:2.25);
    \draw[plum, thin, ->  ] (330:2.25) arc (330:285:2.25);
    \draw[plum, thin      ] (285:2.25) arc (285:240:2.25);
    \draw[red,  thin, ->  ] (330:2.25)--(330:1.125);
    \draw[red,  thin,     ] (330:0)--(330:1.125);
    \draw[red,  thin, ->  ] (240:2.25)--(240:1.125);
    \draw[red,  thin,     ] (240:0)--(240:1.125);
    
    \foreach \p in {(330:2.25), (240:2.25), (0:0)} {
        \fill[black] \p circle (0.75pt);
        }

    \node at (90:0.25 ) {\(q\)      };
    \node at (330:2.5 ) {\(q_{i}\)  };
    \node at (240:2.5 ) {\(q_{i+1}\)};

    \node at (285:2.5)  {\(e_{i}\)  };
    \node at (345:1.25) {\(c_{i}\)  };
    \node at (225:1.25) {\(c_{i+1}\)};
    
    \node at (285:1.5) {\(\textnormal{T}_i\)};
  \end{tikzpicture}
  \caption{A triangle \(\textnormal{T}_i\). The edges are all oriented so that they leave the triangle on their left.}
  \label{fig:triangles}
\end{figure}

\smallskip

\subsubsection{Realising a projective disc}\label{sssec:projective disc} In this subsection, we finally combine the projective polygons as defined in \S\ref{sssec:proj poly} to define the desired projective disc. More specifically, we shall glue the polygons of the collections \(\mathcal Q\) and \(\mathcal T\) together in a suitable manner. The first fundamental observation is that these polygons are now considered as \(2\)-dimensional simplexes in their own right. 

\begin{rmk}
    In fact, each polygon of the collections \(\mathcal Q\) and \(\mathcal T\) has a real projective structure with boundary made of projective segments arising from the natural embedding into the \(2\)-dimensional sphere endowed with its natural projective structure. Notice that for each polygon, its corresponding embedding serves as a developing map. By gluing two polygons from the aforementioned collections, by identifying two edges via a projective map (possibly the identity map), the respective embeddings glue together to define a developing map for the projective structure on the resulting newborn polygon. Although these embeddings are injective, their gluing no longer needs to be injective and the glued polygons may overlap once developed. This happens whenever two polygons, realised as in \S\ref{sssec:proj poly}, overlap in their interior. Compare Figures \ref{fig:upper hemi} and \ref{fig:non overlap}. 
\end{rmk}

\begin{figure}[htbp]
  \centering
  \begin{tikzpicture}[scale=2]
    \pattern [pattern=north west lines, pattern color=plum, opacity=0.25] (000:2.25)--(000:2.75) arc (000:360:2.75)--(000:2.25) arc (360:000:2.25);
    
    \draw[white, opacity=0.25] (0,0) circle (2.25cm);
    \shade[ball color = white, opacity=.20, shading angle=0] (0,0) circle (2.25cm);

    \draw[thick, plum] (0,0) circle (2.25cm);

    \draw[red] (090.00:1.5 )--(090.00:2.25); 
    \draw[red] (142.50:1.25)--(142.50:2.25); 
    \draw[red] (197.50:0.75)--(197.50:2.25);
    \draw[red] (287.50:0.75)--(287.50:2.25);
    \draw[red] (355.00:1.00)--(355.00:2.25);
    \draw[red] (033.75:1.25)--(033.75:2.25);

    \fill[white] (090:1.5)--(142.50:1.25)--(197.50:0.75)--(287.00:0.75)--(355.50:1)--(033.75:1.25)--(090:1.5);
    
    \draw[thick, orange] (090:1.5)--(142.50:1.25)--(197.50:0.75)--(287.00:0.75)--(355.50:1)--(033.75:1.25)--(090:1.5);

    \foreach \p in {(090.00:1.5), (142.50:1.25), (197.50:0.75), (287.00:0.75), (355.50:1), (033.75:1.25)} {
    \fill[black] \p circle (1pt);
    }

    \foreach \p in {(090.00:2.25), (142.50:2.25), (197.50:2.25),  (287.50:2.25), (355.00:2.25), (033.75:2.25)} { 
    \fill[red] \p circle (1pt);
        }

    \foreach \theta/\i in {120/1, 170/2, 232.5/3, 315/4, 15/5, 60/6} {
    \node at (\theta:1.75) {\(\textnormal Q_{\i}\)};
    }

    \foreach \theta/\i in {120/1, 170/2, 232.5/3, 315/4, 15/5, 60/6} {
    \node at (\theta:2.375) {\(e_{\i}\)};
    }

    \foreach \p/\i in {(090.00:1.25)/1, (140:1)/2, (195:0.5)/3, (287.00:0.5)/4, (355.50:0.75)/5, (035:1)/6} {
    \node at \p {\(p_{\i}^\star\)};
    }

    \foreach \p/\i in {(060.00:1.375)/6, (115:1.375)/1, (165:1)/2, (240.00:0.675)/3, (325:0.875)/4, (015:1.1875)/5} {
    \node[orange] at \p {\(s_{\i}^\star\)};
    }

    \node[orange] at (060:1  ) {\(\mathcal B\)};
    \node[plum  ] at (210:2.375) {\(\mathcal E_1\)};

    \end{tikzpicture}
    \caption{The figure shows the quadrilaterals, obtained from the polygonal chain shown in Figure \ref{fig:upper hemi}, glued together by identifying the projective segments with the same label (\textit{i.e.}, the segment \(r_i\) of \(\textnormal Q_{i-1}\) with the segment \(r_i\) of \(\textnormal Q_{i}\,\)). Treating the quadrilaterals as simplexes in their own right, the resulting figure is a topological annulus endowed with a projective structure with two boundary components, \(\mathcal E_1\) which is locally projectively equivalent to a great circle in the sphere endowed with its natural projective structure and the other is \(\mathcal B\), a chain of projective segments \(\{\,s_i^\star\,\}\) with corners at \(\{\,p_i^\star\,\}\). The developing map of such a structure, restricted to each quadrilateral, is injective. The developing map of the overall structure fails to be injective, and its image is shown in Figure \ref{fig:upper hemi}.}
    \label{fig:non overlap}
\end{figure}
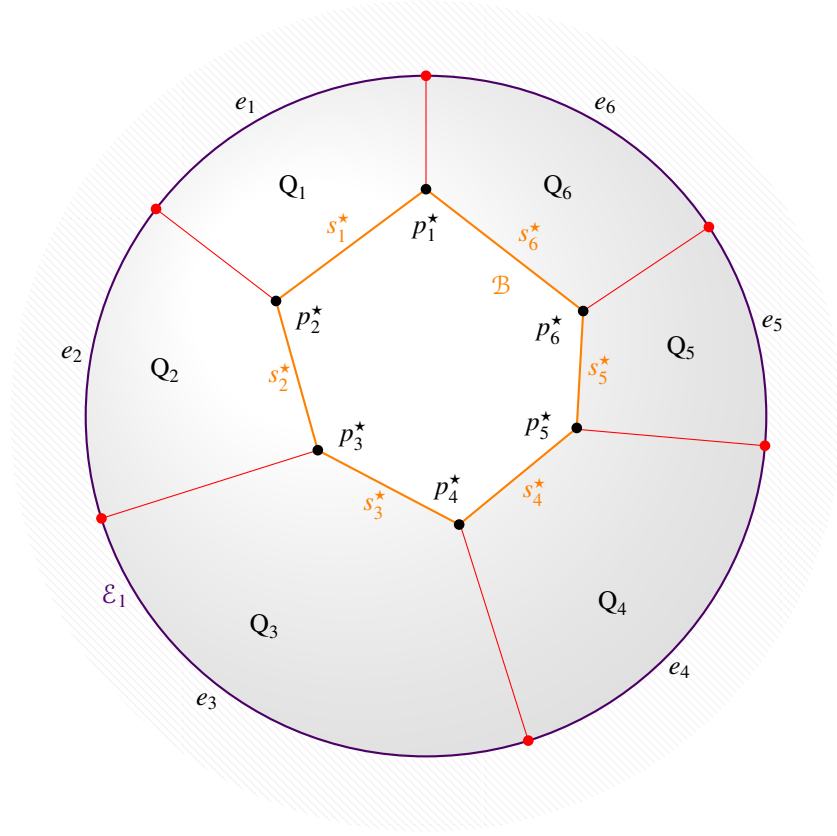

\smallskip

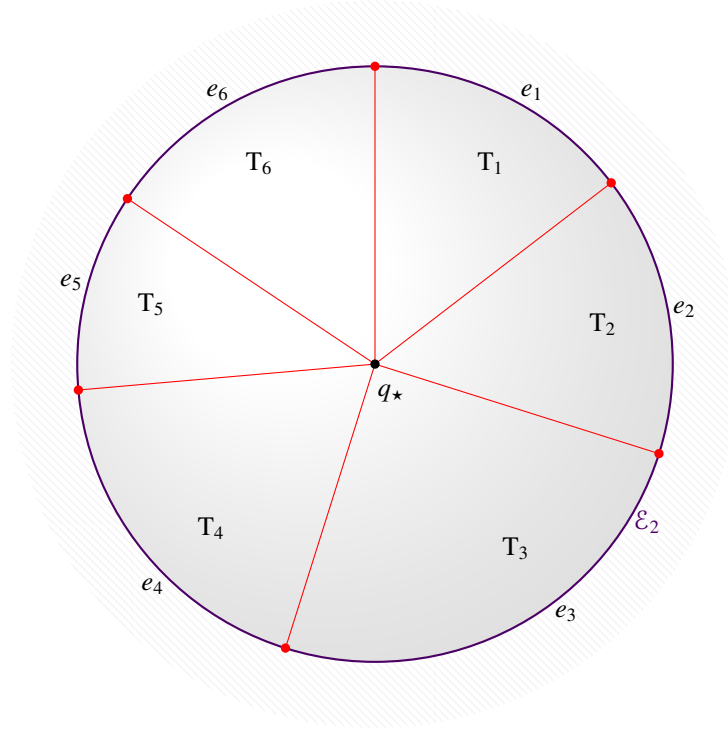
\begin{figure}[htbp]
  \centering
  \begin{tikzpicture}[scale=1.75]
    \pattern [pattern=north west lines, pattern color=plum, opacity=0.25] (000:2.25)--(000:2.75) arc (000:360:2.75)--(000:2.25) arc (360:000:2.25);
    
    \draw[white, opacity=0.25] (0,0) circle (2.25cm);
    \shade[ball color = white, opacity=.20, shading angle=0] (0,0) circle (2.25cm);

    \draw[thick, plum] (0,0) circle (2.25cm);

    \draw[red] (090.00:0) -- (090.00:2.25); 
    \draw[red] (037.50:0) -- (037.50:2.25); 
    \draw[red] (342.50:0) -- (342.50:2.25);
    \draw[red] (252.50:0) -- (252.50:2.25);
    \draw[red] (185.00:0) -- (185.00:2.25);
    \draw[red] (146.25:0) -- (146.25:2.25);

    \fill[black] (0,0) circle (1pt);

    \foreach \p in {(090.00:2.25), (037.50:2.25), (342.50:2.25), (252.50:2.25), (185.00:2.25), (146.25:2.25)} { 
    \fill[red] \p circle (1pt);
    }

    \foreach \theta/\i in {60/1, 10/2, 307.5/3, 225/4, 165/5, 120/6} {
    \node at (\theta:1.75) {\(\textnormal{T}_{\i}\)};
    }

    \foreach \theta/\i in {60/1, 10/2, 307.5/3, 225/4, 165/5, 120/6} {
    \node at (\theta:2.375) {\(e_{\i}\)};
    }

    \node[plum  ] at (330:2.375) {\(\mathcal E_2\)};
    \node[black ] at (300:0.25  ) {\(q_\star\)};

    \end{tikzpicture}
    \caption{The figure shows the triangles \(\textnormal T_i\) glued together by identifying the projective segments with the same label (\textit{i.e.}, the segment \(c_i\) of \(\textnormal T_{i-1}\) with the segment \(r_i\) of \(\textnormal T_{i}\,\)). Treating the triangles as simplexes in their own right, the resulting figure is a topological disc endowed with a projective structure with one boundary component locally projectively equivalent to the equator \(\mathcal E\). The developing map of such a structure, restricted to each triangle, is injective. The developing map of the overall structure fails to be injective.}
    \label{fig:non overlap 2}
\end{figure}

\noindent Having made these preliminary observations, we glue the polygons of the collections \(\mathcal Q\) and \(\mathcal T\) according to the following rules:
\begin{itemize}
    \item[1.] Two quadrilaterals \(\textnormal{Q}_i\) and \(\textnormal{Q}_j\) are adjacent if \(j = i + 1 \pmod n\). In this case, we glue together the edges bearing the same label \(r_{i+1}\). Once all such edges are glued, the resulting space is an annulus, say \(A_{\mathcal Q}\), endowed with a real projective structure with boundary such that one of its boundary components, say \(\mathcal E_1\) is locally projective equivalent to to a great circle in \(\sph^2\) and the other, say \(\mathcal B\) is a concatenation of \(n\) projective edges that projects to the polygonal chain \(\mathcal P\). We denote the set of corners of \(\mathcal B\) as \(\{p_i^\star\}\) and we denote by \(s_i^\star\) the geodesic segment connecting \(p_i^\star\) with \(p_{i-1}^\star\), where indexes are taken modulo \(n\). Recall the edges are oriented to leave the disc on their left according to our standing assumption. See Figure \ref{fig:non overlap}.
    \smallskip
    \item[2.] Two triangles \(\textnormal{T}_i\) and \(\textnormal{T}_j\) are adjacent if \(j \equiv i + 1 \pmod n\). In this case, we glue together the edges bearing the same label \(c_{i+1}\). We may note that, prior to gluing, every triangle in the collection \(\mathcal T\) shares a vertex at \(q \in \mathcal L\). Once all such edges are identified, these vertices merge into a single point, which we denote by \(q_\star\). The resulting space is a topological disc, say \(D_{\mathcal T}\), endowed with a real projective structure with a branch point at \(q_\star\) of some order \(m\), the precise value of which we shall establish later, see Claim~\ref{lem:geometry around q}, and whose boundary component \(\mathcal E_2\) is locally projectively equivalent to the great circle in \(\mathcal E\) in \(\sph^2\). See Figure \ref{fig:non overlap 2}.
    \smallskip
    \item[3.] Finally, a quadrilateral \(\textnormal{Q}_i\) is glued to the triangle \(\textnormal{T}_i\) by identifying the edges with label \(e_i\). By observing that both boundary components \(\mathcal E_1 \subset A_{\mathcal Q}\) and \(\mathcal E_2 \subset D_{\mathcal T}\) are partitioned into projective segments \(\{e_i\}\), in this final step we identify the boundary components of \(A_{\mathcal Q}\) and \(D_{\mathcal T}\). We shall denote the resulting curve by \(\mathcal E_\star\). See Figure \ref{fig:upper hemi 2}.
\end{itemize}

\smallskip

\noindent We may observe that all gluings are consistent with the orientations of the edges. The resulting space is a topological closed disc \(\mathcal C=A_{\mathcal Q}\,\cup\,D_{\mathcal T}\) endowed with a real projective structure, say \(\sigma_{\textnormal{core}}\), with piecewise geodesic boundary. By construction, its developing map restricts to an embedding on every polygon of either collections \(\mathcal Q\) and \(\mathcal T\) and the boundary develops onto the chain as defined in Equation \eqref{eq:chain} with \(n\) corner points as desired. It remains to deal with the geometry around the point \(q\). In what follows we shall define \(\mathcal C\) as the \textit{core disc}. In order to conclude the proof of Proposition \ref{prop:coredisc}, we shall need some technical claims.

\smallskip

\begin{figure}[htbp]
  \centering
  \begin{tikzpicture}[scale=1.125]
    \pattern [pattern=north west lines, pattern color=plum, opacity=0.25] (000:2.25)--(000:2.75) arc (000:360:2.75)--(000:2.25) arc (360:000:2.25);
    
    \draw[white, opacity=0.25] (0,0) circle (2.25cm);
    \draw[black, ->] (090:3) arc (090:120:3);
    \shade[ball color = white, opacity=.20, shading angle=0] (0,0) circle (2.25cm);

    \draw[thick, plum] (0,0) circle (2.25cm);

    \draw[red] (090.00:1.5 )--(090.00:2.25); 
    \draw[red] (142.50:1.25)--(142.50:2.25); 
    \draw[red] (197.50:0.75)--(197.50:2.25);
    \draw[red] (287.50:0.75)--(287.50:2.25);
    \draw[red] (355.00:1.00)--(355.00:2.25);
    \draw[red] (033.75:1.25)--(033.75:2.25);

    \fill[white] (090:1.5)--(142.50:1.25)--(197.50:0.75)--(287.00:0.75)--(355.50:1)--(033.75:1.25)--(090:1.5);
    
    \draw[thick, orange] (090:1.5)--(142.50:1.25)--(197.50:0.75)--(287.00:0.75)--(355.50:1)--(033.75:1.25)--(090:1.5);

    \foreach \p in {(090.00:1.5), (142.50:1.25), (197.50:0.75), (287.00:0.75), (355.50:1), (033.75:1.25)} {
    \fill[black] \p circle (1pt);
    }

    \foreach \p in {(090.00:2.25), (142.50:2.25), (197.50:2.25),  (287.50:2.25), (355.00:2.25), (033.75:2.25)} { 
    \fill[red] \p circle (1pt);
        }


    \foreach \theta/\i in {120/1, 170/2, 232.5/3, 315/4, 15/5, 60/6} {
    \node at (\theta:1.75) {\(\textnormal Q_{\i}\)};
    }

    \node at (270:2.00) {\(\mathcal A_{\mathcal Q}\)};
    \node at (270:2.50) {\(\mathcal D_{\mathcal T}\)};

    \node[plum] at (120:2.5) {\(\mathcal E_\star\)};


    \begin{scope}[shift={(6,0)}]
    \pattern [pattern=north west lines, pattern color=plum, opacity=0.25] (000:2.25)--(000:2.75) arc (000:360:2.75)--(000:2.25) arc (360:000:2.25);
    
    \draw[white, opacity=0.25] (0,0) circle (2.25cm);
    \draw[black, ->] (090:3) arc (090:060:3);
    \shade[ball color = white, opacity=.20, shading angle=0] (0,0) circle (2.25cm);

    \draw[thick, plum] (0,0) circle (2.25cm);

    \foreach \p in {(090.00:2.25), (037.50:2.25), (342.50:2.25), (252.50:2.25), (185.00:2.25), (146.25:2.25)} { 
    \fill[red] \p circle (1pt);
    }


    \foreach \p in {(090.00:2.25), (037.50:2.25), (342.50:2.25), (252.50:2.25), (185.00:2.25), (146.25:2.25)} { 
    \draw[red] \p--(0:0);
    }

    \fill[black] (0:0) circle (1pt);
    \node at (300:0.25) {\(q_\star\)};

    \foreach \theta/\i in {120/6, 170/5, 232.5/4, 315/3, 15/2, 60/1} {
    \node at (\theta:1.75) {\(\textnormal T_{\i}\)};
    }

    \node at (270:2.50) {\(\mathcal A_{\mathcal Q}\)};
    \node at (270:2.00) {\(\mathcal D_{\mathcal T}\)};

    \node[plum] at (120:2.5) {\(\mathcal E_\star\)};
    \end{scope}

  \end{tikzpicture}
  \caption{The figure shows on the left the annulus obtained by gluing the quadrilaterals of the family \(\mathcal Q\), and on the right the disc obtained by gluing the triangles of the family \(\mathcal T\). We can observe that these spaces satisfy the listed properties. In particular, one of the boundary components of the annulus and the boundary of the disc are locally projectively equivalent to a great circle in \(\sph^2\) and have opposite orientations, both bounding the space to their left. Furthermore, we can observe that these boundary components are subdivided into arcs that are identified in pairs.}
  \label{fig:upper hemi 2}
\end{figure}
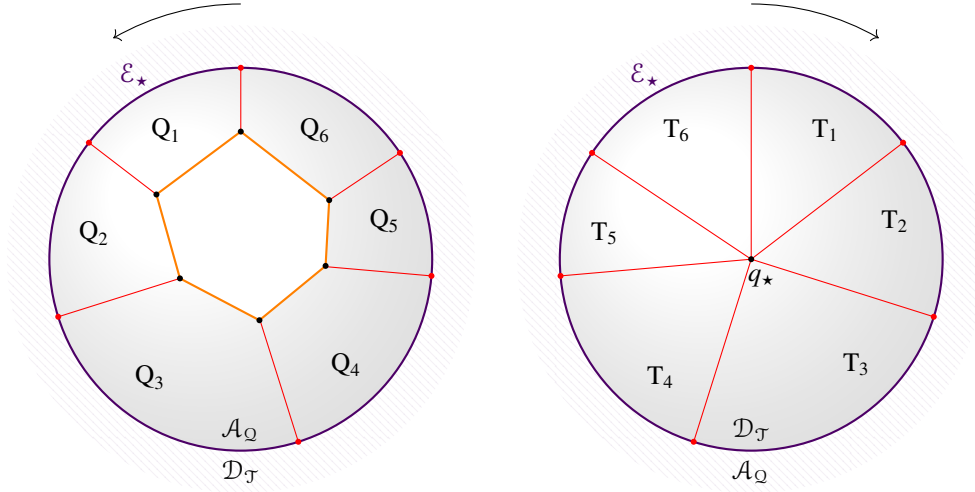

\smallskip

\subsubsection{Some technical claims}\label{sssec: technical claims} We begin with the following claim, which we will use several times in the sequel.

\begin{claim}\label{claim:branch point}
    Let \(D\) be a closed disc with boundary \(\partial D\). Let \(q\) be any point in the interior of \(D\). Let \(\sigma\) be an unbranched real projective structure on \(D^* = D \setminus \{\,q\,\}\) with trivial holonomy and no corner points. Then, the developing map \(\textnormal{dev}_\sigma\) descends to a well-defined single-valued map, say \(f \colon D^* \longrightarrow \rp2\).
    In addition, assume that the image of \(f\) is contained in some affine chart \(\R^2\) and that \(f\) extends continuously over the point \(q\). Then \(f\) is homotopic to a branched covering map of degree \(d+1\ge1\) and the projective structure \(\sigma\) uniquely extends to a branched real projective structure on \(D\). In particular, \(q\) is a branch point of order \(d\).

\end{claim}

\begin{proofclaim}
    By assumption, the holonomy \(\rho \colon \pi_1(D^*) \longrightarrow \pslr3\) of \(\sigma\) on \(D^*\) is the trivial representation. As a consequence, the developing map \(\textnormal{dev}_\sigma \colon U \longrightarrow \rp2\) is invariant under the action of \(\pi_1(\,D^*\,) \cong \mathbb{Z}\) and directly descends to a well-defined, single-valued real analytic map, say \(f \colon D^* \longrightarrow \rp2\). Since \(\sigma\) has no branch points, \(f\) is a local diffeomorphism. 
    For the second part, suppose \(f\) extends continuously over \(q\), thus mapping \(D\) into an affine chart \(\R^2\subset\rp2\). We now invoke Stoilow's Theorem to analyse the local behaviour of \(f\) around the point \(q\). Recall that a continuous map between surfaces is called light if the preimage of any point is totally disconnected (or discrete, in the case of isolated points). In our setting, since \(\sigma\) is an unbranched projective structure on \(D^*\), the map \(f\) is a local diffeomorphism on \(D^*\), hence it is automatically open and light. Since \(q\) is an isolated point and \(f\) extends continuously over it, the extended map \(f \colon D \to \mathbb{R}^2\) remains continuous, open, and light. We now invoke Stoilow's Theorem, see \cite{stoilow1928}, that states that for each point \(x \in D\), there exist a positive integer \(d\), an open neighbourhood \(U \subset D\) of \(x\), and two homeomorphisms, say \(\psi \colon U \longrightarrow \mathbb{D}\) and \(\phi \colon f(U) \longrightarrow \mathbb{D}\) onto the standard unit disc \(\mathbb{D}\) such that \(\psi(\,x\,) = 0\) and \(\phi(\,f(\,x\,)\,) = 0\), and the following diagram commutes:
    \begin{equation}
        \begin{CD}
            U @>{f|_U}>> f(U) \\
            @V{\psi}VV @VV{\phi}V \\
            \mathbb{D} @>{z\, \mapsto\, z^{d}}>> \mathbb{D}
        \end{CD}
    \end{equation}
    By choosing \(x = q\), the commutativity of the diagram implies that \(f|_U = \phi^{-1} \circ (z \mapsto z^d) \circ \psi\). In our setting, we may assume \(U=D\) and since the target is contained in an affine chart \(\mathbb{R}^2\), we can continuously deform the homeomorphisms \(\psi\) and \(\phi\) to the identity by using the Alexander Trick, see \cite{farb2011primer}. This yields a homotopy connecting \(f\) to a standard branched covering map of degree \(d\), establishing that \(q\) is a branch point of order \(d-1\) as desired.
\end{proofclaim}

\begin{claim}\label{lem:geometry around q}
   Let \(\textnormal{Ind}(\mathcal P; q)\) denote the winding number of the chain \(\mathcal P \subset \sph^2\) with respect to \(q\). Then \(q_\star\) is a branch point for \(\sigma_{\textnormal{core}}\) and its order is equal to
    \begin{equation}
    m + 1 = \textnormal{Ind}(\mathcal P; q).
    \end{equation}
\end{claim}

  \begin{proofclaim}
      We begin by observing that \(q_\star\) is a branch point as a consequence of Claim \ref{claim:branch point}. Alternatively, 
      we may observe that the disc \(D_{\mathcal T}\) comes with a degree-\(d\) covering map over \(\mathcal L \cup \mathcal E\), branched at \(q\), where \(d\) is the smallest positive integer for which the triangles unfold, by design. Therefore, \(q_\star\) is a branch point of order equal to \(m = d - 1\). As a consequence, we deduce that the winding number of the \(\mathcal E_\star\) with respect to the point \(q\) is equal to \(d\). On the other hand, this also coincides with the winding number of the piecewise geodesic boundary \(\mathcal B\) developed onto the polygonal chain \eqref{eq:chain}. Hence, the result follows.
  \end{proofclaim}

\noindent This latter Claim concludes the proof of Proposition \ref{prop:coredisc}.

\smallskip

\subsection{Projective handles}\label{ssec:handle realisation} We next focus on the realisation of projective handles whose holonomy representations match the prescribed algebraic data. As a consequence of our reduction, as outlined in \S\ref{ssec:geom process initialisation}, in the present section we aim to realise a representation \(\rho\colon\pi_1(\,S_{1,1}\,)\longrightarrow \pslr3\) as the holonomy of some branched projective structure on \(S_{1,1}\). More precisely, in the present subsection we aim to prove the following statements.

\begin{prop}\label{prop:representations of punctured tori non abel}
    Every non-abelian representation \(\rho\colon\pi_1(\,S_{1,1},b\,)\longrightarrow \pslr3\) arises as the holonomy of some branched real projective structure on \(S_{1,1}\) with at most one branch point in the interior and piecewise geodesic boundary with at most one corner point at \(b\).
\end{prop}

\noindent \noindent In the case of abelian representations, however, the situation exhibits specific subtleties that lead to a slightly different formulation.

\begin{prop}\label{prop:representations of punctured tori abel}
    Every abelian representation \(\rho\colon\pi_1(\,S_{1,1}\,)\longrightarrow \pslr3\) arises as the holonomy of some branched real projective structure on \(S_{1,1}\). Moreover, if \(\rho\) is non-trivial, there exists a branched projective structure with a single branch point in its interior.
\end{prop}

\noindent The argument to prove the main statements of this section is mostly subsumed in \S\ref{ssec:base surface}. More specifically, given a non-trivial representation \(\rho\), we shall define an auxiliary representation \(\chi\colon\pi_1(\,S_{0,5}\,)\longrightarrow \pslr 3\) and then use this latter to apply Proposition~\ref{prop:coredisc}. As a result, we shall have a topological polygon equipped with a real projective structure having one branch point in the interior. As we shall see, such a polygon is either a quadrilateral or a pentagon depending on whether the representation is abelian or not. Finally, we will appropriately identify two pairs of edges to obtain a punctured torus with the desired holonomy. The case of the trivial representation requires a special treatment because Proposition \ref{prop:coredisc} does not apply in this case, see \S\ref{sssec:trivial projective handles}. In what follows, we shall adopt the following notation. Let \(\{\,\alpha,\beta\,\}\) be a set of handle generators for \(\pi_1(\,S_{1,1}\,)\) and let \(g=\rho(\,\alpha\,)\) and \(h=\rho(\,\beta\,)\) be their respective images via \(\rho\). Finally, let \(c=\rho(\,[\alpha,\beta]\,)\) be their commutator. Up to replacing \(\{\,\alpha,\beta\,\}\) with another set of handle generators if necessary, we may assume that both \(g\) and \(h\) are non-trivial.

\smallskip

\subsubsection{Generic handles}\label{sssec:generic handle} We begin by considering the case of non-trivial representations. We stress that the following argument holds both for abelian and non-abelian representations. Before introducing the auxiliary representation, we first provide a different presentation of \(\pi_1(\,S_{1,1}\,)\). Let \(\{\,\alpha,\beta\,\}\) be a set of handle generators just introduced, set 
\begin{equation}
    \delta_1=\alpha\beta\alpha^{-1}\qquad\delta_2=\alpha\qquad\delta_3=\beta^{-1}\qquad\delta_4=\beta\alpha^{-1}\beta^{-1}\qquad\delta_5=\big[\,\beta,\,\alpha\,\big]=\beta\alpha\beta^{-1}\alpha^{-1},
\end{equation}

\noindent and consider the following presentation
\begin{equation}\label{eq:new presentation}
    \pi_1(\,S_{1,1}\,)\,=\,\Biggl\langle\,\delta_1,\,\delta_2,\,\delta_3,\,\delta_4,\,\delta_5\,\,\bigg|\,\,
    \delta_1\delta_2\delta_3\delta_4\delta_5=1, \,\,
    \delta_1\delta_2\delta_3\delta_2^{-1}=1, \,\,
    \delta_2\delta_3\delta_4\delta_3^{-1}=1 \,\, \Biggr\rangle.
\end{equation}

\smallskip

\noindent We next introduce the auxiliary representation \(\chi\colon\pi_1(\,S_{0,5}\,)\longrightarrow \pslr 3\) as follows by specifying the images of the generators of this new presentation as follows. Let \(\pi_1(\,S_{0,5}\,) = \big\langle \gamma_1, \dots, \gamma_5 \;|\; \gamma_1 \gamma_2 \cdots \gamma_5 = 1 \big\rangle\), and define a homomorphism \(\phi\colon\pi_1(\,S_{0,5}\,)\longrightarrow \pi_1(\,S_{1,1}\,)\) by setting \(\phi(\,\gamma_i\,)=\delta_i\) and hence \(\chi\) as its post-composition with the representation \(\rho\). According to our new presentation as in \eqref{eq:new presentation}, we may easily observe that \(\phi \) fails to be injective, with its kernel given by 
\begin{equation}
    \textnormal{ker}(\,\phi\,)=\big\langle \,\,\gamma_1\gamma_2\gamma_3\gamma_2^{-1}, \,\,
    \gamma_2\gamma_3\gamma_4\gamma_3^{-1}\,\,\big\rangle,
\end{equation}
determined by the second and the third relation of the presentation in \eqref{eq:new presentation}. With these reductions, the desired result readily follows for non-trivial representations. Consider the representation \(\chi\) just defined. Since \(\chi\) is non-trivial, Proposition~\ref{prop:coredisc} applies, and hence, upon choosing a judicious base point \(r\in\rp2\), we can use \(\chi\) to realise a topological disc, say \(\mathcal D\), endowed with a branched projective structure with at most a single branch point in its interior and with a piecewise projective geodesic boundary. For future purposes, see \S\ref{sssec:assembly generic}, we stress that the boundary develops on the chain of segments
\begin{equation}\label{eq:pentagonal chain}
    r \longmapsto \chi(\,\gamma_5\,)(\,r\,)\longmapsto \chi(\,\gamma_4\gamma_5\,)(\,r\,)\longmapsto \chi(\,\gamma_3\gamma_4\gamma_5\,)(\,r\,)\longmapsto \chi(\,\gamma_2\gamma_3\gamma_4\gamma_5\,)(\,r\,) \longmapsto r,
\end{equation}
\noindent (compare with the chain \eqref{eq:chain} by setting \(r_1=r\)). By replacing \(\chi(\,\cdot\,)\) with their images, the chain can be rewritten in a more pleasant form as
\begin{equation}\label{eq:pentagonal chain 2}
    r \longmapsto [\,h,g\,](\,r\,)\longmapsto g^{-1}(\,r\,)\longmapsto h^{-1}g^{-1}(\,r\,)\longmapsto gh^{-1}g^{-1}(\,r\,) \longmapsto r,
\end{equation}
\noindent see Figure \ref{fig:pentagon}.

\smallskip

\noindent Note that \(\mathcal D\) is either a quadrilateral or a pentagon depending on whether \(\chi(\,\gamma_5\,)\) is trivial or not, and hence on whether \(\rho\) is abelian. In order to realise the desired handle, it remains to identify the edges in pairs. These identifications are determined by the generators of the kernel of the homomorphism \(\phi\). More specifically, by identifying the edges as shown in Figures~\ref{fig:pentagon} and \ref{fig:quadrilateral torus realisation}, the resulting space is a genus-one surface, and exactly one of the following holds by design. Either:

\smallskip
\begin{itemize}
    \item[\textbf{NA}] In the non-abelian case, the resulting structure has a single branch point in its interior and a piecewise geodesic boundary with a corner point on it. In particular, it has holonomy \(\rho\) as desired.\\
    \smallskip
    \item[\textbf{A}] In the abelian case, instead, the resulting structure is a torus with a branch point in its interior and a special point arising from the identification of the vertices. As a direct application of Claim \ref{claim:branch point}, this is also a branch a point (possibly regular). Regardless of the nature of this point, by deleting it, we have the desired structure with holonomy \(\rho\). We observe that this case establishes Theorem \ref{thm:geometrisation} for surfaces of genus one.
\end{itemize}

\smallskip

\noindent In both cases, we can choose the base point of the construction in such a way that it is not fixed by any non-trivial element of the holonomy representation. This completes the proof of Proposition~\ref{prop:representations of punctured tori non abel} and Proposition~\ref{prop:representations of punctured tori abel} for non-trivial representations.

\begin{figure}[htbp]
  \centering
  \begin{tikzpicture}[scale=1.75]
    \pattern [pattern=north west lines, pattern color=plum, opacity=0.25] (000:2.25)--(000:2.75) arc (000:360:2.75)--(000:2.25) arc (360:000:2.25);
    
    \draw[white, opacity=0.25] (0,0) circle (2.25cm);
    \draw[black, <-] (090:3) arc (090:120:3);
    \shade[ball color = white, opacity=.20, shading angle=0] (0,0) circle (2.25cm);

    \draw[thick, plum] (0,0) circle (2.25cm);

    \fill[white] (036:1.5)--(108:1.25)--(180:1.5)--(252:1.25)--(324:1.25)--(036:1.5);
    
    \draw[thick, orange] (036:1.5)--(108:1.25)--(180:1.5)--(252:1.25)--(324:1.25)--(036:1.5);

    \foreach \p in {(036:1.5), (108:1.25), (180:1.5), (252:1.25), (324:1.25)} {
    \fill[black] \p circle (0.75pt);
    }

    \node at (320:1.35 ) {\(r\)};
    \node at (255:1.45 ) {\([\,h,g\,]\,r\)};
    \node at (180:1.785) {\(g^{-1}\,r\)};
    \node at (100:1.375) {\(h^{-1}g^{-1}\,r\)};
    \node at (042:1.65 ) {\(gh^{-1}g^{-1}\,r\)};

    \begin{scope}[rotate=-10]
        \draw[red , <-] (075:0.85) to[bend left] (235:0.85);
        \draw[blue, ->] (000:0.85) to[bend left] (165:0.85);
    \end{scope}

    \node at (135:0.725 ) {\(g^{-1}\)};
    \node at (080:0.725 ) {\(h^{-1}\)};

  \end{tikzpicture}
  \caption{The picture shows the gluing to get the desired handle in the non-abelian case. Notice that the identifications provided by \(g^{-1}\) and \(h^{-1}\) respectively are determined by the generators \(\gamma_1\gamma_2\gamma_3\gamma_2^{-1}\) and \(\gamma_2\gamma_3\gamma_4\gamma_3^{-1}\) of the kernel of \(\phi\) respectively.}
  \label{fig:pentagon}
\end{figure}
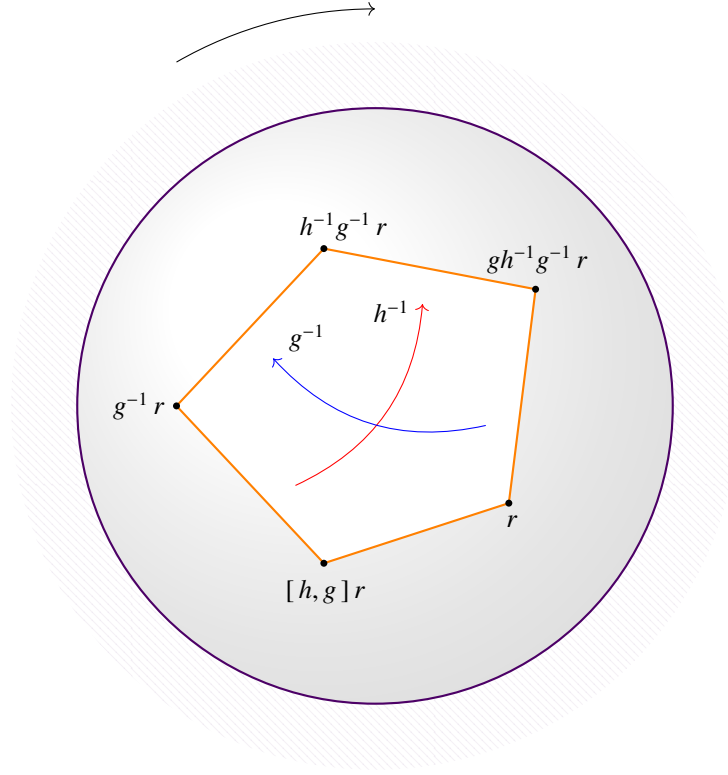

\smallskip

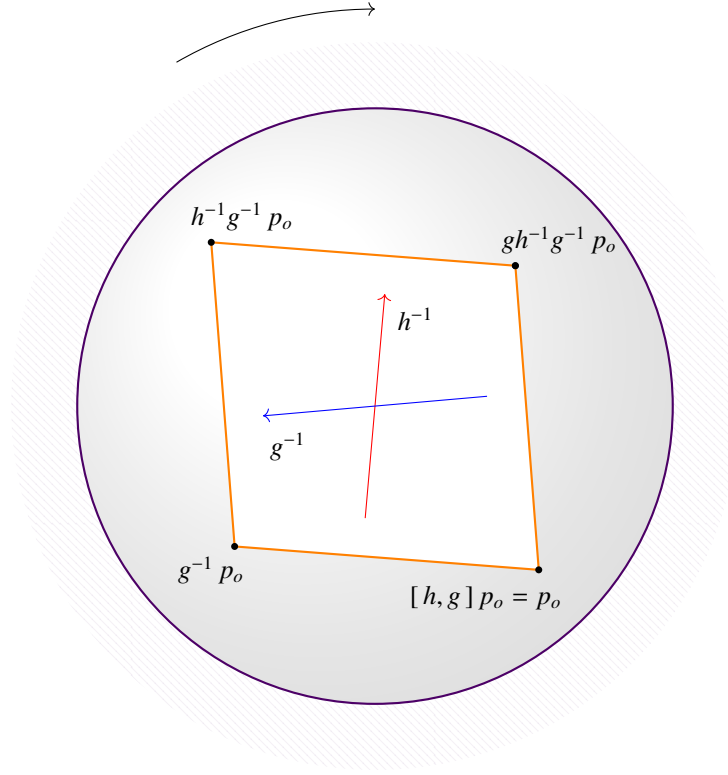
\begin{figure}[htbp]
  \centering
  \begin{tikzpicture}[scale=1.75]
    \pattern [pattern=north west lines, pattern color=plum, opacity=0.25] (000:2.25)--(000:2.75) arc (000:360:2.75)--(000:2.25) arc (360:000:2.25);
    
    \draw[white, opacity=0.25] (0,0) circle (2.25cm);
    \draw[black, <-] (090:3) arc (090:120:3);
    \shade[ball color = white, opacity=.20, shading angle=0] (0,0) circle (2.25cm);

    \draw[thick, plum] (0,0) circle (2.25cm);

    \fill[white] (045:1.5)--(135:1.75)--(225:1.5)--(315:1.75)--(045:1.5);
    
    \draw[thick, orange] (045:1.5)--(135:1.75)--(225:1.5)--(315:1.75)--(045:1.5);

    \foreach \p in {(045:1.5), (135:1.75), (225:1.5), (315:1.75), (045:1.5)} {
    \fill[black] \p circle (0.75pt);
    }

    \node at (300:1.675) {\([\,h,g\,]\,p_o=p_o\)};
    \node at (225:1.75 ) {\(g^{-1}\,p_o\)};
    \node at (125:1.750) {\(h^{-1}g^{-1}\,p_o\)};
    \node at (042:1.875) {\(gh^{-1}g^{-1}\,p_o\)};

    \begin{scope}[rotate=0]
        \draw[red , <-] (085:0.85) -- (265:0.85);
        \draw[blue, ->] (005:0.85) -- (185:0.85);
    \end{scope}

    \node at (205:0.725 ) {\(g^{-1}\)};
    \node at (065:0.725 ) {\(h^{-1}\)};

  \end{tikzpicture}
  \caption{The picture shows the gluing to get the desired handle in the abelian case. As in Figure \ref{fig:pentagon}, the identifications provided by \(g^{-1}\) and \(h^{-1}\) respectively are determined by the generators \(\gamma_1\gamma_2\gamma_3\gamma_2^{-1}\) and \(\gamma_2\gamma_3\gamma_4\gamma_3^{-1}\) of the kernel of \(\phi\) respectively.}
  \label{fig:quadrilateral torus realisation}
\end{figure}

\smallskip

\subsubsection{Projective handles with trivial holonomy}\label{sssec:trivial projective handles}

It remains to deal with the case of handles with trivial holonomy. As already hinted above, we cannot adopt the same approach since Proposition~\ref{prop:coredisc} does not apply. We therefore propose an alternative construction, previously used in \cite{faraco2021}. For this purpose, consider the sphere \(\sph^2\) equipped with its projective structure \(\sigma_{\textnormal{std}}\)
. Recall that a pair of points, say \(\{\,p^\pm\,\}\), are antipodal if \(\pi_{\textnormal{std}}(\,p^+\,)=\pi_{\textnormal{std}}(\,p^-\,)\). Let \(a\) be any arc joining them. Notice that such an arc is half of a great circle passing through these points. For our convenience, we shall refer to the arc \(a\) just defined as a \textit{meridian}. Finally, let \(f\colon\sph^2\to\sph^2\) be a branch covering map of spheres of degree \(\deg(\,f\,)\ge2\) with two ramification points and ramification values at \(\{\,p^\pm\,\}\). By design, the meridian lifts to a collection of \(\deg(\,f\,)\) arcs joining the preimages of \(\{\,p^\pm\,\}\) such that any two adjacent arcs bound a copy of the sphere. Let \(\sigma_f\) be the unique branched projective structure that makes the covering map \(f\) locally projective, that is \(\sigma_f=f^*\sigma_{\textnormal{std}}\). Notice that its developing map, say \(\textnormal{dev}_{\sigma_f}\), is nothing else that the composition \(\pi_{\textnormal{std}}\circ f\). We shall use the structure \((\sph^2,\sigma_f)\) as the base structure to realise the desired handle with trivial holonomy. 

\bigskip


Let \(\{\,\overline{p}^{\,\pm}\,\}\in(\sph^2,\sigma_f)\) be the preimages of \(p^\pm\). Let \(a_1\) and \(a_2\) be two distinct preimages of the meridian \(a\). We define two geodesic segments, say \(\tau_1,\tau_2\colon [0,1]\longrightarrow \sph^2\), such that the following conditions hold
\begin{itemize}
    \item[1.] \(\tau_1(\,0\,)=\tau_2(\,0\,)=\overline{p}^{\,+}\).
    \smallskip
    \item[2.] \(\tau_i\) is a segment of the meridian \(a_i\). In particular \(\tau_1\) and \(\tau_2\) do not share any point other than \(\overline{p}^{\,+}\).
    \smallskip
    \item[3.] The identity \(\textnormal{dev}_{\sigma_f}\Big(\,\tau_1\big(\,[0,t]\,\big)\,\Big)=\textnormal{dev}_{\sigma_f}\Big(\,\tau_2\big(\,[0,t]\,\big)\,\Big)\) holds for every parameter \(t\in[0,1]\).
\end{itemize}

\noindent Notice that the second condition ensures that \(\tau_1\) and \(\tau_2\) are injectively developed and overlap once developed, while the third condition ensures that these paths match at every parameter \(t\). Let \(0 < t_1 < t_2 \le1\) be two time parameters. For \(i=1,2\), let \(\delta_i\) be the subsegment of \(\tau_i\) parametrized by the interval \([t_1,t_2]\). The third condition guarantees that 
\begin{equation}
    \textnormal{dev}_{\sigma_f}\Big(\,\tau_1\big(\,[t_1,t_2]\,\big)\,\Big) = \textnormal{dev}_{\sigma_f}\Big(\,\tau_2\big(\,[t_1,t_2]\,\big)\,\Big).
\end{equation}
For \(i=1,2\), we slit along \(\delta_i\) to obtain a surface with a piecewise geodesic boundary \(\delta_i^1 \cup \delta_i^2\). We then glue the segments \(\delta_1^i\) and \(\delta_2^i\), as shown in Figure~\ref{fig:trivial_handle}. This provides the a handle with trivial holonomy along with together two newborn branch points. By puncturing the surface at the point \(\overline{p}^{\,-}\) we get the desired handle with trivial holonomy.

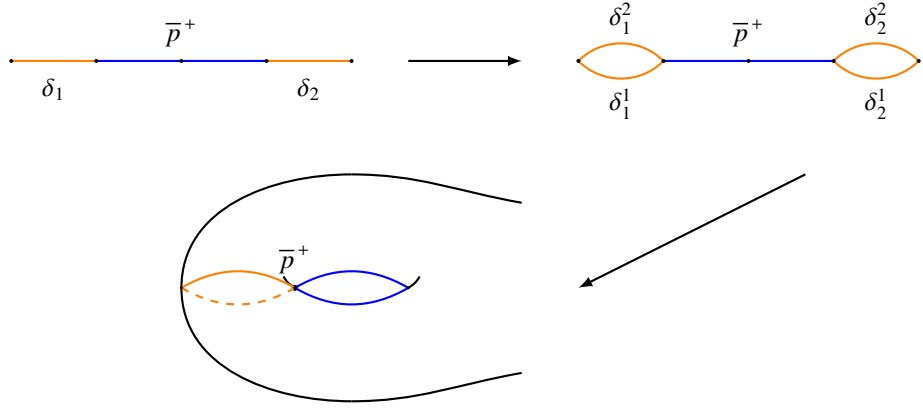
\begin{figure}[!h]
\centering
\begin{tikzpicture}[thick,scale=0.75, every node/.style={scale=1}]
    \draw [blue] (2.5,5) to (4,5);
    \draw [blue] (5.5,5) to (4,5);
    \draw [orange] (2.5,5) to (1,5);
    \draw [orange] (5.5,5) to (7,5);
    
    \fill[black] (4,5) circle (1pt);
    \fill[black] (1,5) circle (1pt);
    \fill[black] (7,5) circle (1pt);
    \fill[black] (2.5,5) circle (1pt);
    \fill[black] (5.5,5) circle (1pt);
    
    \node at (4,5.5) {\(\overline{p}^{\,+}\)};
    \node at (1.75,4.5) {\(\delta_1\)};
    \node at (6.25,4.5) {\(\delta_2\)};
    
    \draw[-latex] (8,5) to (10,5);
    \draw [blue] (12.5,5) to (14,5);
    \draw [blue] (15.5,5) to (14,5);
    \draw[orange] (12.5,5) to[out=135, in=45] (11,5);
    \draw[orange] (12.5,5) to[out=225, in=315] (11,5);
    \draw[orange] (17,5) to[out=135, in=45] (15.5,5);
    \draw[orange] (17,5) to[out=225, in=315] (15.5,5);
    
    \fill[black] (14,5) circle (1pt);
    \fill[black] (12.5,5) circle (1pt);
    \fill[black] (11,5) circle (1pt);
    \fill[black] (15.5,5) circle (1pt);
    \fill[black] (17,5) circle (1pt);

    \fill[black] (6,1) circle (1.5pt);
    
    \node at (14,5.5) {\(\overline{p}^{\,+}\)};
    \node at (11.75,4.25) {\(\delta_1^1\)};
    \node at (16.25,4.25) {\(\delta_2^1\)};
    \node at (11.75,5.75) {\(\delta_1^2\)};
    \node at (16.25,5.75) {\(\delta_2^2\)};
    
    \draw[-latex] (15,3) to (11,1);
    \draw (10,2.5) to[out=170, in=0] (7,3);
    \draw (7,3) to[out=180, in=90] (4,1);
    \draw (4,1) to[out=270, in=180] (7,-1);
    \draw (7,-1) to[out=0, in=190] (10,-0.5);
    
    \draw [blue] (6,1) to [out=330, in=210] (8,1);
    \draw (5.8,1.2) to [out=300, in=150] (6,1);
    \draw  (8,1) to [out=30, in=240] (8.2,1.2);
    \draw [blue] (6,1) to [out=30, in=150] (8,1);
    \draw [orange] (6,1) to [out=150, in=30] (4,1);
    \draw [dashed, orange] (6,1) to [out=210, in=330] (4,1);
    
    \fill[black] (6,1) circle (0.75pt);
    \fill[black] (8,1) circle (0.75pt);
    \fill[black] (4,1) circle (0.75pt);
    
    \node at (6,1.5) {\(\overline{p}^{\,+}\)};
\end{tikzpicture}
\caption{Realising a handle with trivial holonomy.}
\label{fig:trivial_handle}
\end{figure}

\smallskip

\subsection{Final assembly and realisation of prescribed holonomies}\label{ssec:final assembly} Lastly, we combine these building blocks by gluing the prescribed handles on the core surface. This completes the overall construction and yields the realisation of the target holonomy representations. 

\smallskip

\noindent For this purpose, let \(\rho\colon\pi_1(\,S,b\,)\longrightarrow\pslr3\) be any representation. We apply the initialisation process as in \S\ref{ssec:geom process initialisation} to define a collection of \(g+1\) representations, one of which is \(\rho_o\colon\pi_1(\,S_{o,g},b\,)\longrightarrow\pslr3\) and the remaining \(g\) representations are \(\rho_i\colon\pi_1(\,S_{1,1},b_i\,)\longrightarrow\pslr3\). In order to start, we need to fix a base point \(p_o\in\rp2\). According to Lemma \ref{lem:base point}, we select any point \(p_1\in \rp2\setminus \textnormal{B}(\,\rho\,)\), where \(\textnormal{B}(\,\rho\,)\) is defined as in \eqref{eq: bad base points}. We shall now discuss the assembly process by distinguish two possible different cases: at least one representation \(\rho_i\) is abelian, all \(\rho_i\)'s are not abelian. We begin with this latter case which we shall define as generic. 

\smallskip

\subsubsection{The generic case}\label{sssec:assembly generic} Assume no \(\rho_i\) is abelian. First, we consider the representation \(\rho_o\) and apply Proposition \ref{prop:coredisc} to realise a core disc \(\mathcal C\) with a branched projective structure having at most branch point in its interior and a piecewise geodesic boundary \(\{\,s_i^\star\,\}\) with \(g\) corner points \(\{\,p_i^\star\,\}\), see Figures \ref{fig:non overlap} and \ref{fig:non overlap 2}. 
For convenience, we recall that the piecewise geodesic boundary of \((\mathcal C, \sigma_{\textnormal{core}})\) develops onto the chain \eqref{eq:chain two} that we recall for our convenience.
\begin{equation}
        p_1 \xrightarrow{\;\;s_n\;\;} \rho_o(\,\gamma_n\,)(\,p_1\,) = p_n \xrightarrow{\;\;s_{n-1}\;\;} \rho_o(\,\gamma_{n-1}\,)(\,p_n\,) = p_{n-1} \xrightarrow{\;\;s_{n-2}\;\;} \dots \xrightarrow{\;\;s_1\;\;} \rho_o(\,\gamma_1\,)(\,p_2\,) = p_1
    \end{equation}

\smallskip

\noindent We now consider the individual representations \(\rho_i\). According to the standing assumption of this paragraph, none of these representations is abelian, and therefore we can apply Proposition \ref{prop:representations of punctured tori non abel}. Upon selecting a base point, one for each representation, every \(\rho_i\) arises as the holonomy of a branched projective structure with a branch point in the interior and a geodesic boundary with a corner point. More specifically, by choosing \(p_i\) as the base point, we use the representation \(\rho_i\) to realise a pentagonal domain \(\mathcal D_i\) endowed with a branched projective structure with at most a single branched points in its interior and piecewise geodesic boundary, \textit{cf.} Figure \ref{fig:pentagon}. As a consequence of our choices, the edge corresponding to the commutator develops onto the segment \(-s_i\) (notice the orientation) by design. Therefore, the pentagonal domain \(\mathcal D_i\) can be glued to the core space \(\mathcal C\). Repeating the same construction recursively \(g\) times, we obtain a topological \(4g\)-gon \(\mathcal P\), with a branched projective structure with at most \(g+1\) branch points in its interior. The edges of such a polygon are identified in pairs by design and the resulting space after identification is a closed surface of genus \(g\), see Figure \ref{fig:gluing}.

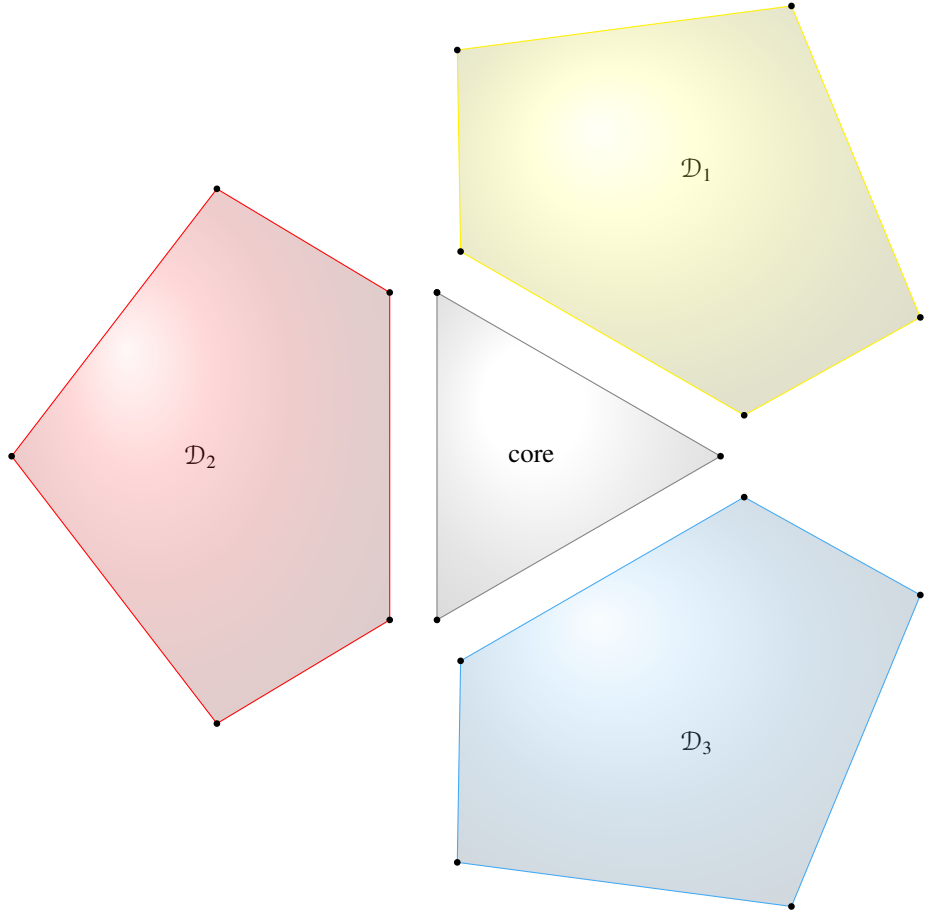
\begin{figure}[htbp]
  \centering
  \begin{tikzpicture}[scale=1.25]
    
    \begin{scope}[rotate=30]
    \draw[white, opacity=0.25] (0,0) circle (2.25cm);
    \shade[ball color = white, opacity=.20, shading angle=0] (090:2)--(210:2)--(330:2)--(090:2);
    \draw[gray] (090:2)--(210:2)--(330:2)--(090:2);

    \foreach \p in {(090:2), (210:2), (330:2), (090:2)} {
    \fill[black] \p circle (1pt);
    }

    \node at (000:0) {core};
    \node at (030:3.5) {\(\mathcal D_1\)};
    \node at (150:3.5) {\(\mathcal D_2\)};
    \node at (270:3.5) {\(\mathcal D_3\)};

    \begin{scope}[shift={(030:0.5)}]
        \shade[ball color = yellow, opacity=.20, shading angle=0] (090:2)--(330:2)--(345:4)--(030:5)--(075:4)--(090:2);
        \draw[yellow] (090:2)--(330:2)--(345:4)--(030:5)--(075:4)--(090:2);

        \foreach \p in {(090:2), (330:2), (345:4), (030:5), (075:4)} {
        \fill[black] \p circle (1pt);
        }
    \end{scope}

    \begin{scope}[rotate=120]
        \begin{scope}[shift={(030:0.5)}]
        \shade[ball color = red, opacity=.20, shading angle=0] (090:2)--(330:2)--(345:4)--(030:5)--(075:4)--(090:2);
        \draw[red] (090:2)--(330:2)--(345:4)--(030:5)--(075:4)--(090:2);

        \foreach \p in {(090:2), (330:2), (345:4), (030:5), (075:4)} {
        \fill[black] \p circle (1pt);
        }
        \end{scope}
    \end{scope}

    \begin{scope}[rotate=-120]
        \begin{scope}[shift={(030:0.5)}]
        \shade[ball color = sky, opacity=.20, shading angle=0] (090:2)--(330:2)--(345:4)--(030:5)--(075:4)--(090:2);
        \draw[sky] (090:2)--(330:2)--(345:4)--(030:5)--(075:4)--(090:2);

        \foreach \p in {(090:2), (330:2), (345:4), (030:5), (075:4)} {
        \fill[black] \p circle (1pt);
        }
        \end{scope}
    \end{scope}
    \end{scope}

    \end{tikzpicture}
    \caption{The figure shows a triangular core to which three pentagons are glued, one per handle. The resulting surface will have genus three.}
    \label{fig:gluing}
\end{figure}

To conclude that the resulting structure is branched projective, one must check the geometry around the point \(p^\star\) arising from the identification of vertices of the \(4g\)-gon. For this purpose, we first observe that \(\mathcal P\) can be used to tile a topological plane invariant under the action of \(\pi_1(\,S\,)\), where any two tiles are related by a unique element of \(\pi_1(\,S\,)\). The developing map of the core surface \(\mathcal C\) glues with the developing maps of the pentagonal domains \(\mathcal D_i\) to form an \(\rp2\)-valued function defined on \(\mathcal P\). We then extend this function to the entire tiled plane, obtaining a developing map that is equivariant with respect to the representation \(\rho\). After tiling, the vertices of \(\mathcal P\) determine a discrete subset of the plane around which the developing map is well-defined. Moreover, upon restricting the developing map to a sufficiently small closed neighbourhood, say \(D\) of any such point, the structure on \(D\) can be assumed to be unbranched (since the branch points on the surface are finite and thus form a discrete subset). As a consequence of the Claim \ref{claim:branch point}, we may conclude that \(p^\star\) is a branch point. Therefore, the topological surface is indeed endowed with a branched projective structure with holonomy \(\rho\), as desired.


\smallskip

\subsubsection{Interlude: gluing structures along slits}\label{interlude: slit constr} We briefly recall for the readers convenience the \textit{slit construction} for two surfaces \(\Sigma_1\) and \(\Sigma_2\) endowed with branched projective structures \(\sigma_1\) and \(\sigma_2\), respectively. Suppose there exist charts \(\phi_1 \colon U_1 \longrightarrow V \subset \rp2\) on \(\Sigma_1\) and \(\phi_2 \colon U_2 \longrightarrow V \subset \rp2\) on \(\Sigma_2\) mapping homeomorphically onto a common open set \(V \subset \rp2\). We select a small projective geodesic segment \(e \subset V\) and consider its preimages \(\gamma_1 = \phi_1^{-1}(\,e\,) \subset U_1\) and \(\gamma_2 = \phi_2^{-1}(\,e\,) \subset U_2\), which are geodesic segments in their respective surfaces. By slitting both surfaces along \(e_1\) and \(e_2\), and gluing the opposite banks of the two slits via the identification induced by the local charts, we obtain a connected sum surface, say \(S\) of genus equal to \(g_1+g_2\), where \(g_i\) is the genus of \(\Sigma_i\). Since the identification is made along projective geodesic segments using compatible local charts, the projective structures \(\sigma_1\) and \(\sigma_2\) extend onto each other. Moreover, the two pairs of endpoints of the slits are identified into two distinct points on the resulting surface: thus, they form two new simple branch points. In our framework, this construction is always well-defined and feasible because both the core structure and all the handles contain, by design, an embedded copy of \(\rp2\).

\begin{rmk}\label{rmk:slit}
    For readers not familiar with this surgery, topologically, slitting along a segment corresponds to replacing the segment with two copies of itself identified exclusively at their endpoints; the resulting surface is homeomorphic to the original surface with an open disc removed. Consequently, gluing two surfaces along projectively equivalent slits serves as the geometric analogue of the connected sum of surfaces.
\end{rmk}

\begin{figure}[htbp]
  \centering
  \begin{tikzpicture}[scale=1.325]

    \draw[black!50] (0,0) circle (2.25cm);
    \shade[ball color=white, opacity=.20] (0,0) circle (2.25cm);

    \foreach \theta/\i in {90,270} {
    \fill[white, rotate=\theta] (1.25, 0.75) to[out=300, in=60] (1.25,-0.75) to[out=115, in=270] (1.05,0) to[out=90, in=245] (1.25, 0.75);
    }

    \foreach \theta in {90,270} {
    \draw [black!50, thin, rotate=\theta] 
            (1.25, 0.75) to[out=300, in= 60] (1.25, -0.75);
    }

    \foreach \theta in {90,270} {
    \draw [black!50, thin, rotate=\theta] 
            (1.3, -0.835) to[out=120, in=240] (1.3,  0.835);
    }

    \draw[plum ]   ( 30:1.25) arc(30:-30:1.25);
    \fill[black]   ( 30:1.25)  circle (0.75pt);
    \fill[black]   (-30:1.25)  circle (0.75pt);
    \node[plum] at (270:2.5) {\((\Sigma_1,\sigma_1)\)};
    \node[plum] at (000:1.5 ) {\(e_1\)};
    
    \begin{scope}[shift={(5.5,0)}]
        \draw[blue!25] (0,0) circle (2.25cm);
        \shade[ball color = sky, opacity=.6] (0,0) circle (2.25cm);

        \foreach \theta/\i in {90,270} {
        \fill[white, rotate=\theta] (1.25, 0.75) to[out=300, in=60] (1.25,-0.75) to[out=115, in=270] (1.05,0) to[out=90, in=245] (1.25, 0.75);
        }

        \foreach \theta in {90,270} {
        \draw [blue!50, thin, rotate=\theta] 
            (1.25, 0.75) to[out=300, in= 60] (1.25, -0.75);
        }

        \foreach \theta in {90,270} {
        \draw [blue!50, thin, rotate=\theta] 
            (1.3, -0.835) to[out=120, in=240] (1.3,  0.835);
        }

        \draw[plum ] (210:1.25) arc(210:150:1.25);
        \fill[black] (210:1.25)   circle (0.75pt);
        \fill[black] (150:1.25)   circle (0.75pt);

        \node[plum] at (270:2.5) {\((\Sigma_2,\sigma_2)\)};
        \node[plum] at (180:1.5) {\(e_2\)};
    \end{scope}
  \end{tikzpicture}
  \caption{Slit construction. We may observe that bubbling arises as a special case of the slit construction. Compare with Figure \ref{fig:bubbling}.
  }
  \label{fig:slit constr}
\end{figure}
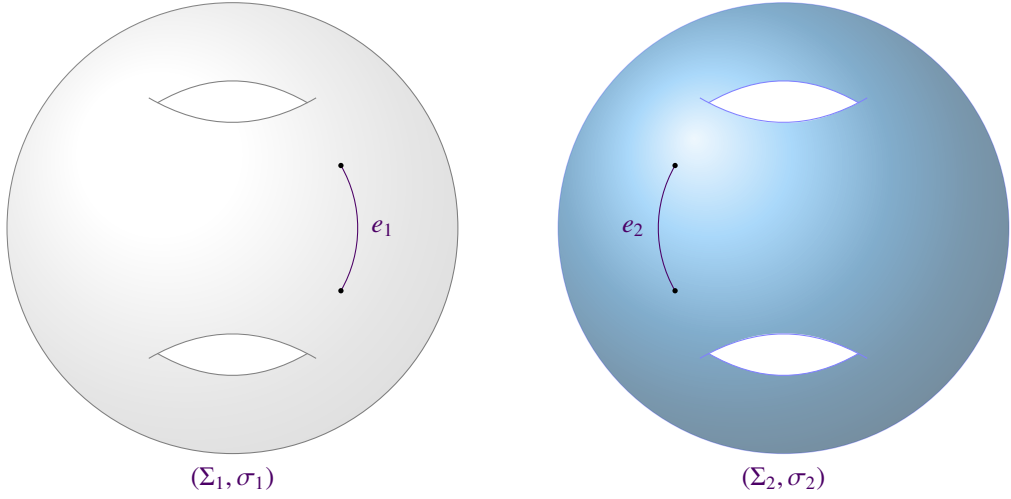

\smallskip

\subsubsection{Gluing handles with abelian holonomy}\label{sssec:abelian handles} Suppose now that there are \(0 < h \le g\) handles with abelian holonomy \(\rho_i\). Up to relabelling, without loss of generality we may assume that these are \(\rho_{g-h+1},\dots,\rho_g\). In this case, the representation \(\rho_o\) determines a branched projective structure on a polygon with \(g-h\) sides, to which we glue the \(g-h\) handles with non-abelian holonomy as in \S\ref{ssec:final assembly}. By doing so, we obtain a surface \(\Sigma\) of genus \(g-h\) endowed with a well-defined branched projective structure \(\sigma_\Sigma\). The remaining \(h\) representations determine \(h\) projective structures on \(h\) tori. We will glue these tori to the structure \((\Sigma, \sigma_\Sigma)\) by means of slit constructions, see the Interlude \S\ref{interlude: slit constr} above. In the case some of the abelian representations are trivial, we use any of the branch point already existing on \((\Sigma, \sigma_\Sigma)\) to realise a handle with trivial holonomy as shown in \S\ref{sssec:trivial projective handles}. The resulting surface will have genus \(g\) and will be endowed with a branched projective structure with holonomy \(\rho\), as desired. Finally, we note that, as a limiting case, all representations may be abelian in their own right. In this situation, the core surface to be used is a sphere equipped with an arbitrary branched projective structure (notice that this is consistent with the fact that a the polygonal chain in \eqref{eq:chain} reduces to a single point). This completes the proof in the case of abelian representations and, consequently, that of Theorem \ref{thm:geometrisation}.

\medskip

\section{Stiefel–Whitney classes of holonomy representations and realisation into strata}\label{sec:stiefel whitney geometrisable representations} 

\noindent The present section aims to establish a relation between the second Stiefel-Whitney class of an oriented rank-\(3\) vector bundle \(E\) over a closed surface \(S\) and the self-intersection of a copy of \(S\) immersed in a closed \(4\)-manifold constructed from \(E\). As a direct application, it follows that for every branched projective structure on \(S\) with branching divisor \(D\), the vanishing of the second Stiefel-Whitney class of the representation is completely determined by the parity of the degree of \(D\). The main result of the present section is the following

\begin{thm}\label{cor: rho lifts to universal cover}
    Let \(\rho\) be the holonomy of a branched \(\rp2\)-structure on \(S\) with branching chain \(D\). Then \(\rho\) lifts to the universal cover of \(\slr{3}\) if and only if \(\deg(\,D\,)\) is even.
\end{thm}

For the readers' convenience, we recall that \(\slr3\) is not simply connected. Indeed, it deformation retracts onto \(\mathrm{SO}(3,\mathbb{R})\), which is diffeomorphic to \(\mathbb{RP}^3\). Hence \( \pi_1(\,\slr3\,)\cong\pi_1(\,\mathrm{SO}(3,\mathbb{R})\,)\cong\mathbb{Z}_2\). Its universal cover is therefore a double cover of \(\slr3\). Notice that it is not a matrix group. An analogous result was obtained in the context of branched \(\mathbb{CP}^1\)-structures \cite[\S 11.2]{GalloKapovichMarden2000Monodromy}. In our setting, however, the situation is slightly more subtle, since all the objects involved are defined over the real numbers. This subsection is structured as follows: in \S\ref{ssec:proof of thm SW class}, the main theorem is proved; in the subsequent subsections, namely \S\ref{ssec:hitchinbranched} and \S\ref{ssec:prop convex BPS}, the realisation of \(\rp2\)-structures in the strata is discussed. As will be shown, Theorem~\ref{cor: rho lifts to universal cover} provides a first necessary condition for the realisation of a representation in a given stratum.



\smallskip

\subsection{Stiefel-Whitney class of holonomy representations}\label{ssec:proof of thm SW class} We begin by establishing the main result of this section. The argument is split into two subsections. In \S\ref{ssec:intersection theory}, we recall several key notions of intersection theory on \(4\)-dimensional manifolds, a treatment that remains largely independent of \(\rp2\)-structures. Next, in \S\ref{ssec:proof of lift citrerion}, we apply these preliminaries to our specific setting to derive Theorem~\ref{cor: rho lifts to universal cover}.

\smallskip

\subsubsection{Intersection theory on the projectivised bundle}\label{ssec:intersection theory}
Let \(E\) be an oriented real vector bundle of rank \(3\) over a surface \(S\). Since the rank of \(E\) exceeds the dimension of the base by one, the bundle decomposes as a direct sum \(E \cong F \oplus \underline{\mathbb{R}}\), where \(F\) is an oriented rank-\(2\) vector bundle and \(\underline{\mathbb{R}}\) denotes the trivial real line bundle. This decomposition follows since \(S\) admits a \(2\)-dimensional CW-complex structure; hence, the obstruction to finding a nowhere-vanishing section over the \(2\)-skeleton lies in the cohomology group \(H^3(S,\mathbb{Z})\cong\{\,0\,\}\), see \cite[Chapter 8, Proposition 1.1]{husemoller1966fibre}. Let \(\pi\colon\mathbb{P}(\,E\,)\longrightarrow S\) be the projectivised fibre bundle whose fibre over a point \(p\in S\) is the projective plane \(\mathbb{P}(\,E_p\,)\cong\mathbb{RP}^2.\) It admits a double covering map \(q\colon\mathbb{S}(\,E\,)\to\mathbb{P}(\,E\,)\), where \(\widehat\pi\colon\mathbb{S}(\,E\,)\longrightarrow S\) denotes the sphere bundle associated with \(E\). More precisely, \(\mathbb{S}(\,E\,)\) is the total space of the fibre bundle over \(S\) whose fibre at a point \(p\in S\) is defined as
\begin{equation}
    \mathbb{S}(\,E_p\,)\doteq\frac{\,E_p\setminus\{\,0\,\}\,}{\mathbb{R}_{>0}}\cong \mathbb{S}^2.
\end{equation}

\noindent The covering map is induced fibre-wise by identifying antipodal points in each sphere and therefore \(\mathbb{S}(\,E\,)\) is naturally a two-fold cover of \(\mathbb{P}(\,E\,)\).
\begin{prop}\label{prop: fibre_bundle_non_orientable}
    \(\mathbb{P}(\,E\,)\) is a non-orientable 4-dimensional closed manifold, with \(\mathbb{S}(\,E\,)\) being its orientable two-fold cover.
\end{prop}
\begin{proof}
We begin by showing that the manifold \(\mathbb{P}(\,E\,)\) is non-orientable. Let us consider the smooth fibre bundle \(\pi\colon\mathbb{P}(\,E\,)\to S \). By recalling that \(\pi\) is a surjective submersion, each fibre \( F_p\doteq \pi^{-1}(\,p\,) \) over \(p\in S\) is an embedded sub-manifold of \(\mathbb{P}(\,E\,)\) of dimension \(2\). We shall denote by \(j_p\colon F_p\hookrightarrow \mathbb{P}(\,E\,)\) the inclusion. Moreover, for every \(y\in F_p\), one has \( T_yF_p=\ker(\,\mathrm d_y\pi\,).\) By definition, the normal bundle \(\nu_{F_p}\) of \(F_p\) in \(\mathbb{P}(\,E\,)\) is the quotient \(j_p^*T\mathbb{P}(\,E\,)/TF_p\). Equivalently, the normal bundle fits in a short exact sequence as follows
\begin{equation}
    0\longrightarrow TF_p\longrightarrow j_p^*T\mathbb{P}(E)\longrightarrow \nu_{F_p}\longrightarrow 0.
\end{equation}
On the other hand, since the differential \(\mathrm{d}_y\pi\) is surjective for every \(y \in F_p\) and its kernel is \(T_yF_p\), we obtain an induced vector-space isomorphism \( T_y\mathbb{P}(\,E\,) / T_yF_p \cong T_pS\). As \(y\) varies smoothly in \(F_p\), these isomorphisms give rise to a smooth vector bundle isomorphism \( j_p^*T\mathbb{P}(\,E\,)/TF_p \cong F_p\times T_pS\).
Equivalently, the bundle morphism
\begin{equation}
\mathrm d\pi:j_p^*T\mathbb{P}(\,E\,)\longrightarrow F_p\times T_pS
\end{equation}
has kernel \(TF_p\), and hence \(\nu_{F_p}\cong F_p\times T_pS\). 
In other words, the normal bundle of \(F_p\) in \(\mathbb{P}(\,E\,)\) is the trivial rank-\(2\) bundle over \(F_p\) with constant fibre \(T_pS\). We now conclude by computing the first Stiefel--Whitney class. By the Whitney sum formula applied to the short exact sequence above, we have
\begin{equation}
    \begin{aligned}
    j_p^*w_1\big(T\mathbb{P}(\,E\,)\big)
                                    &=w_1\big(j_p^*T\mathbb{P}(\,E\,)\big) \\
                                    &=w_1(TF_p)+w_1(\nu_{F_p}) \\
                                    &=w_1(TF_p)+w_1(F_p\times T_pS) \\
                                    &=w_1(TF_p).
    \end{aligned}
\end{equation}
Since \(F_p\cong \rp2\), this gives \(j_p^*w_1\big(T\mathbb{P}(\,E\,)\big) = w_1(T\rp2)\neq 0\). As a consequence, the restriction of \(w_1\big(T\mathbb{P}(\,E\,)\big)\) to the fibre \(F_p\) is non-zero. In particular, \(w_1\big(T\mathbb{P}(\,E\,)\big)\neq 0\), and therefore \(\mathbb{P}(\,E\,)\) is not orientable.

\smallskip

\noindent Let us now consider the \(4\)-manifold \(\mathbb{S}(\,E\,)\). We may observe that, upon choosing a bundle metric \(h\) on \(E\), the space \(\mathbb{S}(\,E\,)\) can be identified with the unit sphere bundle of \(E\) with projection given by \(\widehat\pi:\mathbb{S}(\,E\,)\longrightarrow S\). The tangent bundle of \(\mathbb{S}(\,E\,)\) fits into the short exact sequence
\begin{equation}
    0\longrightarrow T^{\mathrm{vert}}\mathbb{S}(\,E\,)
    \longrightarrow T\mathbb{S}(\,E\,)
    \longrightarrow \widehat\pi^{\,*}TS
    \longrightarrow 0.
\end{equation}
Since \(S\) is oriented, the bundle \(\widehat\pi^{\,*}TS\) is also oriented. It remains to orient the vertical tangent bundle. A point of \(\mathbb{S}(\,E\,)\) is a pair \((p,v)\), where \(v\in E_p\) is a unit vector. The vertical tangent space at \((p,v)\) is canonically identified with \( v^{\perp_{h_p}}\subset E_p\). Since \(E_p\) is oriented and \(v\) is an oriented unit vector, the plane \(v^{\perp_{h_p}}\) inherits an orientation by the rule that a basis \((w_1,w_2)\) of \(v^\perp\) is positive if and only if \( (v,w_1,w_2) \) is a positive basis of \(E_p\). This construction depends smoothly on \((p,v)\), and therefore gives a global orientation of \(T^{\mathrm{vert}}\mathbb{S}(\,E\,)\). Thus both \(T^{\mathrm{vert}}\mathbb{S}(\,E\,)\) and \(\widehat\pi^{\,*}TS\) are oriented, and the exact sequence above induces an orientation of \(T\mathbb{S}(\,E\,)\), hence \(\mathbb{S}(\,E\,)\) is orientable. The desired result readily follows. The projection map \(q\colon\mathbb{S}(\,E\,)\longrightarrow \mathbb{P}(\,E\,)\)
is defined fibre-wise as \(v\longmapsto [\,v\,]\), where \([\,v\,]\) denotes the unoriented line spanned by \(v\). As a consequence, its fibre over a point \([\,v\,]\in\mathbb{P}(\,E\,)\) consists of the two points \(v\) and \(-v\). Hence \(q\) is a two-sheeted covering.
\end{proof}

\smallskip

\noindent Let \(M\) be an orientable closed \(4\)-manifold. Upon choosing an orientation, its integral intersection pairing is the bilinear form
\begin{equation}
    Q_M\colon H_2(M;\mathbb{Z})\times H_2(M;\mathbb{Z})\longrightarrow \mathbb{Z}
\end{equation}

\noindent defined by
\begin{equation}
    Q_M(a,b)\,\doteq\,\langle \alpha\smile\beta,\,[\,M\,]\rangle,
\end{equation}
where \(\alpha,\beta\in H^2(M;\mathbb{Z})\) are the Poincar\'e duals of \(a,b\), respectively, and \([\,M\,]\in H_4(M,\mathbb Z)\) is the fundamental class of the \(4\)-manifold. Equivalently, if the classes \(a\) and \(b\) admit transverse, oriented embedded surfaces, say \(\Sigma_1\) and \(\Sigma_2\), as representatives, then \(Q_M(a,b)\) is the algebraic intersection number of \(\Sigma_1\) and \(\Sigma_2\). Recall that if \(M\) is non-orientable, there is no canonical integral fundamental class. On the other hand, every manifold is \(\mathbb{Z}_2\)-orientable, see \cite[\S 3.3]{hatcher2005algebraic}. In particular, if \(M\) is a closed manifold then it admits a unique fundamental class in \(H_{\mathrm{top}}(M;\mathbb{Z}_2)\), and the natural intersection pairing is given by the mod-\(2\) pairing
\begin{equation}
    Q_M^{(2)}\colon H_2(M;\mathbb{Z}_2)\times H_2(M;\mathbb{Z}_2)\longrightarrow \mathbb{Z}_2, \qquad Q_M^{(2)}(a,b)\doteq\langle \alpha\smile\beta,[M]_2\rangle,
\end{equation}
where \([M]_2\in H_4(M;\mathbb{Z}_2)\) is the mod-\(2\) fundamental class and \(\alpha,\beta\in H^2(M;\mathbb{Z}_2)\) are the Poincar\'e duals of the classes \(a,b\) (see \cite[\S 1.1.1]{donaldson1987} for more details).

\smallskip

\noindent Returning to the case of \(\mathbb{P}(\,E\,)\) and \(\mathbb{S}(\,E\,)\), the presence of the trivial bundle in the decomposition of \(E\) implies the existence of a smooth section \(\sigma \colon S \to \mathbb{P}(\,E\,)\). Indeed, for any \(p \in S\), the fibre \(\underline{\mathbb{R}}_p\) naturally determines a line through the origin in the vector space \(E_p\), and hence defines a point in \(\mathbb{P}(\,E_p\,)\). Moreover, the section \(\sigma\) obtained in this way is smooth because the decomposition \(E \cong F \oplus \underline{\mathbb{R}}\) is itself smooth. Since the line defined by the section \(\sigma\) is orientable inside the \(3\)-dimensional fibre, there are two possible lifts \(\widehat{\sigma}^{\pm}\) of \(\sigma\) into \(\mathbb{S}(\,E\,)\). We define \(\Sigma \doteq \sigma(S) \subset \mathbb{P}(\,E\,)\), and denote by \(\widehat{\Sigma}^+ \doteq \widehat{\sigma}^+(S) \subset \mathbb{S}(\,E\,)\) the lift obtained from the local orientation preserving map \(\widehat{\sigma}^+\). The restriction \(q|_{\widehat{\Sigma}^+}\colon\widehat{\Sigma}^+\longrightarrow\Sigma\) being a diffeomorphism, the differential of \(q\) induces an isomorphism of normal bundles
\begin{equation}
    \nu_{\widehat{\Sigma}^+}^{\mathbb{S}(\,E\,)} \cong (q|_{\widehat{\Sigma}^+})^*\nu_{\Sigma}^{\mathbb{P}(\,E\,)}.
\end{equation}

\noindent Let \(\Sigma_\varepsilon\) be a small transverse perturbation of \(\Sigma\) in \(\mathbb{P}(\,E\,)\), obtained from a generic section of \(\nu_{\Sigma}^{\mathbb{P}(\,E\,)}\), and let \(\widehat{\Sigma}^+_\varepsilon\) be the corresponding lift near \(\widehat{\Sigma}^+\). Then the intersections \(\widehat{\Sigma}^+\cap\widehat{\Sigma}_\varepsilon^+\) are in one-to-one correspondence with the intersections \(\Sigma\cap\Sigma_\varepsilon\). Consequently, the parity of the intersection number in the target manifold \(\mathbb{P}(\,E\,)\) is equal to the reduction modulo \(2\) of the signed intersection number computed in the domain space \(\mathbb{S}(\,E\,)\). In other words the following identity holds
\begin{equation}\label{eq: relation auto-intersection between Z and Z_2}
    Q_{\mathbb{P}(\,E\,)}^{(2)}([\,\Sigma\,]_2,[\,\Sigma\,]_2) \equiv Q_{\mathbb{S}(\,E\,)}([\,\widehat{\Sigma}^+\,], [\,\widehat{\Sigma}^+\,]) \pmod 2.
\end{equation}

\smallskip

\begin{lem}\label{lem: auto intersection and 2nd Stiefel-Whitney}
    The following relation holds:
    \begin{equation}\label{eq:self_intersection_w2}
        Q_{\mathbb{P}(\,E\,)}^{(2)}([\,\Sigma\,]_2,[\,\Sigma\,]_2)=\left\langle w_2(\,E\,),[\,S\,]_2\right\rangle,
    \end{equation}
    where \(\langle\cdot,\cdot\rangle\) denotes the perfect pairing between \(H^2(S,\mathbb Z_2)\) and \(H_2(S,\Z_2)\).
\end{lem}
\begin{proof}

\noindent Recall that, for an embedded closed surface \(\Sigma\) in a closed \(4\)-manifold, its mod-\(2\) self-intersection is computed by the second Stiefel--Whitney class of its normal bundle, that is,
\begin{equation}
    Q_{\mathbb{P}(\,E\,)}^{(2)}([\,\Sigma\,]_2,\,[\,\Sigma\,]_2)\,=\, \left\langle w_2(\nu_{\Sigma}),[\,\Sigma\,]_2\right\rangle_\Sigma,
\end{equation}

\noindent where \(\nu_{\Sigma}\) denotes the normal bundle of \(\Sigma\) in \(\mathbb{P}(\,E\,)\) and \(\langle\cdot,\cdot\rangle_\Sigma\) is the perfect pairing between the second (co)homology of \(\Sigma\) with values in \(\Z_2\) (see \cite[Chapter 11]{milnor1974characteristic}). We identify the isomorphism class of the normal bundle of \(\Sigma\) in \(\mathbb{P}(\,E\,)\) as follows. If \(\ell_p\subset E_p\) denotes the line corresponding to \(\sigma(\,p\,)\), then the normal directions to \(\Sigma\) are precisely the infinitesimal variations of \(\ell_p\) inside the fixed vector space \(E_p\). Hence
\begin{equation}
(\nu_\Sigma)_{\sigma(p)}\cong T_{[\ell_p]}\mathbb{P}(\,E_p\,)\cong\operatorname{Hom}(\ell_p,E_p/\ell_p).
\end{equation}

\noindent Since the section \(\sigma\) is induced by the splitting \(E\cong F\oplus\underline{\mathbb{R}}\), we have
\begin{equation}
    \nu_\Sigma\cong\operatorname{Hom}(\underline{\mathbb{R}},E/\underline{\mathbb{R}})\cong\underline{\mathbb{R}}^*\otimes F\cong F.
\end{equation}

\noindent The Whitney product formula gives \(w_2(E)=w_2(F\oplus\underline{\mathbb{R}})=w_2(F),\) hence \(w_2(\nu_{\Sigma})=w_2(F)=w_2(E),\) where, more precisely, the class on \(\Sigma\) is the pull-back \(w_2(\nu_{\Sigma})=(\pi|_{\Sigma})^*w_2(E)\). Since \(\pi|_{\Sigma}\colon\Sigma\to S\) is a diffeomorphism, we obtain
\begin{equation}
    Q_{\mathbb{P}(\,E\,)}^{(2)}([\,\Sigma\,]_2,\,[\,\Sigma\,]_2)=\left\langle w_2(\nu_{\Sigma}),[\,\Sigma\,]_2\right\rangle_\Sigma=\left\langle w_2(E),[\,S\,]_2\right\rangle. \qedhere
\end{equation}
\end{proof}

\smallskip

\subsubsection{Into the proof of Theorem \ref{cor: rho lifts to universal cover}}\label{ssec:proof of lift citrerion} In the present section we explain how the intersection theory just recalled in \S\ref{ssec:intersection theory} applies to prove Theorem \ref{cor: rho lifts to universal cover}.
For this purpose, let \((S,\sigma)\) be a branched \(\rp2\)-structure with branching chain \(D=\sum_{i=1}^k m_i p_i\) (see Remark \ref{rmk:branchingchain}). We denote by \((\mathrm{dev},\rho)\) the developing-holonomy pair as determined in \S\ref{ssec:devhol_pair}, see Proposition \ref{prop:BPSasmaps}. Let
\begin{equation}\label{eq:flatbundle Erho}
    E_\rho\doteq\frac{\widetilde S\times\R^3}{\sim_\rho}\longrightarrow S
\end{equation}
be the oriented flat vector bundle associated with the holonomy representation. Let us consider the smooth section \(\sigma_\rho\colon S\longrightarrow\mathbb{P}(\,E_\rho\,)\) corresponding to the developing map, namely \(\sigma_\rho(p):=[\tilde p,\operatorname{dev}(\tilde p)]\), see \cite[Proposition 4.1]{alessandrini2019}). The smooth section \(\sigma_\rho\) corresponds to a rank \(1\) sub-bundle \(L_\rho\subset E_\rho\) which is orientable, hence trivial. In particular, it admits two possible lifts \(\widehat{\sigma}_\rho^\pm\) to \(\mathbb{S}(\,E_\rho\,)\). We now state the following theorem, from which the main result of this section will follow.

\begin{thm}\label{thm:equation between deg(D) e auto intersezione}
Given a branched \(\mathbb{RP}^2\)-structure on \(S\) with branching chain \(D\), the following relation holds: 
\begin{equation}
     Q_{\mathbb{S}(\,E\,)}([\,\widehat{\Sigma}^+\,],[\,\widehat{\Sigma}^+\,])=\mathrm{deg}(\,D\,)+2-2g.
\end{equation}
\end{thm}

\noindent We adapt the strategy used by Gallo--Kapovich--Marden \cite[Proposition 11.2.2]{GalloKapovichMarden2000Monodromy} to the real projective setting.

\begin{proof}
Let \(E_\rho\) be the flat vector bundle defined as in \eqref{eq:flatbundle Erho}. The flat connection on \(E_\rho\) induces a horizontal distribution on the sphere bundle \(\widehat{\pi}\colon\mathbb{S}(\,E_\rho\,)\longrightarrow S\). Thus we have a splitting
\begin{equation}
    T\mathbb{S}(\,E_\rho\,)=\mathcal H\oplus T^{\mathrm{vert}}\mathbb{S}(\,E_\rho\,),
\end{equation}
where \(T^{\mathrm{vert}}\mathbb{S}(\,E_\rho\,)=\ker(\,d\widehat{\pi}\,)\). We denote by
\(r\colon T\mathbb{S}(\,E_\rho\,)\longrightarrow T^{\mathrm{vert}}\mathbb{S}(\,E_\rho\,)\) the projection onto the vertical factor with respect to this splitting. That is, if \(W=W^{\mathcal H}+W^{\mathrm{vert}}\in\Gamma(T\mathbb{S}(\,E_\rho\,))\), then \(r(W)=W^{\mathrm{vert}}\). Along the section \(\widehat{\sigma}^{+}_{\rho}\), the vertical bundle \((\widehat{\sigma}^{+}_{\rho})^*T^{\mathrm{vert}}\mathbb{S}(\,E_\rho\,)\) is naturally identified with the normal bundle of \(\widehat{\Sigma}^{+}\) in \(\mathbb{S}(\,E_\rho\,)\). Indeed, the tangent space to the total space splits into horizontal and vertical directions, while the differential of the section projects isomorphically onto the tangent space of the base. Hence, after this identification, a vertical vector along the section represents a normal direction to \(\widehat{\Sigma}^{+}\). Now let us choose a Morse function \(f\colon S\to\mathbb{R}\)
with exactly one minimum, one maximum, and \(2g\) saddle points. Notice that we may choose \(f\) so that its critical points are disjoint from the branch points \(p_1,\dots,p_k\). Let \(X\doteq\operatorname{grad}_h(f)\) with respect to an auxiliary Riemannian metric \(h\) on \(S\). Then \(X\) has isolated non-degenerate zeros, all disjoint from the branch points, and by the Poincar\'e--Hopf theorem, see \cite{hopf1927vektorfelder, Poincare1881},
\begin{equation}
\sum_{x\in Z(\,X\,)}\operatorname{ind}_x(\,X\,)=\chi(\,S\,)=2-2g.
\end{equation}
More explicitly, the minimum and the maximum have index \(+1\), while the \(2g\) saddle points have index \(-1\). We now define a section \(N\) of the normal bundle \(\nu_{\widehat{\Sigma}^{+}}\) as follows
\begin{equation}
    N\doteq r\,\big(\,d\widehat{\sigma}^{+}_{\rho}(\,X\,)\,\big).
\end{equation}

\noindent Equivalently, \(N\) measures the vertical component, with respect to the flat connection, of the derivative of the lifted section in the direction of the vector field \(X\). We now analyse the zero set of \(N\). Recall that the developing map is a local diffeomorphism away from the branch points, see Proposition \ref{prop:BPSasmaps}. According to \cite[\S4.2]{alessandrini2019}, this is equivalent to say that the section \(\widehat{\sigma}^{+}_{\rho}\) is transverse to the horizontal distribution induced by the flat connection. On the other hand, at a branch point, say \(p_i\) of order \(m_i\), the section \(\widehat{\sigma}^{+}_{\rho}\) fails to be transverse to the horizontal distribution and the derivative of the developing map vanishes to order \(m_i\) at \(p_i\). Since \(X(\,p_i\,)\neq0\), its vertical component \(r\,\big(\,d\widehat{\sigma}^{+}_{\rho}(\,X\,)\big)\), has a zero of index \(m_i\) at \(\widehat\sigma_\rho^+(\,p_i\,)\). As a consequence, away from the branch points, \(N\) vanishes at the zeros of \(X\), that is, \(N\) vanishes at \(Z(\,X\,)\,\cup\,\{\,p_1,\dots,p_k\,\}\). In particular, the normal vector field \(N\) is not transverse to the \(0\)-section. See Figure \ref{fig:section}. 

\begin{figure}[htbp]
  \centering
  \begin{tikzpicture}[scale=2]

    \clip (-0.6, 0) rectangle (5.8, 3.8);
    
    \path[name path=bottom] (-0.5,1.6) .. controls (1.5,2.1) and (3.5,1.3) .. (5.5,1.8);
    \path[name path=top] (-0.5,3.6) .. controls (1.5,4.1) and (3.5,3.3) .. (5.5,3.8);

    \foreach \yoffset in {-0.4, -0.1, 0.2, 0.5, 0.8, 1.1, 1.4, 1.7, 2.0, 2.2} {
      \draw[green!60!black, semithick] 
        (-0.5, 1.6 + \yoffset) .. controls (1.5, 2.4 + \yoffset - 0.3) and (3.5, -0.5 + \yoffset + 0.3) .. (5.5, 2.3 + \yoffset);
    }

    \fill[white] (-1,  0  ) rectangle (6, 1.50); 
    \fill[white] (-1,  3.5) rectangle (6, 4.50);

    \draw[thick] (-0.5,0.2) .. controls (1.5,0.7) and (3.5,-0.1) .. (5.5,0.4);
    \node[right] at (5.5,0.5) {\(S\)};

    \draw[->, thick] (2.5,1.2) -- (2.5,0.5) node[midway, right] {\(\pi\)};
    
    \node[green!60!black] at (4.2, 1.3) {\(\mathcal H\)};

    \draw[red!80!black, thick] 
      (-0.5, 2.2) .. controls (1.2, 3.1) and (3.8, 1.8) .. (5.5, 3.2);
    \node[red!80!black, right] at (5.5, 3.2) {\(\widehat\sigma_\rho^+\)};

    \draw[thin, dashed] (3.55, 1)--(3.45,3.5);

    \fill[black] (3.4875,2.5125) circle (1.0pt);
    \fill[black] (3.4875,0.2325   ) circle (1.0pt);

    \node[scale=0.75] at (3.4875, 0.125) {\(p\)};
    \node[scale=0.75] at (3.75,2.3725 ) {\(\widehat\sigma_\rho^+(\,p\,)\)};

    \node[left] at (0, 3.5) {\(\mathbb S(\,E_\rho\,)\)};

    \node[scale=0.75] at (3.5, 3.625) {\(T^{\mathrm{vert}}\mathbb{S}(\,E_\rho\,)\)};

  \end{tikzpicture}
  \caption{The pictures show the fibre bundle \(\mathbb S(\,E_\rho\,)\longrightarrow S\). Green lines denote the horizontal foliation \(\mathcal H\) determined by the flat connection, while the red line denotes the section \(\widehat\sigma_\rho^+\). The latter may be tangent to the horizontal foliation at a (finite) set of points. At a tangency point, the vertical vector field \(N\) has a zero, since \(T^{\textnormal{vert}}_p\mathbb S(\,E_\rho\,)\) is the kernel of the differential of the projection.}
  \label{fig:section}
\end{figure}
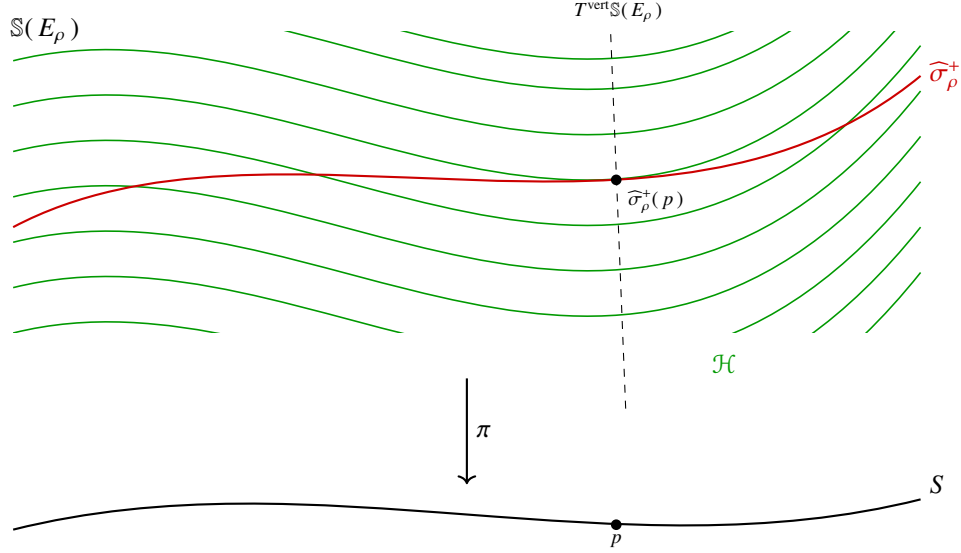

\noindent By standard results, see \cite[Chapter~3, \S~5, Exercise~7]{guillemin1974differential}, an arbitrarily small perturbation of \(N\) supported in a small neighbourhood of \(p_i\) decomposes the zero at \(p_i\) into \(m_i\) simple zeros, each with local index \(+1\). Furthermore, such a perturbation does not alter the degree of the branching data and yields a transverse section, say \(N_\varepsilon\) of the normal bundle. Since the self-intersection of \(\widehat{\Sigma}^{+}\) is the Euler number of its normal bundle, we obtain
\begin{equation}
    Q_{\mathbb{S}(\,E_\rho\,)}([\widehat{\Sigma}^{+}],[\,\widehat{\Sigma}^{+}\,])\,=\,e(\,\nu_{\widehat{\Sigma}^{+}}\,)[\,S\,]\,=\,\sum_{x\in Z(N_\varepsilon)}\operatorname{ind}_x(\,N_\varepsilon\,).
\end{equation}
\noindent By the discussion above, the zeros of \(N_\varepsilon\) consist of the zeros of \(X\), together with the branch points \(\{p_1,\dots,p_d\}\), where now each \(p_i\) is repeated \(m_i\)-times. Therefore
\begin{equation}
    \begin{aligned}
        Q_{\mathbb{S}(E_\rho)}([\widehat{\Sigma}^{+}],[\widehat{\Sigma}^{+}])
                                &= \sum_{x\in Z(X)}\operatorname{ind}_x(\,N_\varepsilon\,) \,+\,\sum_{i=1}^k \operatorname{ind}_{p_i}(N_\varepsilon) \\ 
                                &=\sum_{x\in Z(X)}\operatorname{ind}_x(\,X\,)\,+\,\sum_{i=1}^k m_i \\
                                &= \chi(\,S\,)+\deg(\,D\,) \\
                                &= 2-2g+\deg(\,D\,).
\end{aligned}
\end{equation}
This proves the desired formula.
\end{proof}

\begin{proof}[Proof of Theorem \ref{cor: rho lifts to universal cover}]
As a consequence of Theorem \ref{thm:equation between deg(D) e auto intersezione} combined with equation \eqref{eq: relation auto-intersection between Z and Z_2}, we readily obtain the identity
\begin{equation}
    Q_{\mathbb{P}(\,E_\rho\,)}^{(2)}([\,\Sigma\,]_2,[\,\Sigma\,]_2)=\deg(\,D\,) \pmod 2.
\end{equation}

\noindent Moreover, by Lemma \ref{lem: auto intersection and 2nd Stiefel-Whitney} we know that \(Q_{\mathbb{P}(\,E_\rho\,)}^{(2)}([\,\Sigma\,]_2,[\,\Sigma\,]_2)=\left\langle w_2(\,E\,),[\,S\,]_2\right\rangle\), and hence 
\begin{equation}
\left\langle w_2(\,E_\rho\,),\,[\,S\,]_2\right\rangle=\deg(\,D\,) \pmod 2.
\end{equation}
\noindent This implies that \(w_2(\,E_\rho\,)=0\) if and only if \(\deg(\,D\,)\) is even, which is equivalent to \(\rho\) lifting to the universal cover of \(\slr3\), see \cite[\S 2.4]{goldman1988}.
\end{proof}

\begin{rmk}
    In the case the \(\rp2\)-structure is unbranched, so that \(D=0\), then it was already observed by Goldman (\cite{goldman1990}) that the corresponding holonomy representation lifts to the universal cover of \(\slr3\).
\end{rmk}

\smallskip

\subsection{Realisation into strata of branched projective structures}\label{ssec:geometrisation in strata} While Theorem \ref{thm:geometrisation} established in \S\ref{sec:geometrisation} ensures that every representation is geometrisable, Theorem \ref{cor: rho lifts to universal cover} imposes necessary conditions when the branching orders are prescribed. The goal of the present section is to provide a partial answer to the realisation problem into strata.

\smallskip

\noindent In \S\ref{sssec:moduli spaces}, we defined the strata of structures as specific subspaces of the moduli space \(\mathcal{RP}^2(\,S\,)\) which comprise projective structures with prescribed singularities. It is a challenging problem to determine whether a representation can be realised in a prescribed stratum. Theorem~\ref{cor: rho lifts to universal cover} provides a first characterisation, here stated as a corollary.

\begin{cor}\label{cor: char realisation}
    If a representation \(\rho \colon \pi_1(S) \longrightarrow \pslr3\) arises as the holonomy of a branched projective structure in the stratum \(\mathcal{RP}^2(\,\mu\,)\), then its second Stiefel--Whitney class satisfies the compatibility relation
        \begin{equation}
            \textnormal{w}_2(\,\rho\,)=\langle w_2(\,E_\rho\,), [S] \rangle \equiv \deg(\,\mu\,) \pmod 2,
        \end{equation}
    where \([S] \in H_2(S, \mathbb{Z}_2)\) denotes the fundamental class of the surface.
\end{cor}


Although the general case remains to be fully addressed, as outlined in the introduction, see \S\ref{ssec:strata}, in what follows we shall provide a complete characterisation for Hitchin representations. For this purpose, we first introduce the following topological surgery.

\smallskip

\subsubsection{Bubbling}\label{sssec:bubbling} We now provide a way to alter a branched projective structures by preserving the topology of the underlying surface and the holonomy representation. Such an operation is well-known as \textit{bubbling} (see, \textit{e.g.}, \cite{calsetall2014} and \cite{GalloKapovichMarden2000Monodromy} in the context of complex projective structures). For this purpose, we consider the two-dimensional sphere \(\sphn2\) with its standard projective structure \(\sigma_{\textnormal{std}}\) that makes the (unique) covering projection \(\pi_{\textnormal{std}}\colon\sphn2\longrightarrow\rp2\) a projective map. Notice that \(\pi_{\textnormal{std}}\) is a developing map for \(\sigma_{\textnormal{std}}\). We recall for the readers' convenience that such a structure has trivial holonomy, see Remark \ref{rmk:trivial holo realisable}. 

\smallskip

\noindent The bubbling construction is a special case of the slit construction introduced in Interlude \S\ref{interlude: slit constr}, where at least one of the two surfaces is a sphere. Let \((S,\sigma_o)\) be a (possibly branched) \(\rp2\)-structure and let \((U_1,\varphi_1)\) be an open chart. Up to shrink the open set \(U\) a little, if necessary, let \((U_2,\varphi_2)\) be a local chart for \((\sphn2, \sigma_{\textnormal{std}})\) such that \(\varphi_1(\,U_1\,)=\varphi_2(\,U_2\,)=V\subset\rp2\). Let \(\gamma\subset V\) be any segment and let \(\gamma_1\) and \(\gamma_2\) be the respective preimages of \(\gamma\) in \(U_1\) and \(U_2\). Let \((S,\sigma)\) be the \(\rp2\)-structure obtained by cutting \((S,\sigma_o)\) along \(\gamma_1\) and gluing a copy of the standard projective structure \(\sigma_{\textnormal{std}}\) on \(\sphn2\) cut along \(\gamma_2\). Topologically, this is equivalent to the connected sum of \(S\) with a sphere, and hence the topology remains unaltered. Some remarks are in order, see Figure \ref{fig:bubbling}.

\begin{figure}[htbp]
  \centering
  \begin{tikzpicture}[scale=1.325]

    \draw[black!50] (0,0) circle (2.25cm);
    \shade[ball color=white, opacity=.25] (0,0) circle (2.25cm);

    \foreach \theta/\i in {90,270} {
    \fill[white, rotate=\theta] (1.25, 0.75) to[out=300, in=60] (1.25,-0.75) to[out=115, in=270] (1.05,0) to[out=90, in=245] (1.25, 0.75);
    }

    \foreach \theta in {90,270} {
    \draw [black!50, thin, rotate=\theta] 
            (1.25, 0.75) to[out=300, in= 60] (1.25, -0.75);
    }

    \foreach \theta in {90,270} {
    \draw [black!50, thin, rotate=\theta] 
            (1.3, -0.835) to[out=120, in=240] (1.3,  0.835);
    }

    \draw[plum ]   ( 30:1.25) arc(30:-30:1.25);
    \fill[black]   ( 30:1.25)  circle (0.75pt);
    \fill[black]   (-30:1.25)  circle (0.75pt);
    \node[plum] at (270:2.5) {\((S,\sigma)\)};
    \node[plum] at (000:1.5 ) {\(\gamma_1\)};
    
    \begin{scope}[shift={(5.5,0)}]
        \draw[mauve!25] (0,0) circle (2.25cm);
        \shade[ball color = sky, opacity=.6] (0,0) circle (2.25cm);

        \draw[plum ] (210:1.25) arc(210:150:1.25);
        \fill[black] (210:1.25)   circle (0.75pt);
        \fill[black] (150:1.25)   circle (0.75pt);

        \node[plum] at (270:2.5) {\((\sph^2,\sigma_{\textnormal{std}})\)};
        \node[plum] at (180:1.5) {\(\gamma_2\)};
    \end{scope}
  \end{tikzpicture}
  \caption{An illustration of the bubbling surgery. A copy of the round sphere \((\mathbb{S}^2, \sigma_{\textnormal{std}})\) is attached to a surface equipped with a given branched projective structure to produce a new structure on the same surface. Topologically, this operation corresponds to taking the connected sum of the surface with a 2-sphere.}
  \label{fig:bubbling}
\end{figure}
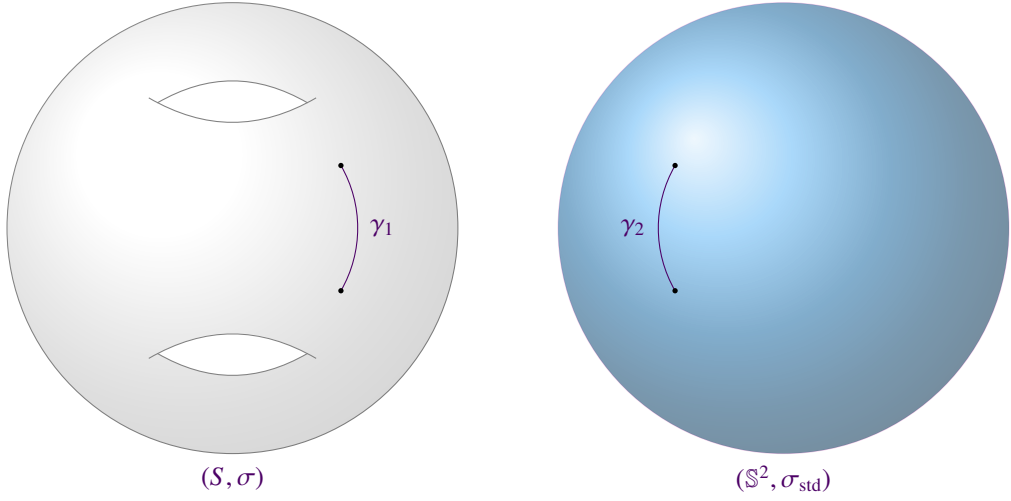

\begin{rmk}\label{rmk:flexibility}
    The resulting structure is highly sensitive to the choices made during the construction just performed. Indeed, several arbitrary choices are made throughout the process, such as the local chart and the geodesic segment in \(\rp2\), with different choices leading to different structures. In fact, two sets of choices yield the same structure if and only if they are related by a projective diffeomorphism \(f \in \textnormal{Proj}(\,\sigma_o\,)\).
\end{rmk}

\begin{rmk}\label{rmk:bubbledarebranched}
    Regardless of the nature of the initial structure, the resulting projective structure is necessarily branched, where the branch points arise from the identification of the extremal points of \(\gamma_1\) and \(\gamma_2\).
\end{rmk}

\begin{rmk}\label{rmk:iterativebubbles}
    The operation of bubbling a structure alters the total branching divisor, increasing its degree by two for each bubble glued. Furthermore, the aforementioned construction readily extends to the framework of branched charts, thereby allowing for an increase in the orders of pre-existing branch points on the bubbled surface.
\end{rmk}

\begin{defn}
    We say that \((S,\sigma)\) is obtained by \textit{bubbling} \((S,\sigma_o)\).
\end{defn}

\noindent The terminology of this surgery is quite intuitive, as the sphere takes the shape of a bubble. The key feature of this surgery is recalled in the following

\begin{prop}\label{prop:bubblepreservesholo}
     Let \((S,\sigma_o)\) be a \(\rp2\)-structure with holonomy \(\rho_o\) and let \((S,\sigma)\) be the branched \(\rp2\) structure obtained by bubbling \((S,\sigma_o)\). Then \((S,\sigma)\) has holonomy \(\rho_o\).
\end{prop}

\begin{proof}
    This result is a direct consequence of a broader algebraic fact. By the Seifert-Van Kampen theorem, the fundamental group \(\pi_1(\,S\,)\) of \(S\), seen as the connected sum of \(S\) itself and a \(\sphn2\), is given by the pushout of the groups \(\pi_1(\,S\,)\) and \(\pi_1(\,\sphn2\,)=\{\,1\,\}\) over the amalgamated subgroup generated by to the gluing curve. Let \(\rho_o\) be the holonomy representation of \((S,\sigma_o)\), and recall that the standard structure \((\sphn2,\sigma _{\textnormal{std}})\) has trivial holonomy. By the universal property of the pushout, there exists a unique representation \(\rho\) defined on the amalgamated product that extends both holonomy representations. Since the holonomy of \((S,\sigma _{\textnormal{std}})\) is trivial, it readily follows that \(\rho=\rho_o\), as desired.
\end{proof}

\begin{rmk}
    More generally, two (possibly branched) structures on topological surfaces, both of genus at least one, can be glued together along a slit by defining a branched projective structure on a topological surface homeomorphic to the connected sum of the former ones. A straightforward adaptation of the argument just proposed to prove Proposition \ref{prop:bubblepreservesholo} shows that the holonomy representation of the resulting structure is defined as the extension of the holonomy representations of the former structures.
\end{rmk}

\smallskip

\subsubsection{Hitchin representations as holonomies of branched structures}\label{ssec:hitchinbranched} In their work \cite{ChoiGoldman1993Convex}, Choi and Goldman showed that every Hitchin representation arises as the holonomy representation of a unique convex projective structure on \(S\). Relying on their result, we may use the bubbling construction recalled in \S\ref{sssec:bubbling} to realise branched projective structures with Hitchin holonomy. More specifically, as a direct consequence of Proposition \ref{prop:bubblepreservesholo}, we have the following Corollaries whose proofs are immediate.

\begin{cor}
    Let \((S,\sigma_o)\) be a convex projective structure with Hitchin holonomy \(\rho\) and let \((S,\sigma)\) be the branched projective structure obtained by bubbling \((S,\sigma_o)\). Then the branched structure \((S,\sigma)\) has Hitchin holonomy \(\rho\). In particular, every Hitchin representation arises as the holonomy of infinitely many branched projective structures on \(S\) obtained by bubbling a convex projective structure \((S, \sigma_o)\).
\end{cor}

\begin{cor}\label{cor: hitchin branched}
    Let \(\mu\) be any signature of even degree. Every Hitchin representation \(\rho\) arises as the holonomy of some branched projective structure in \(\mathcal{RP}^2(\,\mu\,)\).
\end{cor}

\begin{proof}[Sketch of the proof of Corollary \ref{cor: hitchin branched}.]
    Let \(\mu=(m_1,m_2\dots,m_k)\) be any signature of even degree. We preliminary observe that a branched projective structure in \(\mathcal{RP}^2(\,\mu\,)\) can be realised by bubbling an appropriate branched projective in \(\mathcal{RP}^2(\,\mu'\,)\), where \(\mu'=(m_1-1,m_2-1\dots,m_k)\). Based on this observation, we proceed as follows. Let \(\mu\) be any signature and assume that \(m_1\le m_2\le\cdots\le m_k\). We thus define a finite sequence of signature as follows
    \begin{equation}\label{eq:reduction bubble}
        \begin{aligned}
            \mu=&\mapsto (m_1-1,m_2-1,\dots,m_k) &\mapsto&\,\, (m_1-2,m_2-2,\dots,m_k) &\mapsto \cdots \\
            &\mapsto (m_2-m_1,m_3,\dots,m_k) &\mapsto& \,\, (m_2-m_1-1,m_3-1,\dots,m_k) & \mapsto\cdots \\
            & \mapsto \cdots &\mapsto &\,\, \cdots & \mapsto \cdots \\
            &\mapsto (\,m\,),
        \end{aligned}
    \end{equation}
    \textit{i.e.}, we reduce the degree of the signature by two at each step to obtain a signature of length one, say \((\,m\,)\), for some \(m\ge0\). Notice that \(m\) must be even because the starting signature has even degree. Set \(m=2n\). Assume that a representation \(\rho\) can be realised in the stratum \(\mathcal{RP}^2(\,2n\,)\). By bubbling sphere sufficiently many times by undoing the reduction of \eqref{eq:reduction bubble}, we thus realise \(\rho\) in the stratum \(\mathcal{RP}^2(\,\mu\,)\). It remains to show that such a representation can be realised in \(\mathcal{RP}^2(\,\mu\,)\). For this purpose, reduce \((\,2n\,)\) further to \((\,2n-2\,)\) and so on until we get the empty one. By \cite{ChoiGoldman1993Convex}, every Hitchin representation arises as the holonomy of a unique convex projective structure, say \(\sigma\) on \(S\). Pick a simple closed curve and consider its geodesic representative in its free homotopy class. Fix any point on such a geodesic and slit it. Bubble a sphere as shown in Figure \ref{fig:bubbling2}. By bubbling \(n\) spheres recursively, it is possible to realise every Hitchin representation in the stratum \(\mathcal{RP}^2(\,2n\,)\) for every \(n\ge0\) and hence in every stratum with signature of even degree.
\end{proof}

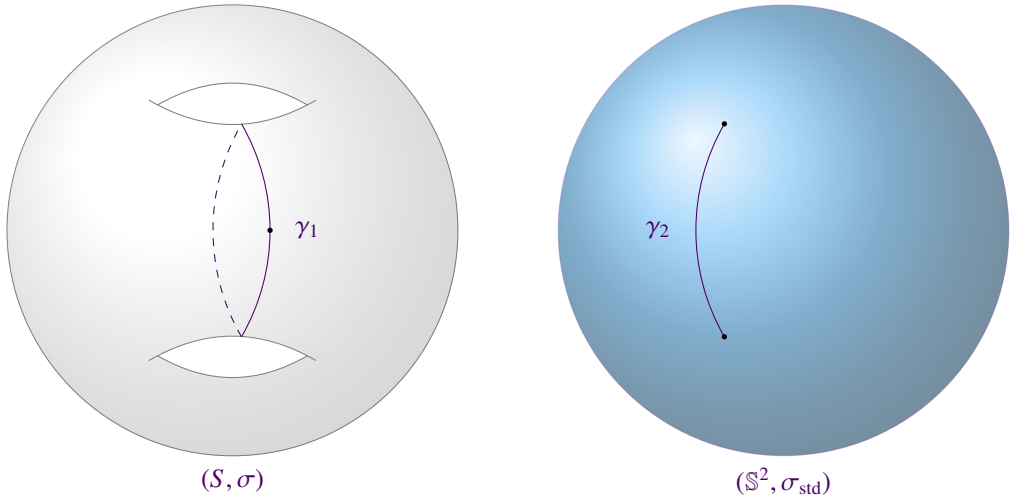
\begin{figure}[htbp]
  \centering
  \begin{tikzpicture}[scale=1.325]

    \draw[black!50] (0,0) circle (2.25cm);
    \shade[ball color=white, opacity=.25] (0,0) circle (2.25cm);

    \foreach \theta/\i in {90,270} {
    \fill[white, rotate=\theta] (1.25, 0.75) to[out=300, in=60] (1.25,-0.75) to[out=115, in=270] (1.05,0) to[out=90, in=245] (1.25, 0.75);
    }

    \foreach \theta in {90,270} {
    \draw [black!50, thin, rotate=\theta] 
            (1.25, 0.75) to[out=300, in= 60] (1.25, -0.75);
    }

    \foreach \theta in {90,270} {
    \draw [black!50, thin, rotate=\theta] 
            (1.3, -0.835) to[out=120, in=240] (1.3,  0.835);
    }

    \begin{scope}[shift={(-1.75,0)}]
        \draw[plum ]           ( 30:2.125) arc( 30:-30:2.125);
        \draw[plum, dashed ]   (-30:2.125) arc(210:150:2.125);
        \fill[black]   ( 0:2.125)  circle (0.75pt);
        \node[plum] at (000:2.5 ) {\(\gamma_1\)};
    \end{scope}
    
    \node[plum] at (270:2.5) {\((S,\sigma)\)};

    \begin{scope}[shift={(5.5,0)}]
        \draw[mauve!25] (0,0) circle (2.25cm);
        \shade[ball color = sky, opacity=.6] (0,0) circle (2.25cm);
        \begin{scope}[shift={(1.25,0)}]
            \draw[plum ] (210:2.125) arc(210:150:2.125);
            \fill[black] (210:2.125)   circle (0.75pt);
            \fill[black] (150:2.125)   circle (0.75pt);
        \end{scope}

        \node[plum] at (270:2.5) {\((\sph^2,\sigma_{\textnormal{std}})\)};
        \node[plum] at (180:1.25) {\(\gamma_2\)};
    \end{scope}
  \end{tikzpicture}
  \caption{An illustration of the bubbling surgery in the case \(\gamma_1\) is a geodesic curve. By slitting \(S\) along \(\gamma_1\) the resulting surface has two boundary components, say \(S'\). Next, we slit \((\mathbb{S}^2, \sigma_{\textnormal{std}})\) along a lift, say \(\gamma_2\), of the developed image of \(\gamma_1\). Identify the extremal points of \(\gamma_2\) after slitting. We thus have two boundary components. By gluing these latter boundary components with the boundary components of \(S'\) in the appropriate way, the resulting surface has a branched point of order \(2\). By gluing \(n\) copies of \((\mathbb{S}^2, \sigma_{\textnormal{std}})\) in the same way, the resulting surface has a branched structure with a single branched point of order \(2n\).}
  \label{fig:bubbling2}
\end{figure}

\noindent These statements provide a full description of the realisation of Hitchin representations and hence yield a the desired detailed characterisation.

\smallskip

\subsection{Properly convex branched projective structures}\label{ssec:prop convex BPS} Convex projective structures play a relevant role in the classical theory of projective structures due to their remarkable property of having holonomy representations in the Hitchin component, see \cite{ChoiGoldman1993Convex}. The aim of this section is to show that, in contrast, properly convex branched projective structures are quite peculiar. We begin with the following characterisation.

\begin{prop}\label{prop: hitchin factors}
    Let \(\rho\colon\pi_1(\,S\,)\longrightarrow \pslr3\) be a representation, then its image preserves a convex domain on which it acts freely, properly discontinuously, and co-compactly if and only if \(\rho\) factors through a Hitchin representation \(\rho_o\).
\end{prop}

\begin{proof}
Let \(\rho\) be a representation satisfying the standing hypotheses. If \(\rho\) is injective, then \(\rho\) is a Hitchin representation and there is nothing to prove since \(\rho=\rho_o\). Thus, we may assume that \(\rho\) is not injective. One direction is clear: if \(\rho\) factors through an Hitchin representation then its image preserves a convex domain on which it acts freely, properly discontinuously, and co-compactly. We only need to show the opposite implication. For this purpose, let \(\Omega\) be the convex domain preserved by the action of \(\Gamma = \rho(\,\pi_1(S)\,)\). 
By assumption, \(\Gamma\) acts freely, properly discontinuously and co-compactly on the convex domain \(\Omega \subset \rp2\), so the quotient \(\Omega / \Gamma\) is isomorphic to a smooth closed surface, say \(\Sigma\), equipped with an unbranched convex projective structure whose holonomy is a Hitchin representation \(\rho_o \colon \pi_1(\,\Sigma\,) \longrightarrow\Gamma \subset \pslr3\). Since \(\rho_o\) is an isomorphism, \(\rho\) factors through a homomorphism, say \(\pi\colon\pi_1(\,S\,)\longrightarrow\pi_1(\,\Sigma\,)\), such that \(\rho=\rho_o\circ\pi\) as desired. 
\end{proof}

Motivated by Proposition \ref{prop: hitchin factors}, in what follows we shall say that a representation \(\rho\) is \textit{purely Hitchin} if it factors through a Hitchin representation --- Hitchin representations are automatically included. This family of representations extends the family of purely hyperbolic representations introduced in \cite{faraco2021}. The following Corollary extends \cite[Theorem 3.14]{faraco2021} to the framework of branched real projective structures. To state it, we recall for the readers' convenience that a map \(f \colon S \to \Sigma\) between surfaces to be a \textit{pinch} if there is a compact, connected submanifold \(S_1 \subset S\), with boundary consisting of a single simple closed curve in the interior of \(S\), such that \(\Sigma = S / S_1\), the quotient of \(S\) with \(S_1\) identified to a point, and such that \(f\) is the quotient map. According to Edmonds, see \cite{Edmonds1979}, every continuous map of surfaces of non-zero degree splits as the composition of a pinch map and a branched cover. In particular, a continuous map \(f \colon S \longrightarrow \Sigma\) is homotopic to a branched covering if and only if \(f_* \colon \pi_1(\,S\,) \to \pi_1(\,\Sigma\,)\) is injective or the inequality \(\deg f > [\pi_1(\,\Sigma\,) \colon f_* \pi_1(\,S\,)]\) holds, where \([G \colon H]\) denotes the index of a subgroup \(H\) in the overall group \(G\).

\begin{cor}\label{cor: purely Hit holo char}
    Let \(\rho\colon\pi_1(\,S\,)\longrightarrow \Gamma\subset \pslr3\) be a purely Hitchin representation. Then \(\rho\) arises as the holonomy of some properly convex branched projective structure if and only if the induced map \(f\colon S\longrightarrow \Sigma\cong\Omega/\Gamma\) is not homotopic to a pinch map. 
\end{cor}

\begin{proof}
    By Proposition \ref{prop: hitchin factors}, a representation \(\rho\) satisfying the standing assumption factors through a Hitchin representation, say \(\rho_o\), via a homomorphism \(\pi\colon\pi_1(\,S\,)\longrightarrow\pi_1(\,\Sigma\,)\). Fix a base point, say \(b\in S\) and a base point, say \(p\in\Sigma\). Since surfaces are \(K(\pi,1)\)-spaces, by \cite[Proposition 1B.9]{hatcher2005algebraic}, there exists a unique continuous map, say \(f\colon (S,b)\longrightarrow (\Sigma,p)\) such that \(f_*=\pi\). Since the mapping \(\pi\) is surjective, it readily follows that \(\deg(\,f\,)\) is non-zero. According to \cite{Edmonds1979}, the function \(f\) just defined is either homotopic to a branched covering map or a pinch map. We shall distinguish these two cases. 
    \begin{itemize}
        \item[1.] Suppose \(f\) is homotopic to a branched covering map, and let \(\sigma_o\) be the convex projective structure on \(\Sigma\). Then \(\sigma=f^*\sigma_o\) is a branched projective structure on \(S\) with holonomy \(\rho\) by design. It remains to show that \(\sigma\) is convex. The developing map \(\textnormal{dev}_o\colon\widetilde\Sigma\longrightarrow\Omega\subset\rp2\) is a diffeomorphism onto the image which is convex. The desired conclusion follows by observing that \(\textnormal{dev}\colon\widetilde S\longrightarrow \rp2\) factors as the composition \(\textnormal{dev}=\textnormal{dev}_o\circ\widetilde{f}\), by design. 
        \smallskip
        \item[2.] Suppose \(f\) is homotopic to a pinch map and suppose there exists a branched projective structure \(\sigma\) with holonomy \(\rho\). Consider its developing map \(\textnormal{dev}_\sigma \colon \widetilde{S} \longrightarrow \Omega\). Since it is \(\pi_1(\,S\,)\)-equivariant with respect to \(\rho\), it descends to a branched map \(f \colon S \longrightarrow \Omega / \Gamma = \Sigma\). The induced map \(f_*\) on the fundamental groups is such that \(\rho = \rho_o\circ f_*\); thus, it is just that of the pinching map because it coincides with \(f_*\). Hence \(\deg(\,f\,) = 1\), implying that such a map is a branched map of degree one, that is, a homeomorphism, a contradiction.
    \end{itemize}
    Having established this latter case, the proof is complete. 
\end{proof}

The above Corollary characterises all branched projective structures with purely Hitchin holonomy representations by saying that they arise from branched covers over properly convex projective structures. The following observation shows that purely Hitchin representation cannot have odd Stiefel-Whitney number.

\begin{rmk} Let \(\rho\) be a purely Hitchin representation and let \(\rho_o \colon\pi_1(\,\Sigma\,) \longrightarrow \pslr3\) be a Hitchin representation such that \(\rho = \rho_o \circ f_*\), where \(f\colon S\longrightarrow \Sigma\) is the branched covering map determined as in Corollary~\ref{cor: purely Hit holo char}. Then \(\textnormal{w}_2(\,\rho\,) = d \cdot \textnormal{w}_2(\,\rho_o\,)\), where \(d\) is the degree of \(f\). In particular, \(\textnormal{w}_2(\,\rho\,)\) is zero.
\end{rmk}

As a consequence, we may observe that not all properly convex branched projective structures arise from branched covers, and hence not all properly convex branched projective structures have purely Hitchin holonomy. For instance, branched hyperbolic structures with holonomy having an odd Euler number, and thus non-zero second Stiefel--Whitney number, cannot have purely Hitchin holonomy for topological reasons. The question of determining necessary and sufficient conditions for a branched projective structure to be properly convex therefore remains open. Finally, in the case of a purely Hitchin representation, the following proposition provides conditions to realise a purely Hitchin representation in a given stratum.

\begin{cor}
   Let \(\mu=(\,m_1,\dots,m_k\,)\) be a signature of degree \(\deg(\,\mu\,)\). Then the stratum \(\mathcal{RP}^2(\,\mu\,)\) contains properly convex branched projective structures with purely Hitchin holonomy if and only if \(\deg(\,\mu\,)\) is even and there exists a branched covering \(f\colon S\longrightarrow\Sigma\) of degree \(\deg(\,f\,)\ge \max\big\{\,m_1,\dots,m_k\,\}+1\), where \(\Sigma\) is a surface of genus at least two. In particular, \(S\) has genus \(g\ge4\).
\end{cor}

\begin{proof}
    We just have shown that all properly convex branched projective structures with purely Hitchin holonomy arise as branched covers of genuine convex projective structures. Therefore, a structure is realisable in a stratum with signature \(\mu\) if and only if \(\deg(\,\mu\,)\) is even and there exists a branched covering with \(k\) branched points of orders \(m_1,\dots,m_k\). According to \cite{allan1984}, this is always possible as long as the target surface has genus at least two. Since the signature records the order of the branch points, the inequality \(\deg(\,f\,)\ge \max\big\{\,m_1,\dots,m_k\,\}+1\) readily follows. Finally, since there are not branched covers from surfaces of genus \(g\le3\) over a surface of genus two, it also follows that \(g\ge4\). 
\end{proof}

\medskip

\appendix 

\section{Non-orientable surfaces}\label{app:non orientable surfaces}

In the present appendix, we aim to extend Theorem \ref{thma:geometrisation} to non-orientable surfaces. In fact, the real projective plane being non-orientable, it is possible to endow non-orientable surfaces with a real projective structure. In the present Appendix, a topological surface \(S\) could be either orientable or non-orientable unless otherwise specified. We may observe that Definition \ref{def:branchedprojestructure} extends to non-orientable surfaces and, as in the orientable case, a geometric structure on a non-orientable surface is specified by a developing map equivariant with respect to a holonomy representation. The main result is the following:

\begin{thm}\label{thm_app: geom non orientable case}
    Let \(S\) be a closed surface. Every representation \(\rho\colon\pi_1(\,S\,)\longrightarrow
    \pslr3\) arises as the holonomy of some branched projective structure on \(S\).
\end{thm}

As we shall see, the argument presented in \S\ref{sec:geometrisation} easily extends to the case of closed non-orientable surfaces. We recall for the reader's convenience that, according to the classification of finite-type surfaces and Dyck's theorem, see \cite{Dyck1888}, a closed non-orientable surface of genus \(g\) is homeomorphic to the connected sum of a closed orientable surface 
and either one or two real projective planes, according to the parity of \(g\). As a consequence, a presentation of the fundamental group of a closed non-orientable surface of genus \(g\) is given by:
\begin{equation}
    \begin{aligned}
        \pi_1(\,N_g\,) &= \left\langle\, \alpha_1, \beta_1, \dots, \alpha_{k}, \beta_{k}, \gamma \;\middle|\; [\alpha_1, \beta_1] \cdots [\alpha_{k}, \beta_{k}] \gamma^2 = 1 \,\right\rangle\quad &\quad \textnormal{if \(g=2k+1,\)}\\
        \smallskip
        \pi_1(\,N_g\,) &= \left\langle\, \alpha_1, \beta_1, \dots, \alpha_{k}, \beta_{k}, \gamma_1,\gamma_2 \;\middle|\; [\alpha_1, \beta_1] \cdots [\alpha_{k}, \beta_{k}] \gamma_1^2\gamma_2^2 = 1 \,\right\rangle\quad &\quad \textnormal{if \(g=2k+2.\)}
    \end{aligned}
\end{equation}

Since the case of orientable surfaces has already been addressed, it remains to prove Theorem \ref{thm_app: geom non orientable case} for non-orientable surfaces. We distinguish three mutually disjoint cases as follows.

\begin{proofclaim}[Trivial representation]
    This is the easiest case to handle. For every non-orientable surface \(N_g\), there always exists a branched covering map \(f\colon N_g\longrightarrow \rp2\), see \cite{allan1984}. The pull-back of the real projective structure on \(\rp2\) yields a branched projective structure on \(N_g\) as desired.
\end{proofclaim}

We next focus on non-trivial representations, by distinguishing three cases based on whether the separating curves \(\gamma\) for \(g\) odd or \(\gamma_1\) and \(\gamma_2\) for \(g\) even have trivial holonomy. We begin by considering the two cases, which are specific for odd \(g=2k+1\).

\begin{proofclaim}[Case 1.]
    Assume \(\rho(\,\gamma\,)\) is trivial. Then \(\rho\) yields a representation \(\overline\rho\colon\pi_1(\,S\,)\longrightarrow \pslr3\), where \(S\) is an \textit{orientable} surface of genus \(k\). We realise this latter as the holonomy of some branched projective structure on \(S\) and finally glue a real projective plane copy of \(\rp2\) along a slit to realise the desired structure, see Figure \ref{fig:bubbling rp2}.
\end{proofclaim}

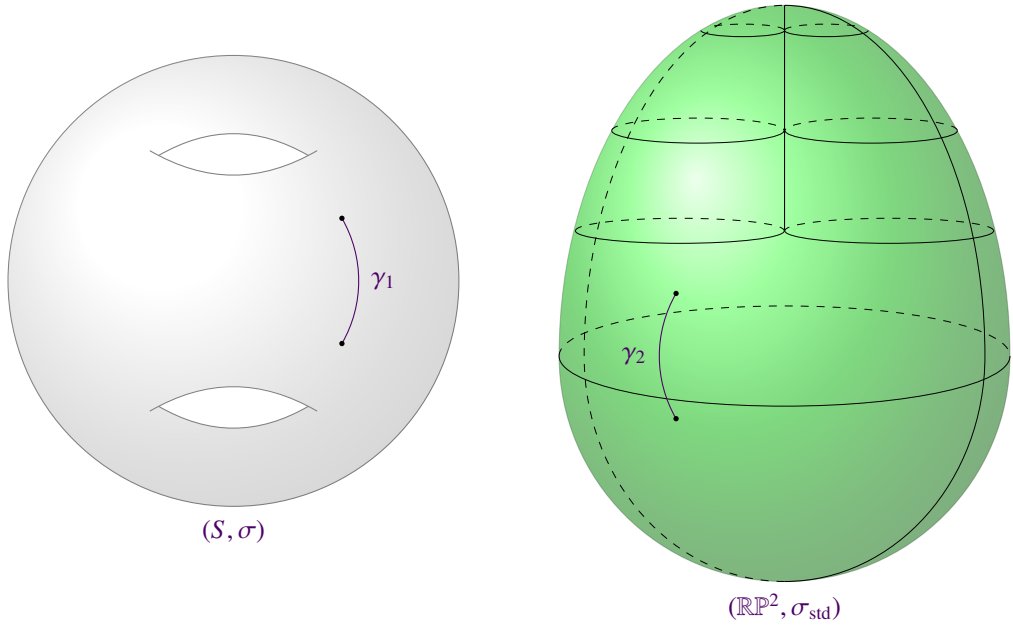
\begin{figure}[htbp]
  \centering
  \begin{tikzpicture}[scale=1.325]

    \draw[black!50] (0,0) circle (2.25cm);
    \shade[ball color=white, opacity=.25] (0,0) circle (2.25cm);

    \foreach \theta/\i in {90,270} {
    \fill[white, rotate=\theta] (1.25, 0.75) to[out=300, in=60] (1.25,-0.75) to[out=115, in=270] (1.05,0) to[out=90, in=245] (1.25, 0.75);
    }

    \foreach \theta in {90,270} {
    \draw [black!50, thin, rotate=\theta] 
            (1.25, 0.75) to[out=300, in= 60] (1.25, -0.75);
    }

    \foreach \theta in {90,270} {
    \draw [black!50, thin, rotate=\theta] 
            (1.3, -0.835) to[out=120, in=240] (1.3,  0.835);
    }

    \draw[plum ]   ( 30:1.25) arc(30:-30:1.25);
    \fill[black]   ( 30:1.25)  circle (0.75pt);
    \fill[black]   (-30:1.25)  circle (0.75pt);
    \node[plum] at (270:2.5) {\((S,\sigma)\)};
    \node[plum] at (000:1.5 ) {\(\gamma_1\)};
    
    \begin{scope}[shift={(5.5,-0.75)}]
        \draw[black!25, thin] ( 2.25,0) arc (360:180:2.25) (-2.25,0) arc [start angle=180, end angle=90, x radius=2.25cm, y radius=3.5cm] arc [start angle=90, end angle=00, x radius=2.25cm, y radius=3.5cm];
        \shade[ball color = green, opacity=0.5]  (2.25,0) arc (360:180:2.25) (-2.25,0) arc [start angle=180, end angle=90, x radius=2.25cm, y radius=3.5cm] arc [start angle=90, end angle=00, x radius=2.25cm, y radius=3.5cm];

        \draw[black, thin] (0,3.5)--(0,1.25);
        
        \draw[black, very thin] (0,-2.25) arc [start angle=270, end angle=360, x radius=2cm, y radius=2.25cm];%
        \draw[black, very thin] (2,0) arc [start angle=000, end angle=90, x radius=2cm, y radius=3.5cm]; 

        \draw[black, dashed, very thin] (0,-2.25) arc [start angle=270, end angle=180, x radius=2cm, y radius=2.25cm];%
        \draw[black, dashed, very thin] (-2,0) arc [start angle=180, end angle=90, x radius=2cm, y radius=3.5cm]; 

        \draw[black, very thin] (2.25,0) arc [start angle=360, end angle=180, x radius=2.25cm, y radius=0.5cm];
        \draw[black, dashed, very thin] (-2.25,0) arc [start angle=180, end angle=000, x radius=2.25cm, y radius=0.5cm];

        \draw[black, very thin] (2.0875,1.25) arc [start angle=360, end angle=180, x radius=1.0425cm, y radius=0.125cm];
        \draw[black, very thin] (-2.0875,1.25) arc [start angle=180, end angle=360, x radius=1.0425cm, y radius=0.125cm];
        \draw[black, dashed, very thin] (2.0875,1.25) arc [start angle=000, end angle=180, x radius=1.0425cm, y radius=0.125cm];
        \draw[black, dashed, very thin] (-2.0875,1.25) arc [start angle=180, end angle=000, x radius=1.0425cm, y radius=0.125cm];

        \draw[black, very thin] (1.725,2.25) arc [start angle=360, end angle=180, x radius=0.86cm, y radius=0.125cm];
        \draw[black, very thin] (-1.725,2.25) arc [start angle=180, end angle=360, x radius=0.86cm, y radius=0.125cm];
        \draw[black, dashed, very thin] (1.725,2.25) arc [start angle=000, end angle=180, x radius=0.86cm, y radius=0.125cm];
        \draw[black, dashed, very thin] (-1.725,2.25) arc [start angle=180, end angle=000, x radius=0.86cm, y radius=0.125cm];

        \draw[black, very thin] (0.835,3.25) arc [start angle=360, end angle=180, x radius=0.4175cm, y radius=0.0625cm];
        \draw[black, very thin] (-0.835,3.25) arc [start angle=180, end angle=360, x radius=0.4175cm, y radius=0.0625cm];
        \draw[black, dashed, very thin] (0,3.25) arc [start angle=180, end angle=030, x radius=0.4175cm, y radius=0.0625cm];
        \draw[black, dashed, very thin] (0,3.25) arc [start angle=000, end angle=150, x radius=0.4175cm, y radius=0.0625cm];

        \draw[plum ] (210:1.25) arc(210:150:1.25);
        \fill[black] (210:1.25)   circle (0.75pt);
        \fill[black] (150:1.25)   circle (0.75pt);

        \node[plum] at (270:2.5) {\((\rp2,\sigma_{\textnormal{std}})\)};
        \node[plum] at (180:1.5) {\(\gamma_2\)};
    \end{scope}
  \end{tikzpicture}
  \caption{Gluing a copy of \(\rp2\) along a slit}
  \label{fig:bubbling rp2}
\end{figure}

\begin{proofclaim}[Case 2.]
    We finally assume \(\rho(\,\gamma\,)\) is not trivial. The proof in this case is analogous to that of the orientable case. The substantial difference lies in the initial and the final steps, where one defines a closed polygonal chain as in \eqref{eq:chain} with the addition of two extra segments determined by \(\rho(\,\gamma\,)\) (note that, given the presentation of the fundamental group, without these the polygonal chain would not close). We may proceed as in \S\S\ref{sssec:assembly generic}-\ref{sssec:abelian handles}. The resulting surface will have a boundary component consisting of two geodesic segments that develops, via the developing map, into a chain of segments of the form: 
    \begin{equation}
        \rho(\,\gamma\,)^{-1}(\,p\,) \longmapsto p \longmapsto \rho(\,\gamma\,)(\,p\,)
    \end{equation}
    for some point \(p\in\rp2\), see Figure \ref{fig:gluing non orientable}. By identifying the latter in a suitable manner, one obtains the desired cross-cap, and the resulting surface will be a non-orientable surface homeomorphic to \(N_g\), equipped with a real projective structure having holonomy \(\rho\), as desired.
\end{proofclaim}

\begin{proofclaim}[Case 3.]
    This latter case applies only when the genus of \(N_g\) is even and equal to \(g=2k+2\). In this case, the non-orientable surface arises as the connected sum of an orientable surface of genus \(k\) and a Klein bottle or, equivalently, two copies of the real projective plane \(\rp2\). In this case, the separating curves \(\gamma_1\) and \(\gamma_2\) may both have trivial holonomy, in which case we proceed as in case one above, twice. In the case where only one curve has trivial holonomy and the other does not, we apply both cases above. We first realise the surface \(N_{g-1}\) as in Case 2, and then we add an additional cross cap as in Case 1. Finally, both \(\gamma_1\) and \(\gamma_2\) may have non-trivial holonomy. Even in this case, we may proceed as in \S\S\ref{sssec:assembly generic}-\ref{sssec:abelian handles}. The resulting surface will have a boundary component consisting of four geodesic segments that develops, via the developing map, into a chain of segments of the form: 
    \begin{equation}
        p \longmapsto \rho(\,\gamma_1\,)(\,p\,) \longmapsto \rho(\,\gamma_1\,)^2(\,p\,)=q\,\longmapsto  \rho(\,\gamma_2\,)(\,q\,)\longmapsto \rho(\,\gamma_2\,)^2(\,q\,)
    \end{equation}
    for some point \(p\in\rp2\), see Figure \ref{fig:gluing non orientable2}.
\end{proofclaim}

Since there are no remaining cases to consider, this completes the proof in the non-orientable case and hence the proof of Theorem \ref{thm_app: geom non orientable case}.

\begin{figure}[htbp]
  \centering
  \begin{tikzpicture}[scale=1.75]
    
    \begin{scope}[rotate=00]
    \draw[white, opacity=0.25] (0,0) circle (2.25cm);
    \shade[ball color = white, opacity=.20, shading angle=0] (000:2)--(072:2)--(144:2)--(216:2)--(288:2);
    \draw[gray] (144:2)--(216:2)--(288:2)--(360:2)--(072:2)--(144:2);

    \foreach \p in {(072:2), (144:2), (216:2), (288:2), (360:2)} {
    \fill[black] \p circle (1pt);
    }

    \node at (000:0) {core};
    \node at (-36:3.5) {\(\mathcal D_1\)};
    \node at (036:3.5) {\(\mathcal D_2\)};
    \node at (108:3.5) {\(\mathcal D_3\)};

    \node[scale=0.75] at (-72 :2.25) {\(\rho(\,\gamma\,)(\,p\,)\)};
    \node[scale=0.75] at (-144:2.25) {\(p\)};
    \node[scale=0.75] at (-216:2.5 ) {\(\rho(\,\gamma\,)^{-1}(\,p\,)\)};

    \node[red, scale=0.75, align=center] at (-135:3.5) {
    these edges are glued in such\\
    a way a cross-cap appears
    };

    \begin{scope}
        \draw[->, red, thick] (216:3) arc[start angle=216, end angle=342, radius=0.75];
        \draw[->, red, thick] (216:3) arc[start angle=216, end angle=90, radius=0.75];
    \end{scope}

    \begin{scope}[shift={(-36:0.75)}, rotate=-72]
        \shade[ball color = yellow, opacity=.20, shading angle=0] (072:2)--(000:2)--(000:4)--(030:5)--(060:4)--(072:2);
        \draw[yellow] (072:2)--(000:2)--(000:4)--(030:5)--(060:4)--(072:2);

        \foreach \p in {(072:2), (000:2), (000:4), (030:5), (060:4)} {
            \fill[black] \p circle (1pt);
        }
    \end{scope}

    \begin{scope}[shift={(36:0.75)}, rotate=00]
        \shade[ball color = red, opacity=.20, shading angle=0] (072:2)--(000:2)--(000:4)--(030:5)--(060:4)--(072:2);
        \draw[red] (072:2)--(000:2)--(000:4)--(030:5)--(060:4)--(072:2);

        \foreach \p in {(072:2), (000:2), (000:4), (030:5), (060:4)} {
            \fill[black] \p circle (1pt);
        }
    \end{scope}

    \begin{scope}[shift={(108:0.75)}, rotate=72]
        \shade[ball color = sky, opacity=.20, shading angle=0] (072:2)--(000:2)--(000:4)--(030:5)--(060:4)--(072:2);
        \draw[sky] (072:2)--(000:2)--(000:4)--(030:5)--(060:4)--(072:2);

        \foreach \p in {(072:2), (000:2), (000:4), (030:5), (060:4)} {
            \fill[black] \p circle (1pt);
        }
    \end{scope}
    \end{scope}

    \end{tikzpicture}
    \caption{The figure shows a pentagonal core to which three pentagons are glued, one per handle. The remaining edges are identified to realise a cross-cap.}
    \label{fig:gluing non orientable}
\end{figure}
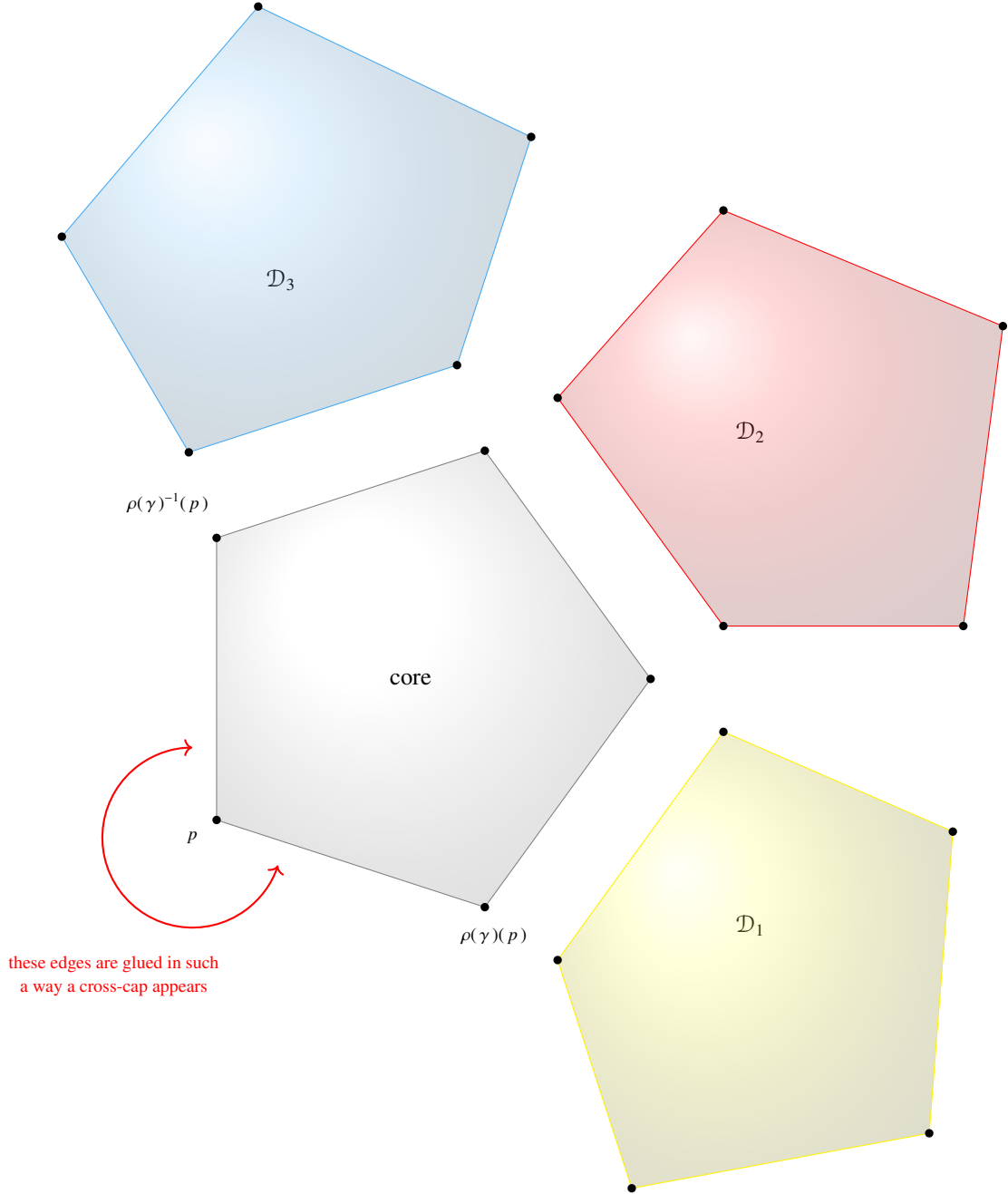

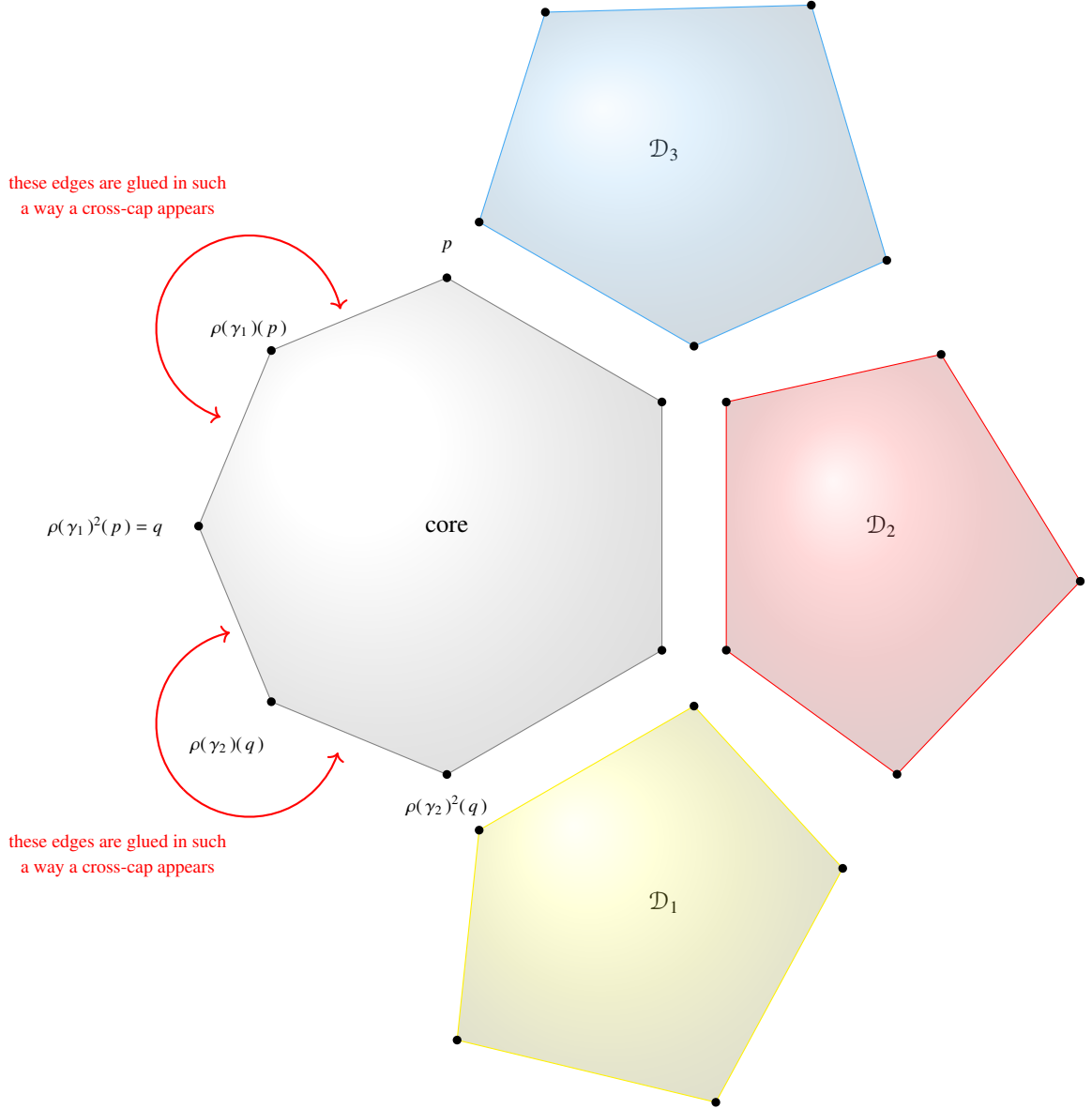
\begin{figure}[htbp]
  \centering
  \begin{tikzpicture}[scale=1.75]
    
    \begin{scope}[rotate=00]
    \draw[white, opacity=0.25] (0,0) circle (2.25cm);
    \shade[ball color = white, opacity=.20, shading angle=0] (270:2)--(330:2)--(030:2)--(090:2)--(135:2)--(180:2)--(225:2)--(270:2);
    \draw[gray] (270:2)--(330:2)--(030:2)--(090:2)--(135:2)--(180:2)--(225:2)--(270:2);

    \foreach \p in {((270:2), (330:2), (030:2), (090:2), (135:2), (180:2), (225:2)} {
    \fill[black] \p circle (1pt);
    }

    \node at (000:0) {core};
    \node at (-60:3.5) {\(\mathcal D_1\)};
    \node at (000:3.5) {\(\mathcal D_2\)};
    \node at (060:3.5) {\(\mathcal D_3\)};

    \node[scale=0.75] at (090:2.25) {\(p\)};
    \node[scale=0.75] at (135:2.25 ) {\(\rho(\,\gamma_1\,)(\,p\,)\)};
    \node[scale=0.75] at (180:2.75) {\(\rho(\,\gamma_1\,)^2(\,p\,)=q\)};
    \node[scale=0.75] at (225:2.5 )  {\(\rho(\,\gamma_2\,)(\,q\,)\)};
    \node[scale=0.75] at (270:2.25) {\(\rho(\,\gamma_2\,)^2(\,q\,)\)};

    \node[red, scale=0.75, align=center] at (-135:3.75) {
    these edges are glued in such\\
    a way a cross-cap appears
    };

    \node[red, scale=0.75, align=center] at (135:3.75) {
    these edges are glued in such\\
    a way a cross-cap appears
    };
    
    \begin{scope}
        \draw[->, red, thick] (225:3) arc[start angle=225, end angle=342, radius=0.75];
        \draw[->, red, thick] (225:3) arc[start angle=225, end angle=102, radius=0.75];
    \end{scope}

    \begin{scope}[rotate=-90]
        \draw[->, red, thick] (225:3) arc[start angle=225, end angle=342, radius=0.75];
        \draw[->, red, thick] (225:3) arc[start angle=225, end angle=102, radius=0.75];
    \end{scope}

    \begin{scope}[shift={(00:0.875)}, rotate=-36, scale=0.85]
        \shade[ball color = red, opacity=.20, shading angle=0] (072:2)--(000:2)--(000:4)--(030:5)--(060:4)--(072:2);
        \draw[red] (072:2)--(000:2)--(000:4)--(030:5)--(060:4)--(072:2);

        \foreach \p in {(072:2), (000:2), (000:4), (030:5), (060:4)} {
            \fill[black] \p circle (1.1765pt);
        }
    \end{scope}

    \begin{scope}[rotate=60]
        \begin{scope}[shift={(00:0.875)}, rotate=-36, scale=0.85]
        \shade[ball color = sky, opacity=.20, shading angle=0] (072:2)--(000:2)--(000:4)--(030:5)--(060:4)--(072:2);
        \draw[sky] (072:2)--(000:2)--(000:4)--(030:5)--(060:4)--(072:2);

        \foreach \p in {(072:2), (000:2), (000:4), (030:5), (060:4)} {
            \fill[black] \p circle (1.1765pt);
        }
        \end{scope}
    \end{scope}

    \begin{scope}[rotate=-60]
        \begin{scope}[shift={(00:0.875)}, rotate=-36, scale=0.85]
        \shade[ball color = yellow, opacity=.20, shading angle=0] (072:2)--(000:2)--(000:4)--(030:5)--(060:4)--(072:2);
        \draw[yellow] (072:2)--(000:2)--(000:4)--(030:5)--(060:4)--(072:2);

        \foreach \p in {(072:2), (000:2), (000:4), (030:5), (060:4)} {
            \fill[black] \p circle (1.1765pt);
        }
        \end{scope}
    \end{scope}

    \end{scope}

    \end{tikzpicture}
    \caption{The figure shows a pentagonal core to which three pentagons are glued, one per handle. The remaining edges are identified to realise two cross-caps.}
    \label{fig:gluing non orientable2}
\end{figure}

\smallskip

\section{Elementary representations}\label{asec:elementary} 

\noindent In the study of spaces of representations and geometrisation, a relevant role is played by the so-called elementary representations defined as those whose image is an elementary subgroup of \(\pslr3\). While the classification of elementary groups in \(\pslc{3}\) was recently provided in \cite{barreraetall2023elementary} by characterising their Kulkarni limit sets, an intrinsically dynamical approach based on the accumulation points of the group action on \(\pc2\), the real case admits a much more transparent geometric description. We introduce here an alternative definition for elementary groups which, at the best of our knowledge, has never been stated before and extends naturally to any \(n \ge 2\). 
We begin with the following

\begin{defn}\label{defn:elementarygroupfirstkind}
    A group \(\Gamma\subset\pslr3\) is said to be \textit{elementary} if, up to passing to a finite-index subgroup, it stabilises a, possibly degenerate, conic in \(\rp2\) or in \(\big(\,\rp2\,\big)^*\). Equivalently, if, up to passing to a finite-index subgroup, \(\Gamma\) fixes a point of \(\mathbb P\mathrm{Sym}^2\big((\mathbb R^3)^*\big)\) or of \(\mathbb P\mathrm{Sym}^2(\mathbb R^3)\).
\end{defn}

\noindent The notion of elementary subgroups of \(\pslr3\) can be viewed as a natural higher-dimensional analogue of the classification in \(\pslr2\), where elementary subgroups are characterised by globally preserving a conic in the real projective line \(\rp1\). We recall for the readers' convenience that in this lower-dimensional setting, a conic is geometrically realised as either the empty set (the elliptic case), a single point (the parabolic case), or a pair of distinct points (the hyperbolic case). In the same fashion, elementary subgroups of \(\pslr3\) arise as the stabiliser of a conic in \(\rp2\). More specifically, a subgroup \(\Gamma \subset \pslr3\) is elementary of the first kind if and only if it globally preserves one of the geometric configurations of Table \ref{tab:elementaryfirst}.

\begin{table}[ht]
\centering
\caption{Classification of elementary groups in \(\pslr3\).}
\begin{tabular}{cccc}
\hline
\\
\textbf{Rank} & \textbf{Fixed configuration in} \(\rp2\) & \textbf{Invariant conic} & \textbf{Matrix structure} \\ \\
\hline
\\
-- & Point \(\pto P\) & \makecell{Double line in \(\big(\,\rp2\,\big)^*\) \\ but not in \(\rp2\)}  & \(\begin{pmatrix} a & b & c \\ 0 & d & e \\ 0 & f & g \end{pmatrix}\) \\ \\
\hline
\\
1 & Flag (\(\pto P \in \ell\)) & Double line in \(\rp2\) & \(\begin{pmatrix} a & b & c \\ 0 & d & e \\ 0 & 0 & f \end{pmatrix}\) \\ \\
\hline
\\
1 & \makecell{Line \\ (but not a Flag)} & \makecell{Double line in \(\rp2\) \\ but not in \(\big(\,\rp2\,\big)^*\)} & \(\begin{pmatrix} a & b & c \\ d & e & f \\ 0 & 0 & g \end{pmatrix}\) \\ \\
\hline
\\
2 & Point + Line & Pair of incident lines & \(\begin{pmatrix} a & 0 & 0 \\ 0 & b & 0 \\ 0 & 0 & \frac{1}{ab} \end{pmatrix}\) \\ \\
\hline
\\
2 & \makecell{Point \(\pto P\) \\ (Rotation)} & Imaginary pair of lines & \(\begin{pmatrix} 1 & 0 & 0 \\ 0 & \cos\theta & -\sin\theta \\ 0 & \sin\theta & \cos\theta \end{pmatrix}\) \\ \\
\hline
\\
3 & None (Irreducible) & Real ellipse & \(^tM \begin{pmatrix} 1 & 0 & 0 \\ 0 & 1 & 0 \\ 0 & 0 & -1 \end{pmatrix} M = \begin{pmatrix} 1 & 0 & 0 \\ 0 & 1 & 0 \\ 0 & 0 & -1 \end{pmatrix}\) \\ \\
\hline
\\
3 & None (Irreducible) & Imaginary ellipse & \(^tM\begin{pmatrix} 1 & 0 & 0 \\ 0 & 1 & 0 \\ 0 & 0 & 1 \end{pmatrix} M = \begin{pmatrix} 1 & 0 & 0 \\ 0 & 1 & 0 \\ 0 & 0 & 1 \end{pmatrix}\) \\ \\
\hline
\end{tabular}
\label{tab:elementaryfirst}
\end{table}

\smallskip

\begin{ex}[Why elementary groups are defined up to finite-index subgroups]
    We would like to provide an explicit example of why elementary groups must be defined up to finite-index subgroups. Consider the matrix 
    \begin{equation}
        A=\begin{pmatrix}
            a &   &   \\
              & b &   \\
              &   & (ab)^{-1}
        \end{pmatrix}
    \end{equation}
    and let \(\mathcal P_3\cong\mathfrak S_3\) be the group of permutation matrices. Let \(\Gamma=\langle\,A,\,\mathcal P_3\,\rangle<\pslr3\) be the smallest subgroup containing \(A\) and \(\mathcal P\). We may notice that while the subgroup \(\langle\, A\, \rangle\) fixes the set of the three fundamental points \(\big\{\,e_1,\,e_2,\,e_3\,\big\}\), by preserving them individually, the overall group \(\Gamma\) permutes them. In particular, the action of \(\Gamma\) on \(\rp2\) is fixed-point free. Furthermore, there is no line \(\ell\) invariant under the action of \(\Gamma\). To show that no projective line \(\ell\) is fixed by \(\Gamma\), we observe that the group \(\mathcal{P}_3\) acts transitively on the coordinate lines \(\big\{\,\ell_1, \ell_2, \ell_3\,\big\}\), meaning no individual line is invariant. More generally, any line \(\ell\) would correspond to an invariant subspace \(V \subset \mathbb{R}^3\) of dimension two. While \(\langle\, A\, \rangle\) preserves only the coordinate planes \(\textnormal{span}\{e_i, e_j\}\), the permutation matrices in \(\mathcal{P}_3\) act by permuting these planes. Thus, the intersection of the invariant subspaces of \(\langle\, A\, \rangle\) and \(\mathcal{P}_3\) is trivial, implying that no line in \(\rp2\) can be invariant under the full group \(\Gamma\). However, \(\Gamma\) globally preserves the union of the three coordinate axes and permutes the complementary four regions. Thus, there is no conic, neither degenerate nor non-degenerate, that is globally invariant under the action of \(\Gamma\); however, a degenerate conic is invariant under the action of \(\langle A \rangle\).
\end{ex}

\noindent We aim to characterise elementary groups in more algebraic terms for subsequent applications. To this end, we regard \(\pslr3\) as an algebraic group with Lie algebra \(\mathfrak{sl}(3,\R)\). For the reader's convenience, for a Lie group \(\Gamma\) with finite-dimensional Lie algebra \(\mathfrak{g} = \textnormal{Lie}(\,\Gamma\,)\), we recall that the Levi decomposition theorem asserts that there is a semisimple Lie subalgebra \(\mathfrak{s} \subseteq \mathfrak{g}\), called a \textit{Levi subalgebra}, such that \(\mathfrak{g} = \textnormal{rad}(\,\mathfrak{g}\,) \rtimes \mathfrak{s}\), where \(\textnormal{rad}(\,\mathfrak{g}\,)\) denotes the radical of \(\mathfrak{g}\). This means that \(\mathfrak{g} = \textnormal{rad}(\,\mathfrak{g}\,) \oplus \mathfrak{s}\) as vector spaces; the sum is direct, \textit{i.e.}, \(\textnormal{rad}(\,\mathfrak{g}\,) \cap \mathfrak{s} = \{\,0\,\}\), and satisfies \([\mathfrak{s}, \textnormal{rad}(\,\mathfrak{g}\,)] \subseteq \textnormal{rad}(\,\mathfrak{g}\,)\). The aforementioned characterisation is established by the following 

\begin{thm}\label{prop:nonelementary-ZD}
    A group \(\Gamma < \pslr3\) is non-elementary if and only if it is Zariski-dense in \(\pslr3\).
\end{thm}

\begin{proof}
    According to our Definition \ref{defn:elementarygroupfirstkind}
    a group \(\Gamma\) is elementary if it preserves a conic \(\rp2\) or \(\big(\,\rp2\,\big)^*\). As shown in Table~\ref{tab:elementaryfirst}, any conic of rank strictly less than \(3\) is preserved by a subgroup of a maximal parabolic subgroup. Consequently, its closure is likewise contained in this maximal parabolic subgroup and cannot be Zariski-dense. Thus, let us assume \(\Gamma\) preserves a conic \(\mathcal C\) of rank \(3\). By definition, a matrix \(M\in\Gamma\) if and only if \(M\,(\,\mathcal C\,)=\mathcal C\). Let \(Q\) be the symmetric matrix associated to the conic \(\mathcal{C}\). Then, \(\mathcal{C}\) is preserved by a matrix \(M\) if and only if \({^t M} Q M = \lambda Q\) for some \(\lambda \in \mathbb{R}^*\). The latter identity provides algebraic equations that define \(\Gamma\) as a proper closed subgroups \(\pslr3\). Hence it is algebraic and not Zariski-dense in \(\pslr3\). Conversely, assume \(\Gamma\) is not elementary and let \(\overline\Gamma\) be its Zariski closure. Let \(\mathfrak g=\textnormal{Lie}(\,\Gamma\,)=\textnormal{Lie}(\,\overline\Gamma\,)\) be its Lie algebra, and consider it Levi decomposition \(\mathfrak{g} = \textnormal{rad}(\,\mathfrak{g}\,) \rtimes \mathfrak{s}\). If \(\textnormal{rad}(\,\mathfrak{g}\,)=\{\,0\,\}\), then \(\mathfrak g\) is semisimple, and the non-elementary assumption on \(\Gamma\) readily implies \(\mathfrak g=\mathfrak{sl}(3,\R)\); hence, \(\Gamma\) is Zariski-dense. If the Levi subalgebra vanishes, \textit{i.e.}, \(\mathfrak s=\{\,0\,\}\), then \(\mathfrak g\) is solvable and it is either contained in the Borel subalgebra, up to conjugation, or it contains a factor \(\mathfrak{so}(2,\R)\). In the former case, \(\Gamma\) is contained in the Borel subgroup of \(\pslr3\) and fixes a flag. In the latter one, it is contained in a maximal parabolic subgroup of \(\pslr3\). In both cases, we get a contradiction with our non-elementary assumption. Therefore, the semisimple part cannot be trivial. It remains to consider the mixed cases, that is, the cases where both \(\mathfrak s\) and \(\textnormal{rad}(\,\mathfrak{g}\,)\) are non trivial. In principle, the semisimple part \(\mathfrak{s}\) could be either \(\mathfrak{sl}(2,\R)\) or \(\mathfrak{so}(\,3\,)\). Since the latter case cannot occur if \(\textnormal{rad}(\,\mathfrak{g}\,)\) is non trivial; \(\mathfrak{s}\) must be \(\mathfrak{sl}(2,\R)\). This yields three possibilities, namely: 
    \begin{equation}
        \mathfrak{sl}(2,\R)\oplus\R \cong \mathfrak{gl}(2,\R), \qquad \mathfrak{sl}(2,\R)\ltimes\R^2, \qquad \mathfrak{gl}(2,\R)\ltimes\R^2\cong\mathfrak{sl}(2,\R)\oplus(\R\ltimes\R^2).
    \end{equation}
    The first two algebras are clearly subalgebras of \(\mathfrak{gl}(2,\R)\ltimes\R^2\), and hence it is sufficient to consider this latter one. Since \(\mathfrak{gl}(2,\R)\ltimes\R^2\) is the Lie algebra of the parabolic groups in Table \ref{tab:elementaryfirst}, we readily conclude that \(\Gamma\) must be elementary, thus leading to a contradiction. Therefore, \(\mathfrak g=\mathfrak{sl}(3,\R)\) and hence \(\Gamma\) is Zariski-dense. 
\end{proof}

Let \(V=\textnormal{Sym}^2\Big(\,\big(\,\R^3\,\big)^*\,\Big)\cong\R^6\) be the vector space of quadratic forms in \(\R^3\) and let \(\mathbb PV\cong\rp5\) be its projectivisation. Consider the representation \(\phi\colon \pslr3\longrightarrow \textnormal{PGL}(\,V\,)\) that maps each projective transformation of the real projective plane to the induced projectivity on the space of conics. The following establishes a more dynamical characterisation of elementary groups in \(\pslr3\). More specifically:

\begin{thm}\label{thm: nonelementary-SI}
    A group \(\Gamma\) is non-elementary if and only if the action of \(\phi(\,\Gamma\,)\) on \(\mathbb PV\) is strongly irreducible.
\end{thm}
    
For the readers' convenience, we recall that the action of \(\phi(\,\Gamma\,)\) on \(\mathbb{P}V\) is strongly irreducible if it preserves no non-empty \(\phi(\Gamma)\)-invariant, finite collection of proper projective subspaces. Equivalently, the above action is strongly irreducible if and only if every finite-index subgroup acts irreducibly.

\begin{proof}

Let us first suppose that \(\Gamma\) is elementary. By definition, there exists a finite-index subgroup \(\Gamma_o<\Gamma\) which fixes either a point \([q]\in\mathbb PV\) or a point \([B]\in\mathbb PV^*\). In the first case, \(\Gamma_o\) preserves the line \(\mathbb R\cdot q\subset V\), whereas in the second case it preserves the hyperplane \(B^\perp\subset V\) dual to the line \(\mathbb R\cdot B\subset V^*\). Therefore, the action of \(\Gamma_o\) is not irreducible in both cases. 
Conversely, assume the action of \(\phi(\,\Gamma\,)\) is not strongly irreducible on \(\mathbb PV\). Then there exists a finite index subgroup \(\Gamma_o<\Gamma\) and a proper subspace \(W\subset V\) such that \(\phi(\gamma)\cdot W=W\) for every \(\gamma\in\Gamma_o\), that is \(\phi(\Gamma_o)\) preserves a proper subspace of \(V\). Equivalently, \(\Gamma_o\subset \operatorname{Stab}(W)\). The group \(\operatorname{Stab}(W)\) is a proper algebraic subgroup of \(\operatorname{PSL}(3,\mathbb R)\) being the representation \(\phi\) irreducible. Therefore, it is not Zariski dense in \(\operatorname{PSL}(3,\mathbb R)\). Thus \(\Gamma_o\) is not Zariski dense in \(\operatorname{PSL}(3,\mathbb R)\), and hence \(\Gamma\) is elementary by Theorem \ref{prop:nonelementary-ZD}.
\end{proof}

\noindent We finally conclude with the following remark.

\begin{rmk}
    According to our definition, \(3\)-Fuchsian representations, see \S\ref{sssec:hitchincomponent}, are indeed elementary. We recall, however, that these do not represent the generic case of Hitchin representations; on the contrary, they constitute a particular case as they form a subspace with positive codimension. As a further motivation, we recall that the image of every Hitchin representation is Zariski-dense in either \(\pslr3\) or a conjugate of \(\pso\). In particular, the latter holds whenever the representation is \(3\)-Fuchsian.
\end{rmk}


\medskip

\bibliographystyle{amsalpha}
\bibliography{rbpsgeom}

\end{document}